\documentclass[reqno,11pt]{amsart}
\usepackage[utf8]{inputenc}
\usepackage{amsmath,amssymb,amsthm,mathrsfs,color,times,textcomp,yfonts,mathtools}
\usepackage[left=1.2in, right=1.2in, top=1in]{geometry}

\usepackage{subcaption}
\usepackage[normalem]{ulem}
\usepackage[export]{adjustbox}
\usepackage{esint}
\usepackage{xcolor}
\usepackage{array}
\usepackage[colorlinks=true]{hyperref}
\hypersetup{urlcolor=blue, citecolor=red, linkcolor=blue}

\usepackage[square,numbers]{natbib}

\usepackage{float}
\usepackage{ulem}
\usepackage{syntonly}
\usepackage{mathtools}
\usepackage{bm}
\usepackage{amsfonts,amsmath,latexsym,verbatim,amscd,mathrsfs,color,array}
\usepackage[colorlinks=true]{hyperref}

\usepackage{amsmath,amssymb,amsthm,amsfonts,graphicx,color}
\usepackage{amssymb}
\usepackage{epstopdf}

\def \Rn{\mathbb{R}^n}

\newcommand{\1}{\mathbf{1}}

\newcommand{\GG}{ {\mathcal G}}

\newcommand{\dd}{\,\mathrm{d}}

\newcommand{\TT}{{\mathcal T}  }

\newcommand{\pp}{ {\partial} }

\newcommand{\HH}{{{\mathcal H}}  }

\newcommand{\B}{ \mathcal{B} }

\newcommand{\RR}{{{\mathbb R}}}

\newcommand{\cuad}{{\sqcap\kern-.68em\sqcup}}

\newcommand{\be}{\begin{equation}}
\newcommand{\ee}{\end{equation}}

\theoremstyle{plain}
\newtheorem{theorem}{Theorem}[section]
\newtheorem{lemma}[theorem]{Lemma}
\newtheorem{prop}[theorem]{Proposition}
\newtheorem{corollary}[theorem]{Corollary}
\newtheorem{remark}{Remark}[theorem]
\newcommand{\bremark}{\begin{remark} \em}
	\newcommand{\eremark}{\end{remark} }

\numberwithin{equation}{section}

\title[Bubble towers in the ancient solution]{Bubble towers in the ancient solution of energy-critical  heat equation}
\author{Liming Sun}
\address{Academy of Mathematics and Systems Science, the Chinese Academy of Sciences, Beijing 100190, China.}
\email{lmsun@amss.ac.ca}
\author{Jun-cheng Wei}
\address{Department of Mathematics, University of British Columbia, Vancouver, BC, V6T 1Z2, CA.}
\email{jcwei@math.ubc.ca}

\author{Qidi Zhang}
\address{Department of Mathematics, University of British Columbia, Vancouver, BC, V6T 1Z2, CA.}
\email{qidi@math.ubc.ca}

\date{\today}

\begin{document}

	\maketitle

	\begin{abstract}
		We construct a radially smooth positive ancient solution for energy critical semi-linear heat equation in $\mathbb{R}^n$, $n\geq 7$. It blows up at the origin with the profile of multiple Talenti bubbles in the backward time infinity.
	\end{abstract}

\tableofcontents

\setlength{\parskip}{0.5em}
\section{Introduction}
\subsection{Motivation}
This paper deals with the analysis of ancient solutions that exhibit \textit{infinite time blow-up} in the energy critical semi-linear heat equation.
\begin{align}\label{intro:eq:main}
	u_{t} =\Delta u+|u|^{p-1} u \quad& \text { in } \mathbb{R}^{n} \times(-\infty, 0)
\end{align}
where $n\geq 3$ and $p$ is the critical Sobolev exponent $p_S:=\frac{n+2}{n-2}$. We are interested in the positive solutions $u(x,t)$ globally defined for ancient time such that
\begin{align}
    \lim_{t\to-\infty}\|u(\cdot,t)\|_{L^\infty(\Rn)}=+\infty.
\end{align}

 Problem \eqref{intro:eq:main} has a more popular counterpart in the forward direction, namely
\begin{align}\label{intro:forward}
    \begin{cases}
u_{t}=\Delta u+|u|^{p-1} u & \text { in } \mathbb{R}^{n} \times(0, T), \\
u(x, 0)=u_{0}(x) & \text { in } \mathbb{R}^{n}.
\end{cases}
\end{align}
The \textit{energy functional} associated to \eqref{intro:forward} is
\begin{align}
    J(u):=\int_{\Rn} \frac{1}{2}|\nabla u|^2-\frac{1}{p+1}|u|^{p+1} .
\end{align}
The scaling $u(x, t) \mapsto \lambda^{2 /(p-1)} u\left(\lambda x, \lambda^{2} t\right)$ keeps the equation invariant and transforms $J(u)$ to $ \lambda^{\frac{4}{p-1}+2-n}J(u)$. Evidently, \eqref{intro:forward} is energy critical when $p=p_S$.

Problem \eqref{intro:forward} has been extensively studied in the literature. It is well-known that for a large class of initial data, say bounded continuous, there is a unique maximal classical solution  $u(x,t)$ for $t\in (0,T)$. If $T$ is finite, then $u$ will blow up at $T$. There are two types of blow-up depending on the rate
\begin{align}
\text{Type I}&: \limsup _{t \rightarrow T}(T-t)^{\frac{1}{p-1}}\|u(\cdot,t)\|_{L^\infty(\Rn)}<\infty ,\\
\text{Type II}&: \limsup _{t \rightarrow T}(T-t)^{\frac{1}{p-1}}\|u(\cdot,t)\|_{L^\infty(\Rn)}=\infty.
\end{align}
The blow-up is almost completely understood in the sub-critical range $p<p_S$, for instance, by \cite{filippas1992refined, giga1985asymptotically,giga1987characterizing,giga2004blow,quittner1999priori,velazquez1992higher}. The solution always blows up in type I in this range.  The existence of type II blow-up has been established in various settings, for instance by  \cite{herrero1993blow,herrero1994explosion,mizoguchi2004type} when $p>p_{JL}$, where
\begin{align}
    p_{{JL}}=\begin{cases}
\infty & \text { if } n \leq 10, \\
1+\frac{4}{n-4-2 \sqrt{n-1}} & \text { if } n \geq 11 .
\end{cases}
\end{align}
Recently, there are active researches in the energy critical case $p=p_S$ by \cite{filippas2000fast,schweyer2012type,cortazar2019green,del2019type,del2020type,del2020infinite,harada2020higher,harada2020type}. These works found that $u$ can exhibit type II blow-up in finite time in lower dimensions, while \citet{wang2021refined} precluded this fast blow-up for $n\geq 7$.

Ancient solutions play an important role in studies of singularities and long-time behavior of solutions of many evolution problems, for instance in the mean curvature flow, Ricci flow and Yamabe flow. Comparing to the forward direction, the studies to ancient solutions of semi-linear heat equation \eqref{intro:eq:main} are quite limited.
In the sub-critical case, \citet{merle1998optimal} first established the following result.
\begin{theorem}[\cite{merle1998optimal}]
    Let $p<p_{S}$ and $u$ be a positive solution of \eqref{intro:eq:main} satisfying
\[
\|u(\cdot, t)\|_{L^\infty(\Rn)} \leq C(-t)^{-1 /(p-1)} \quad \text { as } \quad t \rightarrow-\infty.
\]
Then there exists $T^{*} \geq 0$ such that $u(x, t)=(p-1)^{-1 /(p-1)}\left(T^{*}-t\right)^{-1 /(p-1)}$.
\end{theorem}
The above result about ancient solutions has some interesting and important consequences in the study of the (forward) blow-up behavior of solutions of \eqref{intro:forward} when $p<p_S$. See \cite{merle1998optimal} for details.

For the super-critical case, one knows that there exists one-parameter radially positive steady states $\{\phi_\alpha\}$ for each $\alpha>0$. Furthermore, if $p>p_{JL}$, then these solutions are ordered as $\phi_\alpha<\phi_\beta$ for $\alpha<\beta$ and $\phi_\beta\to \phi_\infty$ as $\beta\to \infty$, where (see \cite{quittner2019superlinear,wang1993cauchy})
\[\phi_{\infty}(x):=L|x|^{-2 /(p-1)}, \quad L:=\left(\frac{2}{(p-1)^{2}}((n-2) p-n)\right)^{\frac{1}{p-1}}.\]
The following Liouville-type results are known by \citet[Theorem 2.4]{fila2011homoclinic} and \citet[Theorem 1.2]{polavcik2005liouville}.
\begin{theorem}[\cite{fila2011homoclinic,polavcik2005liouville}]
 Let $u$ be a non-negative radial solution of \eqref{intro:eq:main}.
 \begin{enumerate}
     \item Assume $p_{S} \leq p<p_{J L}$ and $u(\cdot, t) \leq \phi_{\infty}$ for all $t \leq 0$. Then $u \equiv 0$.
     \item Assume $p>p_{J L}$ and $\phi_{\alpha} \leq u(\cdot, t) \leq \phi_{\infty}$ for some $\alpha>0$ and all $t \leq 0$. Then $u(\cdot, t) \equiv \phi_{\gamma}$ for some $\gamma \geq \alpha$.
 \end{enumerate}
\end{theorem}
Without this $\phi_\infty$ bound, \cite{fila2011homoclinic} also constructed some radially positive bounded solutions which do depend on time. \citet{polavcik2021entire} classified all radially positive ancient solutions
under some further conditions for the super-critical regime.

We are interested in the energy-critical case $p=p_S$. The steady states of the equation \eqref{intro:forward} satisfy
\be
\Delta u+|u|^{\frac{4}{n-2}} u=0 \quad \text { in } \mathbb{R}^{n} .
\ee
We recall that all positive entire solutions of the equation are given by the family of Aubin-Talenti solitons \cite{aubin1976problemes,talenti1976best,gidas1979symmetry}
\begin{align}
	U_{\mu, \xi}(x)=\mu^{-\frac{n-2}{2}} U\left(\frac{x-\xi}{\mu}\right)
\end{align}
where $U(y)$ is the standard bubble soliton
\begin{align}\label{intro:bubble}
	U(y)=\alpha_{n}\left(\frac{1}{1+|y|^{2}}\right)^{\frac{n-2}{2}}, \quad \alpha_{n}=[n(n-2)]^{\frac{n-2}{4}} .
\end{align}
This family of solitons are also called Aubin-Talenti ground state solitary wave of the energy functional $J$. \citet{collot2017dynamics} classified the ancient solutions near the ground states.
\begin{theorem}[\cite{collot2017dynamics}]
Let $n\geq 7$ and $p=p_S$. There exist two strictly positive, $C^\infty$ radial solutions of \eqref{intro:eq:main}, $Q^+$ and $Q^-$ such that $\lim_{t\to-\infty}\|Q^{\pm}-U\|_{\dot H^1}=0$. Conversely, there exists $0<\delta\ll 1$ such that if $u$ is a solution of \eqref{intro:eq:main} with
\[\sup_{t\leq 0}\inf_{\mu>0,\xi\in\Rn}\|u(t)-U_{\mu,\xi}\|_{\dot H^1}\leq \delta,
\]
then $u=Q^{\pm}$ or $u=U$ up to the symmetry of the flow.
\end{theorem}
They also pointed out the forward behavior: $Q^+$ explodes according to type I blow-up in finite time, and $Q^-$ is global and dissipates $Q^-\to 0$ as $t\to \infty$ in $\dot H^1(\Rn)$.

A natural question is whether we have \textit{multiple} Aubin-Talenti solitons in the backward limit. In the forward direction, \citet{del2019existence} constructed an initial condition $u_0$ such that  \eqref{intro:forward} blows up in infinite time exactly at the origin. The solutions constructed in \cite{del2019existence} consist of sign-changing bubbling towers in the forward limit $t \to +\infty$. In this paper, we investigate the possibility of such phenomenon in the backward direction.

Recall that for any Palais-Smale sequence $\{u(x,t_n)\}_{n=1}^\infty\geq 0$ of the energy functional $J$, Struwe's profile decomposition  \cite{struwe1984global} tells us that passing to a subsequence, there are positive scalars $\{\mu_j(t_n)\}_{j=1}^k$ and points $\{\xi_j(t_n)\}_{j=1}^k$ such that
\begin{align}
    \frac{\mu_i(t_n)}{\mu_j(t_n)}+\frac{\mu_j(t_n)}{\mu_i(t_n)}+\frac{|\xi_i-\xi_j|^2}{\mu_i\mu_j}(t_n)\to \infty \quad \text{as } n\to \infty
\end{align}
and
\be
u(x, t_n)=\sum_{j=1}^{k} \frac{1}{\mu_{j}(t_n)^{\frac{n-2}{2}}} U\left(\frac{x-\xi_j(t_n)}{\mu_{j}(t_n)}\right)+o(1) \quad \text { as } \quad n \rightarrow \infty
\ee
where (after some permutation) $\mu_{k}(t) \leq \cdots \leq \mu_{1}(t)$. Our main result is the following existence  of bubbling-tower solution in the backward limit $ t\to -\infty$:

\begin{theorem}\label{intro:thm:main}
	Let $n \geq 7 $, $k \geq  2$. There exists a radially smooth positive solution of \eqref{intro:eq:main}
	that blows up backward in infinite time exactly at 0 with a profile of the form
	\be\label{W-1.7}
	u(x, t)=
	\left(1+O(|t|^{-\epsilon}) \right)
	\sum_{j=1}^{k} \mu_{j}(t)^{-\frac{n-2}{2}} U\left(\frac{x}{\mu_{j}(t)}\right)
	 \quad \text { for all } \quad t\le t_0
	\ee
	where $O(|t|^{-\epsilon})$ denotes some function $g(x,t)$ satisfying $\|g(\cdot,t)\|_{L^{\infty}(\RR^n)}\lesssim |t|^{-\epsilon}$. Furthermore, we have
	\begin{equation}
	\| u(x,t) - \sum_{j=1}^{k} \mu_{j}(t)^{-\frac{n-2}{2}} U\left(\frac{x}{\mu_{j}(t) }\right) \|_{H^1(\RR^n)} \lesssim |t|^{-\epsilon}
	\quad \text { for all } \quad t\le t_0 .
	\end{equation}
	Here $\epsilon>0$ is small,
	\[
	\mu_{j}(t)=\beta_{j} (-t)^{-\alpha_{j}}(1+O(|t|^{-\sigma})),\quad \alpha_{j}=\frac{1}{2}\left(\frac{n-2}{n-6}\right)^{j-1}-\frac{1}{2}, \quad j=1, \ldots, k,
	\]
where $\beta_j$, $\sigma$ are certain positive constants.
\end{theorem}

One interesting question is the forward behavior of the ancient solution we construct. Either $u(x,t)$ is an eternal solution, or it will blow up in type I in some later time.

There are some other related results.
\cite{del2018ancient} studied the ancient solution in Allen-Cahn equation.
\citet{daskalopoulos2018type}
 constructed the ancient bubbling-tower solution for Yamabe flow. For construction of radially symmetric bubbling-towers in NLS and energy-critical wave equations, we refer to \cite{Jendrej2017,Jendrej2018, Jendrej2019}.

\subsection{Sketch of the proof}

The method of this paper is close in spirit to the analysis in the works \cite{cortazar2019green,del2019existence}, where the inner-outer gluing method is employed. That approach consists of reducing the original problem to solving a basically uncoupled system, which depends in subtle ways on the parameter choices (which are governed by relatively simple ODE systems).

We start with the ansatz solution $\bar U=\sum_{j=1}^k U_j=\sum_{j=1}^k \mu_j(t)^{-\frac{n-2}{2}}U(x/\mu_j(t))$ and search for $\varphi(x,t)$ such that $\bar U+\varphi$ is a solution for
\begin{align}\label{intro:Su}
    S[u]:=-u_{t}+\Delta u+|u|^{p-1}u=0 \quad \text { in } \mathbb{R}^{n} \times\left(-\infty, t_0\right).
\end{align}
Because of the specific form of $\bar U$, we anticipate $\varphi\approx \sum_{j=1}^k\mu_j^{-\frac{n-2}{2}}\phi(x/\mu_j(t))\chi_j$ with some cut-off function $\chi_j$ supporting in the region where $U_j$ dominates other $U_l$, $l\ne j$.
Plugging in $\bar U+\varphi$, we found that in the support of $\chi_j$, the linearized operator of $\varphi$ is $\Delta+pU^{p-1}_j$ and  the leading error is $-\partial_t U_j+p U_j^{p-1}U_{j-1}(0)$. Making a change of variable $y=x/\mu_j(t)$, we will choose the $\phi$ satisfying
\be\label{intro:W-2.14}
\Delta_{y} \phi+p U(y)^{p-1} \phi+h_{j}(y, t)=0 \quad \text { in } \mathbb{R}^{n},
\quad
\phi(y) \rightarrow  0
\mbox{ \ as \ } |y| \rightarrow \infty
\ee
where
\be
h_{j}(y, t )=\mu_{j} \dot{\mu}_{j} Z_{n+1}(y)+p U(y)^{p-1}\left(\frac{\mu_{j}}{\mu_{j-1}}\right)^{\frac{n-2}{2}} U(0).
\ee
One knows that \eqref{intro:W-2.14} is solvable if and only if $\int_{\Rn} h_j(y,t)Z_{n+1}(y)dy=0$. Using the above expression of $h_j$, it implies that $\mu_j\dot \mu_j=c_*(\mu_j/\mu_{j-1})^{\frac{n-2}{2}}$ (see \eqref{W-2.15}). This implies $\mu_j(t)=\beta_j(-t)^{-\alpha_j}$.
We will denote it as $\mu_{0j}$, because the above process is the first approximation.

Next we will start with $u_*=\bar U+\sum_{j=1}^k\mu_{0j}(t)^{-\frac{n-2}{2}}\phi_{0j}(x/\mu_{0j}(t))\chi_j$ and search for $\varphi$ with the form
\[\varphi=\sum_{j=1}^k\mu_j^{-\frac{n-2}{2}}\phi_j(\frac{x}{\mu_j(t)},t)\eta_j+\Psi(x,t)\]
such that $u_*+\varphi$ is a solution of \eqref{intro:Su} and $\mu_j(t)=\mu_{0j}(t)+\mu_{1j}(t)$. Plugging in $u_*+\varphi$ to \eqref{intro:Su} can deduce the following equations of $\phi_j$ and $\Psi$
\begin{align}
	&
	\mu_{j}^{2} \partial_{t} \phi_{j}
	=
	\Delta_{y} \phi_{j}+
	p U(y)^{p-1} \phi_{j}
	+
	\HH_j[\Psi,\vec{\mu}_1](y,t)
	\mbox{ \ \ in \ }
	B_{8R} \times (-\infty,t_0),
	\quad
	j=1,\dots,k,
	\label{intro:W-3.6} \\
	&\pp_t \Psi
	=
	\Delta_{x} \Psi
	+
	\GG[\vec{\phi},\Psi,\vec{\mu}_1]
	(x,t)
	\mbox{ \ \ in \ }
	\RR^n \times (-\infty,t_0)
	,
	\label{intro:W-3.7}
\end{align}
where $\mathcal{H}_j$ is defined in \eqref{def:Hj} and $\GG$ is defined in \eqref{def:calG},  $\vec{\mu}_1(t)=(\mu_{11},\cdots,\mu_{1k})$ and $\vec{\phi}=(\phi_1,\cdots,\phi_k)$. \eqref{intro:W-3.6} is the so-called \textit{inner problem} and \eqref{intro:W-3.7} is the so-called \textit{outer problem}. One will see that these two problems are weakly coupled in the sense that the dependence of $\mathcal{H}_j$ on $\Psi$ and $\GG$ on $\vec{\phi}$ is small in appropriate norm. The strategy to solve \eqref{intro:W-3.6} and \eqref{intro:W-3.7} is: for each fixed $\vec{\phi}$ and $\vec{\mu}_1$, one can solve \eqref{intro:W-3.7} for $\Psi=\Psi[\vec{\phi},\vec{\mu}_1]$. Next,  inserting such $\Psi$ to \eqref{intro:W-3.6} and using fixed point theorem to find $\vec{\phi}$ and $\vec{\mu}_1$.

The foundation of this process lies on a clear understanding of the linearized problem of \eqref{intro:W-3.6} and \eqref{intro:W-3.7} respectively. The study to linearized equation of inner problem \eqref{intro:W-3.6} has been done in \cite{cortazar2019green} for the forward direction. For the backward direction, one encounters new difficulty when taking subsequences. We establish a uniqueness statement to make sure that different subsequences will give the same limit function. The linearized equation of the outer problem  \eqref{intro:W-3.7} occupies the bulk of this paper. Notice \eqref{intro:W-3.7} actually can be thought of nonhomogenous heat equation, we leverage the Duhamel's formula to get a solution $\Psi$. The main difficulty is to find a suitable topology for the outer problem due to bubble tower phenomenon. We spend a great deal of effort to find a good space to put $\GG$. Check Remark \ref{rmk:weight} and \ref{rmk:weight-2} for further explanation. Having set up the right space, we apply the Schauder fixed point theorem to prove the existence of
ancient solution of \eqref{intro:eq:main}.

{ Here is the structure of the paper.} In Section \ref{sec:ansatz}, we derive first approximation from the ansatz solution. Section \ref{sec:to_system} is devoted to splitting the flow equation to a system of inner problem and outer problem. In Section \ref{sec:linear}, we study the linear problem of the inner one and outer one respectively. We put off some tedious computations to Appendix A and B. Section \ref{sec:Ortho} is used to derive the orthogonal equations $\vec{\mu}_1$ should satisfy. In the last section, we put everything together and solve the problem by using Schauder's fixed point theorem.

\subsection{Notations}
Throughout this paper,
we denote $a\lesssim b$ if $a \le Cb$ for some positive constant $C$. Denote $a\approx b$ if $a\lesssim b \lesssim a$. $\chi(s)$ denotes a smooth cut-off function such that $0\le \chi(s) \le 1$,
\begin{equation*}
\chi(s)=\left\{\begin{array}{ll}
	1 & \text { if } s \leq 1 ,
	\\
	0 & \text { if } s \geq 2.
\end{array}\right.
\end{equation*}
For a set $\Omega \subset \mathbb{R}^{n}, \mathbf{1}_{\Omega}$ denotes the characteristic function defined as
\[
\mathbf{1}_{\Omega}(x)=\left\{\begin{array}{ll}
	1 & \text { if } x \in \Omega ,
	\\
	0 & \text { if } x \in \mathbb{R}^{n} \backslash \Omega.
\end{array}\right.
\]
For $j=1,\dots,k$,
$\mu_{j}$, $\mu_{0j}$  are some positive functions about $t$. We will use the notation
\begin{align}\label{def:barU}
    \vec{\mu}=\left(\mu_{1}, \ldots, \mu_{k}\right),\quad \bar U:=\ \sum_{j=1}^{k} U_{j}
\end{align}
where
\be\label{def:Uj}
U_{j}(x, t)={\mu_{j}(t)^{-\frac{n-2}{2}}} U\left(\frac{x}{\mu_{j}(t)}\right)
\ee
and $U(y)$ is given by \eqref{intro:bubble}. We denote
\begin{align}
    \bar\mu_j:=\sqrt{\mu_j\mu_{j-1}},\quad \bar\mu_{0j}:=\sqrt{\mu_{0j}\mu_{0,j-1}},\quad  j=2,\cdots,k
\end{align}
and make the convention that
\begin{align}
\bar\mu_1=\bar\mu_{01}=(-t)^{\delta},\quad \bar\mu_{k+1}=\bar \mu_{0,k+1}=0.
\end{align}
where $\delta>0$ is a small constant.
We write $\langle x\rangle=\sqrt{1+x^2}$.
	\section{A first approximation and the ansatz}\label{sec:ansatz}
Problem \eqref{intro:eq:main} is  equivalent to
\be\label{W-2.1}
S[u]:=-u_{t}+\Delta u+|u|^{p-1}u=0 \quad \text { in } \mathbb{R}^{n} \times\left(-\infty, t_0\right)
\ee
where $t_0$ is a very negative constant. After some translation in time, we can assume the solution lives up to $t=0$.

For any integer $k \geq 2$, let us consider $k$ positive functions
\[
\mu_{k}(t)<\mu_{k-1}(t)<\cdots<\mu_{1}(t) \quad \text { in }\left(-\infty,t_0\right)
\]
which will be chosen later, such that as $t \rightarrow-\infty$,
\be\label{W-2.2}
\mu_{1}(t) \rightarrow 1, \quad \frac{\mu_{j+1}(t)}{\mu_{j}(t)} \rightarrow 0 \quad \text { for all } \quad j=1, \ldots, k-1.
\ee
We assume that for $j=1,\dots,k$, $\mu_{0j}$ is the leading order of $\mu_{j}$ and has the similar property of $\mu_{j}$ above. $\mu_{0j}$ will be determined later. We will get an accurate first approximation to a solution of \eqref{W-2.1} of the form $\bar{U} +\varphi_0$  that reduces the part of the error $S[ \bar{U} ]$ created by the interaction of the bubbles $U_j$ and $U_{j-1}$, $j=2,\cdots,k$. To get the correction $\varphi_0$, we will need to fix the parameters $\mu_j$ at main order around certain explicit values.

Let us introduce the cut-off functions
\begin{align}
	\chi_{j}(x, t)=
	\begin{cases}
		\displaystyle
		\chi\left({2\left|x\right|}/{\bar{\mu}_{0j}}\right)
		-\chi\left(2{\left|x\right|}/{ \bar{\mu}_{0,j+1}}\right) & j=2, \ldots, k-1,
		\\
		\displaystyle\chi\left({2\left|x\right|}/{\bar{\mu}_{0k}}\right) & j=k.
	\end{cases}
\end{align}
One readily sees that
\begin{align}\label{spt-chij}
	\chi_{j}(x, t)=
	\begin{cases}
		0 &\text { if }  \left|x\right| \leq  \frac 12 \bar{\mu}_{0,j+1},
		\\
		1 &\text { if } \bar{\mu}_{0,j+1} \leq\left|x\right| \leq \frac{1}{2} \bar{\mu}_{0j}, \\
		0 & \text { if }  \left|x\right| \geq \bar{\mu}_{0j}.
	\end{cases}
\end{align}
We define our approximate solution to be given by
\be\label{W-2.25}
u_{*}={\bar{U}}+\varphi_{0} .
\ee
The correction $\varphi_{0}$ has the form
\be\label{W-2.6}
\varphi_{0}=\sum_{j=2}^{k} \varphi_{0 j} \chi_{j}
\ee
where
\begin{align}\label{W-2.7}
\varphi_{0 j}(x, t)={\mu_{j}(t)^{-\frac{n-2}{2}}} \phi_{0 j}\left(\frac{x}{\mu_{j}(t)}, t\right)
\end{align}
for certain functions $\phi_{0j}(y,t)$ defined in entire $y\in \mathbb{R}^n$ which we will suitably determine.
Let us write
\begin{align}
	S\left(u_{*}\right)
	=\bar E_1+\mathcal{L}_{{{\bar{U}}}}\left[\varphi_{0}\right]+N_{{{\bar{U}}}}\left[\varphi_{0}\right]
\end{align}
where
\begin{align}
	\mathcal{L}_{{\bar{U}}}\left[\varphi_{0}\right]&=-\partial_{t} \varphi_{0}+\Delta_{x} \varphi_{0}+p{\bar{U}}^{p-1} \varphi_{0}, \\
	N_{{\bar{U}}}\left[\varphi_{0}\right]&=\left|
	{\bar{U}}+\varphi_{0}
	\right|^{p-1}
	\left(
	{\bar{U}}+\varphi_{0}
	\right)
	-p{\bar{U}}^{p-1} \varphi_{0}-{\bar{U}}^p\label{N-phi0},\\
	\bar{E}_{1}&=-\sum_{j=1}^{k}\partial_{t} U_{j}+{\bar{U}}^p-\sum_{j=1}^{k} U_j^p.
\end{align}
Next we write $\mathcal{L}_{{\bar{U}}}\left[\varphi_{0}\right]$ using the form of $\varphi_{0}$ in \eqref{W-2.6} as follows
\[
\begin{aligned}
	\mathcal{L}_{{\bar{U}}}\left[\varphi_{0}\right] =&
	\sum_{j=2}^{k}\left(
	\Delta_{x} \varphi_{0 j}+p U_j^{p-1} \varphi_{0 j}\right) \chi_{j} +\sum_{j=2}^{k} p\left( \bar{U}^{p-1}-U_j^{p-1}\right) \varphi_{0 j} \chi_{j}\\
	&+\sum_{j=2}^{k}\left(
	2 \nabla_{x} \varphi_{0 j} \cdot \nabla_{x}\chi_{j}+\Delta_{x}\left(\chi_{j}\right) \varphi_{0 j}\right)
	-\sum_{j=2}^{k} \partial_{t}\left(\varphi_{0 j} \chi_{j}\right) .
\end{aligned}
\]

In the end, we have the error expansion
	\be\label{W-2.12}
	\begin{split}
		S\left( u_{* } \right)
		=&-\partial_t U_1+\sum_{j=2}^{k}\left(
		\Delta_{x} \varphi_{0 j}+pU_j^{p-1} \varphi_{0 j}-\partial_{t} U_{j}+pU_j^{p-1} U_{j-1}(0)\right) \chi_{j} \\
		&+\bar{E}_{11}+\sum_{j=2}^{k} p\left({\bar{U}}^{p-1}-U_{j}^{p-1}\right) \varphi_{0 j} \chi_{j} \\
		&+\sum_{j=2}^{k}\left(
		2 \nabla_{x} \varphi_{0 j} \cdot\nabla_{x}\chi_{j}+\Delta_{x} \chi_{j} \varphi_{0 j}\right)-\sum_{j=2}^{k} \partial_{t}\left(\varphi_{0 j} \chi_{j}\right)+N_{{\bar{U}}}\left[\varphi_{0}\right]
	\end{split}
	\ee
	where
	\be\label{W-2.10}
	\begin{split}
		\bar{E}_{11}
		={\bar{U}}^p-\sum_{j=1}^kU_j^p -\sum_{j=2}^k pU_j^{p-1}U_{j-1}(0)\chi_j- \sum_{j=2}^{k}(1-\chi_j) \partial_{t} U_{j}
		.
	\end{split}
	\ee
	

The function $\varphi_{0j}$ is chosen to eliminate at main order the terms in the first line of \eqref{W-2.12}, after conveniently restricting the range of variation of $\vec{\mu}$,
\be\label{W-2.13}
\begin{split}
	&E_{j}[\varphi_{0 j} , \vec{\mu}]:= \Delta_{x} \varphi_{0 j}+pU_{j}^{p-1} \varphi_{0 j}-\partial_{t} U_{j}+pU_{j}^{p-1} U_{j-1}(0) \\
	=& \mu_{j}^{-\frac{n+2}{2}}
	\left[\Delta_{y} \phi_{0 j}+p U(y)^{p-1} \phi_{0 j}+\mu_{j} \dot{\mu}_{j} Z_{n+1}(y)\right.+p U^{p-1}(y)\left(\frac{\mu_{j}}{\mu_{j-1}}\right)^{\frac{n-2}{2}} U(0)\big]_{y=\frac{x}{\mu_{j}}}
\end{split}
\ee
where $Z_{n+1}(y)=\frac{n-2}{2} U(y)+y \cdot \nabla U(y)$. The elliptic equation (for a radially symmetric function $\phi(y))$
\be\label{W-2.14}
\Delta_{y} \phi+p U(y)^{p-1} \phi+h_{j}(y, t)=0 \quad \text { in } \mathbb{R}^{n}
\ee
where
\be
h_{j}(y, t)=\mu_{j} \dot{\mu}_{j} Z_{n+1}(y)+p U(y)^{p-1}\left(\frac{\mu_{j}}{\mu_{j-1}}\right)^{\frac{n-2}{2}} U(0)
\ee
has a solution with $\phi(y) \rightarrow 0$ as $|y| \rightarrow \infty$ if and only if $h_{j}$ satisfies the solvability condition
\[
\int_{\mathbb{R}^{n}} h_{j}(y, t) Z_{n+1}(y) d y=0.
\]
The latter conditions hold if the parameters $\mu_{j}(t)$ satisfy the following relations:
\be\label{W-2.15}
\mu_{1}=1, \quad \mu_{j} \dot{\mu}_{j}=c_* \lambda_{j}^{\frac{n-2}{2}}, \quad \lambda_{j}=\frac{\mu_{j}}{\mu_{j-1}} \quad \text { for all } \quad j=2, \ldots, k
\ee
where
\begin{align}\label{W-2.16}
	c_*=-U(0) \frac{p \int_{\mathbb{R}^{n}} U^{p-1} Z_{n+1} d y}{\int_{\mathbb{R}^{n}} Z_{n+1}^{2} d y}=U(0) \frac{n-2}{2} \frac{\int_{\mathbb{R}^{n}} U^{p} d y}{\int_{\mathbb{R}^{n}} Z_{n+1}^{2} d y}>0.
\end{align}
Let $\vec{\mu}_{0}=\left(\mu_{01}, \ldots \mu_{0 k}\right)$ be the solution of \eqref{W-2.15} in $(-\infty,t_{0})$ given by
\be\label{W-2.17}
\mu_{0 j}(t)=\beta_{j} (-t)^{-\alpha_{j}}, \quad t \in\left( -\infty,t_0\right)
\ee
where
\[
\alpha_{j}=\frac{1}{2}\left(\frac{n-2}{n-6}\right)^{j-1}-\frac{1}{2}, \quad j=1, \ldots, k
\]
and the numbers $\beta_{j}$ are determined by the recursive relations
\begin{align}
\beta_{1}=1, \quad \beta_{j}=(\alpha_jc_*^{-1})^{\frac{2}{n-6}} \beta_{j-1}^{\frac{n-2}{n-6}} .
\end{align}
From $\eqref{W-2.15}$, we set
\begin{align}
\lambda_{0 j}(t)=\frac{\mu_{0 j}}{\mu_{0, j-1}}(t) .
\end{align}
We have
\[
h_{j}\left(y, t \right)=\lambda_{0 j}^{\frac{n-2}{2}} \bar{h}(y), \quad \bar{h}(y)
=
\bar{h}(|y|)
= p U(0) U(y)^{p-1}
+
c_* Z_{n+1}(y) .
\]
Since $\int_{\mathbb{R}^{n}} \bar{h} Z_{n+1} d y=0,$ there exists a radially symmetric solution $\bar{\phi}(y)$ to the equation
\[
\Delta \bar{\phi}+p U(y)^{p-1} \bar{\phi}+\bar{h}(|y|)=0 \quad \text { in } \quad \mathbb{R}^{n}
\]
such that $\bar\phi(y)=O(|y|^{-2})$ as $|y|\to +\infty$.

Then we define $\phi_{0 j}(y, t)$ as
\be\label{W-2.20}
\phi_{0 j}(y, t)=\lambda_{0 j}^{\frac{n-2}{2}} \bar{\phi}(y)
.
\ee

In what follows we let the parameters $\mu_{j}(t)$ in \eqref{W-2.2} have the form $\vec{\mu}=\vec{\mu}_{0}+\vec{\mu}_{1}$, namely
\be\label{W-2.21}
\mu_{j}(t)=\mu_{0 j}(t)+\mu_{1 j}(t) ,
\ee
where the parameters $\mu_{1 j}(t)$ to be determined satisfy
\be\label{W-2.22}
\left|\mu_{1 j}(t)\right| \lesssim \mu_{0 j}(t)
(-t)^{-\sigma}
\ee
for some small and fixed constant $0<\sigma <1$.  We ansatz $\frac{3}{4} \le \frac{|\mu_j|}{|\mu_{0j}|} \le \frac{4}{3}$ for $j=1,\dots,k$.

We observe that for some positive number $c_{j}$ we have
\[
\lambda_{0 j}(t)=c_{j} (-t)^{-\frac{2}{n-6}\left(\frac{n-2}{n-6}\right)^{j-2}} .
\]
With these choices, the expression $E_{j}[\varphi_{0 j} ; \vec{\mu}]$ in \eqref{W-2.13} can be decomposed as
\[
\begin{aligned}
	E_{j}[\varphi_{0 j} , \vec{\mu}_{0}+\vec{\mu}_{1}] &=\mu_{j}^{-\frac{n+2}{2}}\left[
	\left(\mu_{j} \dot{\mu}_{j}-\mu_{0 j} \dot{\mu}_{0 j}\right) Z_{n+1}\left(y_{j}\right)+\left(
	\lambda_{j}^{\frac{n-2}{2}}-\lambda_{0 j}^{\frac{n-2}{2}}\right) p U^{p-1}\left(y_{j}\right)U(0)
	\right] \\
	&=\mu_{j}^{-\frac{n+2}{2}} D_{j}\left[\vec{\mu}_{1}\right](y_j,t)
	+
	\mu_{j}^{-\frac{n+2}{2}}\Theta_{j}\left[\vec{\mu}_{1}\right](y_j,t),
	\quad y_{j}=\frac{x}{\mu_{j}(t)}
\end{aligned}
\]
where $j=2,\cdots,k$ and
\be
\begin{split}\label{W-2.23}
	D_{j}[\vec{\mu}_{1}](y_j,t)
	=&
	\left(\dot{\mu}_{0 j} \mu_{1 j}+\mu_{0 j} \dot{\mu}_{1 j}\right) Z_{n+1}\left(y_{j}\right)+\frac{n-2}{2} p U^{p-1}\left(y_{j}\right) U(0) \lambda_{0 j}^{\frac{n-2}{2}} \left(\frac{\mu_{1 j}}{\mu_{0 j}}-\frac{\mu_{1,j-1}}{\mu_{0,j-1}}\right) ,
	\\
	\Theta_{j}[\vec{\mu}_{1}](y_j,t)
	=&
	\mu_{1 j}\dot{\mu}_{1j} Z_{n+1}\left(y_{j}\right)+p U^{p-1}\left(y_{j}\right) \lambda_{0 j}^{\frac{n-2}{2}} O\left(\frac{|\mu_{1 j} |}{\mu_{0 j}}+\frac{|\mu_{1, j-1}|}{\mu_{0,j-1}}\right)^{2}.
\end{split}
\ee
Here we have used the fact that
\begin{align}
	\lambda_j^{\frac{n-2}{2}}-\lambda_{0j}^{\frac{n-2}{2}}=&\frac{n-2}{2}\lambda_{0j}^{\frac{n-2}{2}}\left(\frac{\mu_{1j}}{\mu_{0j}}-\frac{\mu_{1,j-1}}{\mu_{0,j-1}}\right)+\lambda_{0j}^{\frac{n-2}{2}}
	O\left(\frac{|\mu_{1 j} |}{\mu_{0 j}}+\frac{|\mu_{1, j-1}|}{\mu_{0,j-1}}\right)^{2}.
\end{align}
We also introduce the notation
\be\label{W-2.24}
D_{1}[\vec{\mu}_{1}](y_1,t)
=
( 1+\mu_{11})\dot{\mu}_{11} Z_{n+1}\left(y_{1}\right), \quad y_{1}=\frac{x}{\mu_{1}},
\ee
which is derived from
\begin{align}
    -\partial_tU_1=\mu_1^{-\frac{n+2}{2}}D_1[\vec{\mu}_1].
\end{align}


\section{The inner-outer gluing system}\label{sec:to_system}
We consider the approximation $u_{*}=u_{*}[\vec{\mu}_{1}]$ in \eqref{W-2.25} built in the previous section and want to find a solution of equation \eqref{W-2.1} in the form $u=u_{*}+\varphi .$ By Lemma \ref{lem:phi0}, we have $u_{*}>0$ when $t_{0}$ is very negative.  The problem becomes
\be\label{W-3.1}
S\left[u_{*}+\varphi\right]=
-\varphi_{t}+\Delta \varphi
+p u_{*}^{p-1} \varphi+N_{u_{*}}[\varphi]+S\left[u_{*}\right]=0 \quad \text { in } \mathbb{R}^{n} \times\left(-\infty,t_0\right)
\ee
where
\[
N_{u_{*}}[\varphi]=\left|
u_{*}+\varphi
\right|^{p-1}
\left(
u_{*}+\varphi
\right)
-
u_{*}^{p}
-p u_{*}^{p-1} \varphi.
\]

We consider the cut-off functions $\eta_{j}, \zeta_{j}, j=1, \ldots, k,$ defined as
\begin{align}\label{def:eta}
	\eta_{j}(x, t)&=\chi\left(\frac{\left|x\right|}{2R \mu_{0j}(t)}\right)
\end{align}
and
\begin{align}
	\zeta_{j}(x, t)&=\begin{dcases}
	\chi\left(\frac{\left|x\right|}{R \mu_{0 j}(t)}\right)-\chi\left(\frac{R\left|x\right|}{ \mu_{0 j}(t)}\right)\quad &j=1,\cdots, k-1,\\ \chi\left(\frac{|x|}{R\mu_{0k}(t)}\right)\quad &j=k.\end{dcases}\label{def:zeta}
\end{align}
We observe that $\eta_i\zeta_i=\zeta_i$, because
\[
\eta_{j}(x, t)=\left\{\begin{array}{ll}
	1 & \text { for }\left|x\right| \leq2 R \mu_{0j}(t) ,
	\\
	0 & \text { for }\left|x\right| \geq 4 R \mu_{0j}(t) .
\end{array}\right.
\]
and
\begin{align}\label{spt:zeta}
	\zeta_{j}(x, t)
	&
	=\left\{\begin{array}{ll}
		1 & \text { for } \quad 2 R^{-1} \mu_{0 j}(t) \leq\left|x\right| \leq R \mu_{0 j}(t), \\
		0 & \text { for }\quad |x| \geq 2 R \mu_{0 j}(t) \text { or }\left|x\right| \leq R^{-1} \mu_{0 j}(t) .
	\end{array}\right.
	j=1,\dots,k-1.
	\\
	\zeta_{k}(x, t)
	& =\left\{\begin{array}{ll}
		1 & \text { for } \quad \left|x\right| \leq R \mu_{0 k}(t), \\
		0 & \text { for }\quad |x| \geq 2 R \mu_{0 k}(t) .
	\end{array}\right.
\end{align}
{Here $R$ is a large constant to be determined later. In fact, we fix $R$ first, then take $t_0$ very negative.}

We consider functions $\phi_{j}(y, t)$, $j=1, \cdots, k$ defined in $B_{8R} \times (-\infty,t_0)$ and a function $\Psi(x, t)$ defined in $\mathbb{R}^{n} \times\left(-\infty,t_0\right)$. We look for the $\varphi(x, t)$ in \eqref{W-3.1} of the form
\be\label{W-3.4}
\varphi(x,t)=\sum_{j=1}^{k} \varphi_{j} \eta_{j}(x,t)+\Psi(x,t)
\ee
where
\begin{align}\label{W-3.4.5}
\varphi_{j}(x, t)={\mu_{j}^{-\frac{n-2}{2}}} \phi_{j}\left(\frac{x}{\mu_{j}(t)}, t\right).
\end{align}

Let us substitute $\varphi$ given by \eqref{W-3.4} into equation \eqref{W-3.1}. We get
\[
\begin{aligned}
	S\left[u_{*}+\varphi\right]=& \sum_{j=1}^{k} \eta_{j}
	\mu_{j}^{-\frac{n+2}{2}}
	\Big(
	-\mu_{j}^2 \partial_{t} \phi_{j}(y_j,t)
	+\Delta_{y} \phi_{j}(y_j,t)
	+p U(y_j)^{p-1} \phi_{j}(y_j,t)
	\\
	&
	+
	\mu_j^{\frac{n-2}{2}}
	\zeta_{j} p U(y_j)^{p-1} \Psi
	+
 D_{j}[\vec{\mu}_{1}]
 \Big)
 \\
	&-\Psi_{t}+\Delta_{x} \Psi+V \Psi+B[\vec{\phi}]+\mathcal{N}[\vec{\phi}, \Psi , \vec{\mu}_1 ]
	+E^{o u t} .
\end{aligned}
\]
Here we denote for $\vec{\phi}=\left(\phi_{1}, \ldots, \phi_{k}\right)$, $\vec{\mu}=\left(\mu_{1}, \ldots, \mu_{k}\right)$ and
\begin{align}
	&B[\vec{\phi}] =\sum_{j=1}^{k}
	2 \nabla_{x} \eta_{j}\cdot \nabla_{x} \varphi_{j}+\left(-\partial_{t} \eta_{j}+\Delta_{x} \eta_{j}\right) \varphi_{j}+ p\left(u_{*}^{p-1}-U_{j}^{p-1}\right) \varphi_{j}\eta_{j}
	-\dot{\mu}_{j} \frac{\partial}{\partial \mu_{j}} \varphi_{j} \eta_{j} ,
	\label{def:Bphi}\\
	&\mathcal{N}
	[\vec{\phi}, \Psi , \vec{\mu}_1
	]
	=N_{u_{*}}\left(\sum_{j=1}^{k} \varphi_{j} \eta_{j}+\Psi\right), \quad V=pu_{*}^{p-1}-\sum_{j=1}^{k} \zeta_{j} pU_{j}^{p-1},
	\label{Def:NV}
	\\
	&E^{o u t}=S\left[u_{*}\right]-\sum_{j=1}^{k} \mu_{j}^{-\frac{n+2}{2}} D_{j}[\vec{\mu}_{1}]  \eta_{j}
	,
	\label{def:Eout}
\end{align}
where $D_{j}[\vec{\mu}_{1}]$ are defined in \eqref{W-2.23} and \eqref{W-2.24}. We will have that $S\left[u_{*}+\varphi\right]=0$ if the following system of $k+1$ equations are satisfied.
\begin{align}
	&
	\mu_{j}^{2} \partial_{t} \phi_{j}
	=
	\Delta_{y} \phi_{j}+
	p U(y)^{p-1} \phi_{j}
	+
	\HH_j[\Psi,\vec{\mu}_1](y,t)
	\mbox{ \ \ in \ }
	B_{8R} \times (-\infty,t_0),
	\quad
	j=1,\dots,k,
	\label{W-3.6} \\
	&\pp_t \Psi
	=
	\Delta_{x} \Psi
	+
	\GG[\vec{\phi},\Psi,\vec{\mu}_1]
	(x,t)
	\mbox{ \ \ in \ }
	\RR^n \times (-\infty,t_0)
	,
	\label{W-3.7}
\end{align}
where
\begin{align}
\HH_j[\Psi,\vec{\mu}_1](y,t)= \ &
\mu_{j}^{\frac{n-2}{2}}
\zeta_{j}(\mu_j y)
p U(y)^{p-1}  \Psi(\mu_j y,t)
+
D_{j}[\vec{\mu}_{1}](y,t),
\label{def:Hj}
\\
\GG[\vec{\phi},\Psi,\vec{\mu}_1] (x,t) = \ &
V \Psi+B[\vec{\phi}]+\mathcal{N}[\vec{\phi}, \Psi , \vec{\mu}_1 ]+E^{o u t} .\label{def:calG}
\end{align}
In the next sections we will solve this system in a well-designed topology with suitable choice of parameters $\vec{\mu}_1$.


\section{The linear equations}\label{sec:linear}

In order to solve the system \eqref{W-3.6}-\eqref{W-3.7}, we need to study their linear equations respectively. The linear estimates of this section are crucial to the fixed point argument.

\subsection{The linear inner problem}
First, we consider the linear theory of \eqref{W-3.6}.
\begin{align}\label{W-4.17}
    \mu_j(t)^2\partial_t\phi=\Delta_y\phi+pU(y)^{p-1}\phi+h(y,t),\quad B_{8R}\times (-\infty,t_0).
\end{align}
where $\mu_j(t)\approx (-t)^{-\alpha_j}$ and $R$ is a sufficiently large constant. We aim to solve \eqref{W-4.17} by finding a linear mapping $\phi=\phi[h]$ that keeps the spatial decay property of $h$, provided that certain solvability condition for $h$ is satisfied. Making change of variables
\[
\tau(t)=
\tau_0 +
\int_{t_0}^t\mu_j(s)^{-2}ds\approx -(-t)^{2\alpha_j+1},
\]
where $\tau_0 $ is a suitably chosen that $\tau_0 \approx -(-t_0)^{2\alpha_j +1}$,
transforms \eqref{W-4.17} into
\begin{align}
    \partial_\tau \phi=\Delta_y\phi+pU(y)^{p-1}\phi+h(y,\tau),\quad B_{8R}\times (-\infty,\tau_0).
\end{align}
In order to solve this equation, we need to know the space $h(y,\tau)$ belongs to. This amounts to examining the decay of $\mathcal{H}_j$ in \eqref{def:Hj}.
Inspired by the estimate of $|\mathcal{H}_j(y,t)|$ in
Lemma \ref{lem:inner-h},
we define the following norms
\begin{align}\label{def:h^in_nu}
	\|h\|_{\nu(\tau),2+a}^{in}:=\sup _{s<\tau_{0}} \sup_{y \in B_{8R}} \nu^{-1}(s) \langle y \rangle^{2+a}|h(y, s)|,
\end{align}
\begin{equation}\label{def:phi^*_nu}
\|\phi\|_{\nu(\tau),a}^{in,*}
: =
\sup _{s<\tau_{0}} \sup_{y \in B_{8R}}
R^{-(n+1-a)}
\nu^{-1}(s) \langle y \rangle^{n+1 }|\phi(y, s)|.
\end{equation}
where $0<a< 1$ and $\nu(\tau):(-\infty,\tau_0)\to \mathbb{R}_+$ is a positive $C^1$ function satisfying
\begin{align}\label{nu-condi}
    \lim_{\tau\to -\infty}\nu(\tau)= 0, \quad\text{and}\quad \partial_\tau\nu\approx \frac{\nu(\tau)}{-\tau}\quad \text{ for }\tau\le \tau_0.
\end{align}
\begin{lemma}\label{inner lin}
Consider
\begin{equation}\label{W-4.24}
	\begin{aligned}
	\partial_{\tau} \phi =
		\Delta \phi + pU^{p-1} \phi + h(y,\tau)
		\quad \text{ in }
		B_{8R} \times (-\infty,\tau_0) .
	\end{aligned}
\end{equation}
For all sufficiently large $R>0$,  if $\tau_0 = \tau_0(R)$ is very negative, $\|h\|_{\nu(\tau), 2+a}^{in}<+\infty$, and $h(y,\tau)$ satisfies
	\begin{align}\label{h-ortho-1}
	\int_{B_{8R}} h(y, \tau) Z_{j}(y) d y=0 \quad \text { for all } \quad \tau \in\left(-\infty,\tau_{0}\right)
	\end{align}
$j=1, \ldots, n+1$, where $Z_j(y)=\partial_{y_j}U(y)$ and $Z_{n+1}(y)=\frac{n-2}{2} U(y)+y \cdot \nabla U(y)$.
	Then there exists a linear mapping
	\begin{align}\label{def:Tin}
	    \phi=\mathcal{T}^{in}_{\nu(\tau)}[h]
	\end{align}
	which solves \eqref{W-4.24} and satisfies the estimate
	\begin{align}\label{ynablayphi}
	\|\langle y\rangle\nabla_y \phi\|_{\nu(\tau),a}^{in,*}
	+\|\phi\|_{\nu(\tau),a}^{in,*}\leq C^{in}_{\nu(\tau),a}\|h\|_{\nu(\tau),2+a}^{in},
	\end{align}
{where $C^{in}_{\nu(\tau),a}$ is a constant depending on $\nu(\tau)$ and $a$.}
\end{lemma}

\begin{remark}
Since we consider radial scheme throughout this paper, $\int_{B_{8R}} h(y, \tau) Z_{j}(y) d y=0$,  $j=1,\dots,n$, are satisfied automatically.
\end{remark}
The proof inherits the spirit of \cite{cortazar2019green}. First, we consider the linear problem \eqref{W-4.24} in a finite time region $(s,\tau_0)$ and get a uniform estimate independent of the initial time $s$. Second, we make $s$ go to $-\infty$ and get an ancient solution by the compaction argument, like \cite{daskalopoulos2018type}. We need to use some Liouville type theorem to guarantee the uniqueness of the ancient solution derived from this operation, which deduces the existence of the desired linear mapping.

	Since the proof is very similar to the linear theory in \cite{cortazar2019green}, we only stress the difference due to taking subsequence as $s\rightarrow -\infty$. We have to prove no matter what convergent subsequence we choose, the limit function is the same.
	
	First we need the following preparation lemma.
	\begin{lemma}\label{nonnegative energy lemma}
		Given $L(u) = \Delta u + c(x) u$ defined in a bounded domain $\Omega$, if $L$ has a positive supersolution
		$w \in C^2(\Omega) \cap C(\bar{\Omega})$, that is $L(w) \le 0$ in $\Omega$ and $w > 0$ in $\bar{\Omega}$, then for all $\phi \in C^2(\Omega) \cap C(\bar{\Omega}) $ with $\phi = 0$ on $\pp \Omega$, the corresponding energy
		\[
		Q(\phi,\phi) := \int_{\Omega} \left( |\nabla \phi|^2 - c(x)\phi^2 \right) \dd y \ge 0 .
		\]
	\end{lemma}
	\begin{proof}
		Since $w > 0$ in $\bar{\Omega}$,
		$\exists\, \psi \in C^2(\Omega) \cap C^0(\bar{\Omega}) $ such that $\phi = w \psi$. Then
		\begin{align}
		\begin{split}
		    Q(\phi,\phi)&= \int_{\Omega} \left( -
		    w \psi^2 \Delta w  -2w \psi \nabla w \cdot \nabla \psi
		    -
		    w^2\psi
		    \Delta \psi  - c(x)w^2 \psi^2
		    \right) \dd y  \\
		    &= \int_{\Omega} \left[			- \left( \Delta w + c(x)w \right) w \psi^2 -2w \psi \nabla w \cdot \nabla \psi
		    -
		    w^2\psi
		    \Delta \psi   \right] \dd y .
			\end{split}
		\end{align}
		Using the assumption $Lw\leq 0$ and $w>0$, we have
		\begin{align*}
		    Q(\phi,\phi)&\ge \int_{\Omega} \left(  -2w \psi \nabla w \cdot \nabla \psi -
		    w^2\psi
		    \Delta \psi
		    \right) \dd y\\
		    &= \int_{\Omega} \left[  -2w \psi \nabla w \cdot \nabla \psi + \nabla\psi
		    \cdot \nabla(w^2\psi) \right] \dd y   = \int_{\Omega} |\nabla\psi|^2 w^2  \dd y \geq 0.
		\end{align*}
	\end{proof}

We take the following typical lemma, whose counterpart is given in Lemma $7.3$ in \cite{cortazar2019green}, to illustrate the difference with the linear theory in \cite{cortazar2019green} due to taking subsequence. Define $\chi_M(y)= \chi(|y|-M)$.
\begin{lemma}\label{phi*-corollary-ancient}
Consider
	\begin{equation}\label{phi-chiM-ancient}
		\begin{cases}
			\phi_{\tau} = \Delta \phi + p U^{p-1} (1-\chi_M) \phi + h(y,\tau) \mbox{\ \ in \ }
		B_{8R}\times (-\infty,\tau_0),
			\\
			\phi = 0 \mbox{ \ \ on \ } \pp B_{8R}\times (-\infty,\tau_0),
		\end{cases}
	\end{equation}
where $\|h \|_{\nu,a} <+\infty$, $0\le a<n$.
	If $M$ is a large constant, there exists a very negative constant $\tilde{\tau}_0$. If $\tau_0\le \tilde{\tau}_0$,
	there exists a linear map $\phi_*[h]$ satisfying \eqref{phi-chiM-ancient}
	and the following estimate:
	\begin{equation}\label{phi*-ancient}
		|\phi_*[h]| \lesssim_M
		\nu(\tau) \Theta_{R a}^0(|y|)
		\|h \|_{\nu,a}
		,
	\end{equation}
	where
	\[
	\Theta_{Ra}^0(r) =
	\begin{cases}
		(1+r)^{2-a} & \mbox{\ \ if \ } 2 <a < n \\
		\ln R & \mbox{\ \ if \ } a=2 \\
		R^{2-a} & \mbox{\ \ if \ } 0 \le a<2 .
	\end{cases}
	\]
\end{lemma}
\begin{proof}
	First we consider
	\begin{equation*}
		\begin{cases}
			\phi_{\tau}^{s} = \Delta \phi^{s}  + p U^{p-1} (1-\chi_M) \phi^{s}  + h(y,\tau) \mbox{\ \ in \ } B_{8R}\times (-\infty,\tau_0)
			,
			\\
			\phi^{s}  = 0 \mbox{ \ \ on \ } \pp B_{8R}\times (-\infty,\tau_0),
			\\
			\phi^{s}(\cdot,s)  = 0 \mbox{ \ \ in \ }  B_{8R} .
		\end{cases}
	\end{equation*}
	By the same method in Lemma 7.3 in \cite{cortazar2019green}, we have
	\begin{equation}
		|\phi_*^{s}[h]| \lesssim_M
		\nu(\tau) \Theta_{R a}^0(|y|)
		\|h \|_{\nu,a}
		.
	\end{equation}
	Notice this estimate is independent of $s$.
	By parabolic estimate,
	Arzel\`a-Ascoli theorem and
	diagonalization argument, taking $s\rightarrow -\infty$, we find a weak solution
	$\phi_*[h]$ to \eqref{phi-chiM-ancient} with the following estimate
	\begin{equation}\label{upper bound}
		|\phi_*[h]| \lesssim_M
		\nu(\tau) \Theta_{R a}^0( |y| )
		\|h \|_{\nu,a}  .
	\end{equation}
	
	Next, we need to demonstrate this operation is really a mapping. That is, if the operation gives two functions $\phi_*^1[h]$, $\phi_*^2[h]$ due to the different choices of  subsequences, we need to prove $\phi_*^1[h] = \phi_*^2[h]$. In fact, set $\Phi_* = \phi_*^1[h] - \phi_*^2[h]$. By \eqref{phi-chiM-ancient}, \eqref{upper bound}, $\Phi_*$ satisfies
	\begin{equation}
		\begin{cases}
			\pp_{\tau} \Phi_* = \Delta \Phi_* + p U^{p-1} (1-\chi_M) \Phi_*  \mbox{\ \ in \ }
			B_{8R}\times (-\infty,\tau_0),
			\\
			\Phi_* = 0 \mbox{ \ \ on \ } \pp
			B_{8R}\times (-\infty,\tau_0),
			\\
			|\Phi_*| \lesssim_M
			\nu(\tau) \Theta_{R a}^0( |y| )
			\|h \|_{\nu,a} .
		\end{cases}
	\end{equation}
By parabolic regularity theory, $\Phi_*$ is smooth.
	Multiplying $\Phi_*$ for both sides and integrating by part, we have
	\begin{equation}
		\frac 12 \pp_\tau \int_{B_{8R}} |\Phi_* |^2 \dd x=\int_{B_{8R } }
		\left[
		-|\nabla \Phi_* |^2 + p U^{p-1} (1-\chi_M) \Phi_*^2
		\right]
		\dd x
		\le 0.
	\end{equation}
The inequality is due to Lemma \ref{nonnegative energy lemma} since $L_M(\phi) = \Delta \phi + p U^{p-1}(y) (1 - \chi_M)\phi $ has a positive kernel $g_2(|y|)$ given in Lemma 7.3 of  \cite{cortazar2019green}.
	
By the upper bound of $|\Phi_* |$,
	\begin{equation*}
		\int_{B_{8R }} |\Phi_* |^2 \dd x
		\lesssim
		\|h \|_{\nu,a}^2
		\nu^2(\tau) R^{n}
		\begin{cases}
			1 & \mbox{\ \ if \ } 2 <a < n ,
			\\
			 \ln^2(R )
			& \mbox{\ \ if \ } a=2 ,
			\\
		 R^{4-2a}
			 & \mbox{\ \ if \ } 0 \le a<2 .
		\end{cases}
	\end{equation*}
Thus
	\begin{equation}
		\int_{B_{8R }} |\Phi_* |^2 \dd x \rightarrow 0 \mbox{ \ \  as \  } \tau \rightarrow -\infty ,
	\end{equation}
we have $\int_{B_{8R }} |\Phi_* |^2 \dd x \equiv 0$, which implies $\Phi_* \equiv 0$.
	
	By the same argument, we can prove that  $\phi_*[h]$ is a linear mapping. That is, for all functions $f$, $g$ satisfying $\|f\|_{\nu,a}$, $\|g\|_{\nu,a}<\infty$, we have $\phi_*[f+g] = \phi_*[f] + \phi_*[g]$.
\end{proof}

Since we aim to find ancient solutions. The initial value given in the linear theory of \cite{cortazar2019green} will disappear as $s\rightarrow -\infty$.

\subsection{The linear outer problem}
We consider the solution of
\begin{align}\label{W-4.1}
	\psi_t=\Delta_x\psi+g(x,t),\quad \text{in }\RR^n \times(-\infty,t_0).
\end{align}
It is well-known that the above equation has a solution which is given by Duhamel's formula
\begin{align}\label{duhamel}
	\psi(x,t)=\mathcal{T}^{out}[g](x,t):=\frac{1}{(4 \pi)^{\frac{n}{2}}} \int_{-\infty}^{t} \frac{d s}{(t-s)^{\frac{n}{2}}} \int_{\mathbb{R}^{n}} e^{-\frac{|x-y|^{2}}{4(t-s)}} g(y, s) d y
\end{align}
whenever the integral is well-defined.

In order to design a topology to solve the outer problem \eqref{W-3.7}, we define three types of weights.
\begin{equation}\label{def:w1j}
\begin{aligned}
    w_{11}(x,t)
    &
    =
    \frac{|t|^{-1-\sigma}}{1+|x|^{2+\alpha}}
  \1_{\{ |x|\le \bar \mu_{01} \}}  +
  |t|^{-1-\sigma}
  \bar{\mu}_{01}^{n-2-\alpha} |x|^{-1-n} \1_{\{ \bar{\mu}_{01} \le |x| \le |t|^{\frac 12} \}}
  \\
  &  \approx
    |t|^{\gamma_1}
    \1_{\{|x|\le 1\}}
    +
    |t|^{\gamma_1}|x|^{-2-\alpha}
    \1_{\{1\le|x|\le \bar{\mu}_{01} \}}
    +
    |t|^{\gamma_{1}}
    \bar{\mu}_{01}^{n-2-\alpha} |x|^{-1-n} \1_{\{ \bar{\mu}_{01} \le |x| \le |t|^{\frac 12} \}}
    ,
    \\
	w_{1j}(x,t) &=
	\frac{|t|^{-\sigma}}{\mu_{0j}^{\frac{n+2}{2}}} \frac{\lambda_{0j}^{\frac{n-2}{2}}}{1+(\frac{|x |}{\mu_{0j}})^{2+\alpha}}
	\1_{\{|x|\le \bar{\mu}_{0j}\}}
	\\
&
\approx
\mu_{0j}^{-2}|t|^{\gamma_j}
\1_{\{ |x|\le \mu_{0j} \}}
+
\mu_{0j}^{\alpha}|t|^{\gamma_j}
|x|^{-2-\alpha}
\1_{ \{\mu_{0j} \le |x|\le \bar{\mu}_{0j} \} } ,
\end{aligned}
\end{equation}
where $\gamma_1=-1-\sigma$ and $\gamma_j=\frac{n-2}{2}\alpha_{j-1}-\sigma$ for $j = 2,\dots,k$.

\begin{equation}
	\begin{aligned}
		w_{21}(x,t) = \ &
		|t|^{-2\sigma}
		\mu_{02}^{\frac{n}{2} -2 }
		\mu_{01}^{-1}
		|x|^{2-n}
		\1_{\{\bar{\mu}_{02} \le |x| \le 1\}} ,
		\\
	w_{2j}(x,t) = \ &
		|t|^{-2\sigma}
		\mu_{0,j+1}^{\frac{n}{2} -2 }
		\mu_{0j}^{-1}
		|x|^{2-n}
		\1_{\{\bar{\mu}_{0,j+1} \le |x| \le \bar{\mu}_{0j}
		\}} ,
		\mbox{ \ \ for \ } j = 2,\dots, k-1.
	\end{aligned}
\end{equation}
and
\begin{equation}
	w_{3}(x,t) = R |t|^{-1-\sigma} |x|^{2-n} \1_{\{|x|\ge \bar{\mu}_{01} \} },
\end{equation}
where $\delta>0$ is a small constant.

\begin{remark}\label{rmk:weight}
 These ad hoc weights are used to control the behavior of $\GG$ in \eqref{def:calG}. There are four terms in $\GG$, namely $B[\vec{\phi}]$ (the influence of inner problem), $V\Psi$ (linear term on $\Psi$), $E^{out}$ (error comes from ansatz $\bar U$ and mainly depends on $\vec{\mu}_1$), $\mathcal{N}$ (higher order nonlinear term). Roughly speaking, $w_{1j}$ will be used to control $B[\vec{\phi}]$ in the support of $\chi_j$. Specially, $w_{11}$ is also designed to control the influence of $w_{11}^{*}$ in $\{|x|\ge \bar{\mu}_{01}\}$. The regions between the support of $\chi_j$ of $B[\vec{\phi}]$ is controlled by $w_{2j}$. Also notice the support of $B[\vec{\phi}]$ is contained in $\{|x|\leq 4 R\mu_{01}\}$.
 $w_{3}$ is designed for controlling $E^{out}$ in $\{|x|\ge \bar{\mu}_{01}\}$.
 See Remark \ref{rmk:weight-2} how to control the other three terms.
\end{remark}

\begin{lemma}\label{lem:w1j-w1j*}
	 For $ j = 1, \dots,k$, we have the following estimate:
\begin{align}\label{w1j*}
    \mathcal{T}^{out}\left[w_{1j}\right]\lesssim w_{1j}^*:
    =\begin{cases}
    |t|^{\gamma_j} & \text{if }|x|\leq \mu_{0j},
    \\
    |t|^{\gamma_j}\mu_{0j}^{\alpha}|x|^{-\alpha}
    &\text{if } \mu_{0j}\leq |x|\leq\bar\mu_{0j} ,
    \\
    |t|^{\gamma_j}\mu_{0j}^\alpha \bar\mu_{0j}^{n-2-\alpha}|x|^{2-n}
    & \text{if }\bar\mu_{0j}\leq |x|\leq |t|^{\frac12},
    \\
    |x|^{2\gamma_j^*+2-n}& \text{if }|x|\geq |t|^{\frac12}.
    \end{cases}
\end{align}
where $\gamma_j^*=(1-\frac{n}{2})\alpha_j-\frac{\alpha}{2}(\alpha_j-\alpha_{j-1})-\sigma$ for $j = 2,\dots, k$ and $\gamma_1^*=\gamma_1+(n-2-\alpha)\delta$.
\end{lemma}
 Here $\gamma_j^*$ satisfies  $
|t|^{\gamma_j}\mu_{0j}^\alpha \bar\mu_{0j}^{n-2-\alpha}
\approx
|t|^{\gamma_j^*}
$ for simplicity. Approximately, $w_{1j}^*$ is like a radially non-increasing function about $|x|$ for every fixed $t$ up to a constant multiplicity, that is
\begin{align}\label{a1}
    w_{1j}^*
    \approx
    \min\{|t|^{\gamma_j}, |t|^{\gamma_j -\alpha \alpha_{j}}
    |x|^{-\alpha},|t|^{\gamma_j^*}|x|^{2-n},|x|^{2\gamma_j^*+2-n}\} .
\end{align}

Similarly, we have the following fact.
\begin{lemma}\label{lem:w2-w2*}
We have the following estimates:
	\begin{align}\label{w21*}
	    \mathcal{T}^{out}[w_{21}]\lesssim w_{21}^*
	    :=\begin{cases}
	    |t|^{-2\sigma}
	    &\text{if }
	    |x|
	    \leq \bar\mu_{02},
	    \\
	    |t|^{-2\sigma}\bar \mu_{02}^{n-4}|x|^{4-n}&\text{if }\bar \mu_{02}\leq |x|\leq 1,\\
	   |t|^{-2\sigma}\bar \mu_{02}^{n-4}|x|^{2-n}&\text{if }1\leq |x|\leq |t|^{\frac12},
	    \\
	    (|x|^2)^{-2\sigma-(\frac{n}{2}-2)\alpha_2}|x|^{2-n}&\text{if }|x|\geq |t|^{\frac12}.
	    \end{cases}
	\end{align}
and for $ j = 2, \dots, k-1$,
	\begin{align}\label{def:w2j*}
		\mathcal{T}^{out}[w_{2j}]\lesssim
		w_{2j}^*:=\begin{dcases}
			|t|^{-2\sigma} \mu_{0j}^{1-\frac n2}
			&
			\mbox{ \ \ if \ } |x| \le \bar{\mu}_{0,j+1},
			\\
			|t|^{-2\sigma}	\mu_{0,j+1}^{\frac{n}{2} -2} \mu_{0j}^{-1}
			|x|^{4-n}
			&
			\mbox{ \ \ if \ }
			\bar{\mu}_{0,j+1} \le
			|x| \le \bar \mu_{0j},
			\\
			|t|^{-2\sigma}	\mu_{0,j+1}^{\frac{n}{2} -2} \mu_{0,j-1}
			|x|^{2-n}
			&
			\mbox{ \ \ if \ }  \bar\mu_{0j} \le |x| \le |t|^{\frac 12},
			\\
			(|x|^2)^{-2\sigma -(\frac n2 -2)\alpha_{j+1} -\alpha_{j-1}}
			|x|^{2-n}
			&
			\mbox{ \ \ if \ }   |x| \ge |t|^{\frac 12} .
		\end{dcases}
	\end{align}
\end{lemma}

\begin{lemma}\label{lem:w3-w3*}
For $\delta\le \frac 12$, we have the following estimate:
\begin{equation}\label{eq:w3j*}
\TT^{out}[w_3] \lesssim
w_{3}^{*} := R
\begin{cases}
	|t|^{-1-\sigma}  \bar{\mu}_{01}^{4-n}
	&
	\mbox{ \ \ if \ }  |x| \le \bar{\mu}_{01} ,
	\\
	|t|^{-1-\sigma} |x|^{4-n}
	&
	\mbox{ \ \ if \ } \bar{\mu}_{01}  \le |x| \le |t|^{\frac 12},
	\\
	|t|^{-\sigma} |x|^{2-n}
	&
	\mbox{ \ \ if \ }  |x| \ge |t|^{\frac 12}.
\end{cases}
\end{equation}
\end{lemma}
\begin{remark}
Just like \eqref{a1},
$w_{2j}^{*}$, $ j = 1,\dots, k-1$, $w_3^*$ are approximate to some non-increasing functions about $|x|$ for every fixed $t$.
\end{remark}

The proofs of Lemma \ref{lem:w1j-w1j*}, \ref{lem:w2-w2*} and \ref{lem:w3-w3*} are deferred to subsection \ref{out-linear-pf} in the appendix.

For a function $h=h(x,t)$, we define the weighted $L^{\infty} $ norm $\|h\|^{out}_{\alpha,\sigma}$, $\|h\|^{out,*}_{\alpha,\sigma}$ as the following form respectively.
\begin{align}\label{def:out}
\|h\|^{out}_{\alpha,\sigma}
:=
\inf
\left\{ K \ \Big| \
	|h(x,t)|\leq K\left(\sum_{j=1}^kw_{1j}
	+
	\sum_{j=1}^{k-1}
	w_{2j}+w_3\right)(x,t),\quad \Rn\times (-\infty,t_0)
	\right\} .
\end{align}
\begin{align}\label{def:out*}
\|h\|^{out,*}_{\alpha,\sigma} :=
\inf
\left\{ K \ \Big| \
	|h(x,t)|\leq K\left(\sum_{j=1}^kw^*_{1j}
	+
	\sum_{j=1}^{k-1}
	w^*_{2j}
	+
	w^*_3\right)(x,t),\quad \Rn\times (-\infty,t_0)
	\right\} .
\end{align}
Using Lemma \ref{lem:w1j-w1j*}, \ref{lem:w2-w2*}  \ref{lem:w3-w3*} and Lemma \ref{lem:outer-1}-\ref{lem:outer-5}, we get the following proposition:
\begin{prop}\label{prop:G-bd}
	Suppose that $\sigma,\epsilon>0$ is chosen small enough, $\| \vec{\mu}_1\|_{\sigma}\le 1$ and $t_0$ is negative enough. Then there exists a constant $C^{out}>0$, independent of $R$ and $t_0$, such that $\partial_t \psi=\Delta_x\psi+\GG[\vec{\phi},\Psi,\vec{\mu}_1]$ has a solution $\mathcal{T}^{out}[\GG[\vec{\phi},\Psi,\vec{\mu}_1]]$ in $\Rn\times (-\infty,t_0)$ satisfying
	\begin{align*}
	\|\mathcal{T}^{out}[\GG[\vec{\phi},\Psi,\vec{\mu}_1]]\|^{out,*}_{\alpha,\sigma}\leq C^{out}R^{\alpha-a}
	\left(1+\|\vec{\phi}\|^{in,*}_{a,\sigma}+\|\Psi\|^{out,*}_{\alpha,\sigma}+(\|\vec{\phi}\|^{in,*}_{a,\sigma})^p+(\|\Psi\|^{out,*}_{\alpha,\sigma})^p\right)
	\end{align*}
	where $\GG$ is defined in \eqref{def:calG} and $\mathcal{T}^{out}[g]$ is given by \eqref{duhamel}.
\end{prop}

\begin{remark}\label{rmk:weight-2}
 There are some subtlety to bound $V\Psi$. Some term in $V\Psi$ can not be much smaller than $w_{1j}$ in the sense of $L^{\infty}$(see \eqref{Vpsi:1-zeta}). Thanks to its narrow support, we could still get the smallness when applying $\mathcal{T}^{out}$ to it. The estimate of $\mathcal{N}$ and $E^{out}$ are straightforward.
\end{remark}
\section{Orthogonal equations}\label{sec:Ortho}

In this section, we deal with the orthogonal equations
\begin{equation}\label{H equation}
\int_{B_{8R}}\HH_j[\Psi,\vec{\mu}_1](y,t)Z_{n+1}(y)dy=0,
\mbox{ \ \ for \ } j= 1, \dots, k.
\end{equation}
\begin{lemma}
\eqref{H equation} is equivalent to
    \begin{equation}\label{mu 1 eq}
\begin{cases}
\dot\mu_{11}
= M_1[\Psi,\vec{\mu}_1] (t) ,
	\\
\dot\mu_{1j}
+
\frac{n-4}{2}
\frac{\alpha_j}{t}
\mu_{1j}
-
\frac{n-2}{2}
\frac{\alpha_j}{t}
\lambda_{0j}
\mu_{1,j-1}
= 	M_j[\Psi,\vec{\mu}_1](t),
\mbox{ \ \ for \ } j=2, \dots, k,
\end{cases}
\end{equation}
where $M_j$ are given in \eqref{def:M1} and \eqref{def:Mj}.
\end{lemma}
\begin{proof}
For $j=1$, using \eqref{def:Hj} and \eqref{W-2.24}, \eqref{H equation} is equivalent to
\begin{equation}
	\dot\mu_{11}
	= M_1[\Psi, \vec{\mu}_1] (t) ,
\end{equation}
where
\begin{equation}\label{def:M1}
M_1[\Psi,  \vec{\mu}_1] (t) =
	-
	\frac{\mu_1^{\frac{n-2}{2}}}{1+\mu_{11}} \frac{\int_{B_{8R}}\zeta_1(\mu_1 y) p U(y)^{p-1}Z_{n+1}(y)\Psi(\mu_1y,t)dy}{\int_{B_{8R}}Z_{n+1}^2(y)dy} .
\end{equation}
	For $j=2,\dots,k$, by \eqref{def:Hj}
and \eqref{W-2.23},
\eqref{H equation} is equivalent to
\begin{align}
&
\dot{\mu}_{0j}\mu_{1j}(t)+\mu_{0j}\dot\mu_{1j}(t)
+\frac{n-2}{2}\frac{U(0)\int_{B_{8R}}pU^{p-1}(y)Z_{n+1}(y)dy}{\int_{B_{8R}}Z_{n+1}^2(y)dy}\lambda_{0j}^{\frac{n-2}{2}}\left(\frac{\mu_{1j}}{\mu_{0j}}-\frac{\mu_{1,j-1}}{\mu_{0,j-1}}\right)\notag
\\
&= \
 - \frac{\mu_j^{\frac{n-2}{2}}
 \int_{B_{8R}}
 \zeta_j(\mu_j y) p U(y)^{p-1} Z_{n+1}(y)\Psi(\mu_j y,t) \dd y}{\int_{B_{8R}}Z_{n+1}^2(y) \dd y} .\label{dmu0j}
\end{align}
Since $|Z_{n+1}(y)| \lesssim \langle y \rangle^{2-n}$,
\[
\begin{aligned}
	\int_{B_{8R}} p U^{p-1}(y) Z_{n+1}(y) d y &=\int_{\mathbb{R}^{n}} p U^{p-1}(y) Z_{n+1}(y) d y+O\left(R^{-2} \right) ,
	\\
	\int_{B_{8R}} Z_{n+1}^{2}(y) d y &=\int_{\mathbb{R}^{n}} Z_{n+1}^{2}(y) d y+O\left( R^{4-n}\right) .
\end{aligned}
\]
It follows that
\[
-\frac{U(0)  \int_{B_{8R}}
p U^{p-1}(y) Z_{n+1}(y) d y}{\int_{B_{8R}} Z_{n+1}^{2}(y) d y}= c_*+O\left(R^{-2}\right) ,
\]
where $c_*$ is the positive constant defined in \eqref{W-2.16}.
Notice the fact that
\[\frac{\dot\mu_{0j}}{\mu_{0j}}=\frac{\alpha_j}{-t},\quad \frac{c_*\lambda_{0j}^{\frac{n-2}{2}}}{\mu_{0j}^2}=\frac{\dot\mu_{0j}}{\mu_{0j}}
\text{ \ \ for \ } j=2,\dots,k.
\]
We can simplify \eqref{dmu0j} to
\begin{equation}
		\dot\mu_{1j}
		+
		\frac{n-4}{2}
		\frac{\alpha_j}{t}
		\mu_{1j}
		-
		\frac{n-2}{2}
		\frac{\alpha_j}{t}
		\lambda_{0j}
		\mu_{1,j-1}
		= 	M_j[\Psi,\vec{\mu}_1](t) ,
\end{equation}
where
\begin{equation}\label{def:Mj}
M_j[\Psi,\vec{\mu}_1](t)  =
-
\frac{\mu_j^{\frac{n-2}{2}}}{\mu_{0j}} \frac{\int_{B_{8R}}
\zeta_j(\mu_j y)
p U(y)^{p-1} Z_{n+1}(y)\Psi(\mu_j y,t) \dd y}{ \int_{B_{8R}}Z_{n+1}^2(y) \dd y}
+
\dot{\mu}_{0j}
\left(\frac{\mu_{1j}}{\mu_{0j}}-\frac{\mu_{1,j-1}}{\mu_{0,j-1}}\right)O(R^{-2})
	.
\end{equation}
\end{proof}

In order to solve \eqref{mu 1 eq} by the fixed point theorem,  we  reformulate \eqref{mu 1 eq} as the following  mapping.
Let us define $\vec{\mathcal{S} }[\Psi,\vec{\mu}_1 ] = (\mathcal{S}_1[\Psi,\vec{\mu}_1 ], \dots,\mathcal{S}_k[\Psi,\vec{\mu}_1 ])$ where
 and
\begin{equation}\label{def:Sj}
\begin{aligned}
\mathcal{S}_1
[\Psi,\vec{\mu}_1 ](t)
& =
\int_{-\infty}^{t}
M_1[\Psi, \vec{\mu}_1] (s)
\dd s,
\\
\mathcal{S}_j
[\Psi,\vec{\mu}_1 ](t)
& =
(-t)^{-\frac{n-4}{2}\alpha_j}\int_{ t_0 }^t(-s)^{\frac{n-4}{2}\alpha_j}
\left(
\frac{n-2}{2}\frac{\alpha_j}{s}\lambda_{0j}(s)
\mathcal{S}_{j-1}
[\Psi,\vec{\mu}_1 ](s)
+
M_j[\Psi,\vec{\mu}_1 ](s)
\right)
\dd s ,
\end{aligned}
\end{equation}
for $ j = 2,\dots, k$.

For a constant $b$ and a function $g(t)$, we define
\begin{align}\label{W-5.14}
	\|g\|_b^{\#}:=\sup\limits_{t\le t_0}|(-t)^b g(t)| .
\end{align}
We introduce the norm about $\vec{\mu}_1$:
\begin{equation}\label{mu1-topo}
	\|\vec{\mu}_1  \|_{\sigma}
	:=
	\sum\limits_{i=1}^{k}
	\left(
	\|\dot{\mu}_{1i} \|_{1+\alpha_i +\sigma}^{\#}
	+
	\| \mu_{1i} \|_{\alpha_i +\sigma}^{\#}
	\right) ,
\end{equation}
where $\sigma>0$.

\begin{lemma}\label{lem:orth-eq-est}
    Suppose $\Psi$ and $\vec{\mu}_1$ satisfy $\|\Psi\|_{\alpha,\sigma}^{out,*}<\infty$, $\|\vec{\mu}_1\|_\sigma\le 1$,  $0<\sigma<1$ respectively, when $t_0$ is very negative, there exists  $C^\mathcal{S}$ such that
  \begin{equation}
    \|\vec{\mathcal{S}}[\Psi,\vec{\mu}_1]\|_\sigma\leq C^{\mathcal{S}}(\|\Psi\|_{\alpha,\sigma}^{out,*}+O(R^{-2})) .
   \end{equation}
\end{lemma}
\begin{proof} Note that the support of $\zeta_1 $ is contained in $\{ R^{-1}\mu_{01} \le |x| \le 2R \mu_{01}\}$. By Lemma \ref{lem:w1j*-cmp} and \ref{lem:w2j*-cmp}, we have $|\Psi|\lesssim (w_{11}^*+ w_{12}^* +w_{21}^*)\|\Psi\|_{\alpha,\sigma}^{out,*}\lesssim |t|^{-1-\sigma}\|\Psi\|_{\alpha,\sigma}^{out,*}$. Then using \eqref{def:M1}, we have
\begin{align}
    |M_1[\Psi,\vec{\mu}_1]|\lesssim |t|^{-1-\sigma}\|\Psi\|_{\alpha,\sigma}^{out,*}.
\end{align}
By \eqref{def:Sj}, we have
\begin{align}
    \|\dot{\mathcal{S}}_1 [\Psi,\vec{\mu}_1] \|_{1+\sigma}^{\#}
    +\|\mathcal{S}_1[\Psi,\vec{\mu}_1] \|_{\sigma}^{\#}\lesssim \|\Psi\|_{\alpha,\sigma}^{out,*}.
\end{align}
 Similarly, for $j=2,\dots, k$, the support of $\zeta_j$ is contained in $\{ R^{-1}\mu_{0j} \le |x| \le 2R \mu_{0j}\}$. By Lemma \ref{lem:w1j*-cmp} and \ref{lem:w2j*-cmp}, we have
 \begin{equation*}
|\Psi|\lesssim (w_{1j}^*+w_{1,j+1}^*+ w_{2j}^*+w_{2,j-1}^*)\|\Psi\|_{\alpha,\sigma}^{out,*}\lesssim |t|^{\gamma_j}\|\Psi\|_{\alpha,\sigma}^{out,*},
 \end{equation*}
  where $w_{1,j+1}^*$, $w_{2j}^*$ are vacuum if $j=k$.

 Then using \eqref{def:Mj}, we have
\begin{equation}\label{Mj est}
    |M_j [\Psi,\vec{\mu}_1] |\lesssim  |t|^{-\sigma}
    \mu_{0j}^{-1} \lambda_{0j}^{\frac{n-2}{2}}
    \|\Psi\|_{\alpha,\sigma}^{out,*}
    +
    |t|^{-1-\alpha_j-\sigma}
    O(R^{-2})
    \lesssim |t|^{-\alpha_j-1-
    \sigma}(\|\Psi\|_{\alpha,\sigma}^{out,*}+O(R^{-2})),
\end{equation}
where we have used that $\mu_{0j}\dot{\mu}_{0j} = c_{*} \lambda_{0j}^{\frac{n-2}{2}}$.

We will prove
\begin{equation}
\|\dot{\mathcal{S}}_j [\Psi,\vec{\mu}_1]\|_{1+\alpha_j+\sigma}^{\#}
+\|\mathcal{S}_j [\Psi,\vec{\mu}_1] \|_{\alpha_j+\sigma}^{\#}
\lesssim \|\Psi\|_{\alpha,\sigma}^{out,*}+O(R^{-2}),
\end{equation}
by induction. The case $j=1$ has been proved.

Suppose we have proved $\|\dot{\mathcal{S}}_{j-1} [\Psi,\vec{\mu}_1]\|_{1+\alpha_{j-1}+\sigma}^{\#}
+\|\mathcal{S}_{j-1} [\Psi,\vec{\mu}_1] \|_{\alpha_{j-1}+\sigma
}^{\#}
\lesssim \|\Psi\|_{\alpha,\sigma}^{out,*}+O(R^{-2})$ by induction. Consequently  $\left|s^{-1}\lambda_{0j}(s)\mathcal{S}_{j-1}(s)\right|\lesssim (-s)^{-\alpha_j-1-\sigma}\left(\|\Psi\|_{\alpha,\sigma}^{out,*}+O(R^{-2}) \right)$.  Now using \eqref{def:Sj} and \eqref{Mj est},
\begin{equation}
\begin{aligned}
   | \mathcal{S}_j[\Psi,\vec{\mu}_1]|\lesssim\ & (-t)^{-\frac{n-4}{2}\alpha_j}\int_{t_0}^t(-s)^{\frac{n-4}{2}\alpha_j}(-s)^{-\alpha_j-1-\sigma}ds(\|\Psi\|_{\alpha,\sigma}^{out,*}+O(R^{-2}))
   \\
   \lesssim\ &(-t)^{-\alpha_j-\sigma}(\|\Psi\|_{\alpha,\sigma}^{out,*}+O(R^{-2}) ),
\end{aligned}
\end{equation}
where we have used $\sigma<1=\min_{2\leq j\leq k}\{\frac{n-6}{2}\alpha_j\}$. Similarly, we can get
\[|\dot{ \mathcal{S}}_j[\Psi,\vec{\mu}_1]|\lesssim (-t)^{-1-\alpha_j-\sigma}(\|\Psi\|_{\alpha,\sigma}^{out,*}+O(R^{-2}) ).\] This completes the induction.
\end{proof}

\section{The Schauder fixed point argument}

In this section, we will solve the system \eqref{W-3.6}-\eqref{W-3.7} by fixed point argument. We need to set up appropriate topology and operators. Recall \eqref{def:phi^*_nu}, \eqref{def:h^in_nu} and \eqref{def:Tin} in the previous section.
When $\nu(\tau)=(-t)^{\gamma_j}\mu_{0j}^{\frac{n-2}{2}}$, we write
\begin{align}\label{inner-symbol}
    \|\langle y\rangle \nabla \phi\|_{\nu(\tau),a}^{in,*}+\|\phi\|_{\nu(\tau),a}^{in,*}=\|\phi\|_{j,a,\sigma}^{in,*},\quad \|h\|_{\nu(\tau),a}^{in}=\|h\|_{j,a,\sigma}^{in},\quad \mathcal{T}_{\nu(\tau)}^{in}=\mathcal{T}_{j}^{in}
\end{align}
for short,
where $0<a<1$, $\gamma_1 = -1-\sigma$ and
$\gamma_{j} = \frac{n-2}{2}\alpha_{j-1}-\sigma$, $j= 2, \dots, k$.

Now we state precisely the topology we are going to use. Suppose $\sigma$ is small enough and $0<\alpha<a<1$. Define
\begin{align}\label{vecphiin}
 \|\vec{\phi}\|_{a,\sigma}^{in,*}:=\sum_{j=1}^k\|\phi_j\|_{j,a,\sigma}^{in,*} .
\end{align}
We will cope with $\vec{\phi}$, $\Psi$, $\vec{\mu}_{1}$ in the topology \eqref{vecphiin},  \eqref{def:out*} and  \eqref{mu1-topo} respectively.

The following lemma justifies why we choose $\nu(\tau)=(-t)^{\gamma_j}\mu_{0j}^{\frac{n-2}{2}}$.
\begin{lemma}\label{lem:inner-h}
For any $R>0$ large, there exists $t_0$ negative enough such that for $t<t_0$ one has
\begin{align}
    |\mathcal{H}_j(x,t)|&\leq C^{\mathcal{H}}
    \mu_{0j}^{\frac{n-2}{2}} (-t)^{\gamma_j} \langle y_j\rangle^{-4}(\|\vec{\mu}_1\|_\sigma+\|\Psi\|_{\alpha,a}^{out,*}),\quad j=1,\cdots,k.
\end{align}
\end{lemma}
\begin{proof}

By \eqref{mu1-topo}, we have
\begin{align}
	\begin{split}
	    |D_1[\vec{\mu}_1]|\lesssim& |\dot \mu_{11}Z_{n+1}(y_1)|\lesssim (-t)^{\gamma_1}\langle y_1\rangle^{2-n}\|\vec{\mu}_1\|_{\sigma},\\
		|D_j[\vec{\mu}_1]|\lesssim& \lambda_j(t)^{\frac{n-2}{2}}(-t)^{-\sigma}(|Z_{n+1}(y_j)|+\langle y_j\rangle^{-4})\|\vec{\mu}_1\|_{\sigma}
		\lesssim
		\mu_{0j}^{\frac{n-2}{2}} (-t)^{\gamma_j}  \langle y_j\rangle^{-4}\|\vec{\mu}_1\|_{\sigma},
		\label{Djmu-1}
	\end{split}
\end{align}
for $j=2,\dots,k$. By the same estimate in Lemma \ref{lem:orth-eq-est}, we have
\begin{align}\label{xijUp-1}
	|\zeta_jU(y_j)^{p-1}\mu_j^{\frac{n-2}{2}}\Psi|\lesssim \mu_{0j}^{\frac{n-2}{2}} (-t)^{\gamma_j} \langle y_j\rangle^{-4}\|\Psi\|_{\alpha,a}^{out,*},
	\quad j=1,\dots,k .
\end{align}
\end{proof}

We reformulate the inner-outer gluing system and the orthogonal equation into the mapping $\vec{T}$:
\begin{equation}\label{final system}
	(\vec{\phi},\Psi, \vec{\mu}_1)
	= \vec{T}[\vec{\phi},\Psi, \vec{\mu}_1],
\end{equation}
where $\vec{T}= (\vec{T}^1, T^2,\vec{T}^3)$, $\vec{T}^1 = (\vec{T}^1_1,\dots,\vec{T}^1_k)$, $\vec{T}^3 = (\vec{T}^3_1,\dots,\vec{T}^3_k)$, with the following expressions,
\begin{align}
	\vec{T}^1_j[\Psi, \vec{\mu}_1 ]
	= \ &
\TT^{in}_{j} [\HH_j[\Psi,\vec{\mu}_1]
- c_{j}\left[\Psi, \vec{\mu}_{1}\right] Z_{n+1}
] , \quad
j=1, \ldots, k,\label{T1j}\\
T^2[\vec{\phi},\Psi, \vec{\mu}_1]
= \ &
\TT^{out}[\GG[\vec{\phi},\Psi, \vec{\mu}_1]]
 ,	
\label{T2out}\\
\vec{T}^3_j[\Psi, \vec{\mu}_1]
= \ & \mathcal{S}_j[\Psi, \vec{\mu}_1 ],
\quad
j=1, \ldots, k.\label{T3j}
\end{align}
where $c_{j}\left[\Psi, \vec{\mu}_{1}\right](t) = \| Z_{n+1} \|_{L^2(B_{8R})}^{-2} \int_{B_{8R}} \HH_{j}\left[\Psi, \vec{\mu}_{1}\right] (y,t) Z_{n+1}(y) \dd y $.
Here $\TT^{in}_j$ in \eqref{T1j} is obtained from \eqref{def:Tin}. It is well-defined because $\mathcal{H}_j-c_j Z_{n+1}$ satisfies \eqref{h-ortho-1}.
Denote
\begin{equation}\label{final topology}
\B_{out} = \{ \Psi \ | \ \|\Psi\|^{out,*}_{\alpha,\sigma} \le R^{-2\rho}\}
,
\quad
\B_{in} = \{ \vec{\phi} \ | \
\|\vec{\phi}\|_{a,\sigma}^{in,*} \le 1
\},
\quad
\B_{mu} = \{
\vec{\mu}_1
\ | \
\|\vec{\mu}_1 \|_{\sigma} \le R^{-\rho}
\} ,
\end{equation}
where $0<2\rho\le \frac{a-\alpha}{2}$ is a small constant.
\begin{proof}[Proof of Theorem \ref{intro:thm:main}:] $\bullet$ Existence part. Fix $0<\alpha<a<1$. We choose $\sigma,\epsilon>0$ small enough such that Proposition \ref{prop:G-bd} holds.
Let $\B=\B_{in}\times \B_{out}\times \B_{mu}$, then we claim that $\vec{T}$ maps $\B$ to $\B$ provided taking $R$ large enough and $t_0$ negative enough.

First, for any fixed $\Psi \in \B_{out}$, $\vec{\mu}_1 \in \B_{mu}$. Applying \eqref{ynablayphi} with $h=\mathcal{H}_j$ and $\nu(\tau)=(-t)^{\gamma_j}\mu_{0j}^{\frac{n-2}{2}}$, Lemma \ref{lem:inner-h} implies that
\begin{align}
    \|\vec{T}^1[\Psi,\vec{\mu}_1]
    \|_{a,\sigma}^{in,*}\leq
    \sum\limits_{j=1}^{k}
    C_{j,a}^{in}C^{\mathcal{H}}(R^{-\rho}+R^{-2\rho})\leq 1
\end{align}
provided taking $R$ large enough.

Second, for any fixed $\vec{\phi} \in \B_{in},$ $\vec{\mu}_1 \in \B_{mu}$, by Proposition \ref{prop:G-bd}, there exists $t_0=t_0(R)$ negative enough such that
\begin{equation*}
\| T^2[ \vec{\phi},\Psi, \vec{\mu}_1] \|^{out,*}_{\alpha,\sigma}
\le C^{out}R^{\alpha-a}(3+2e^{-\rho})\leq R^{-2\rho}.
\end{equation*}

Third, for any fixed $\vec{\phi}\in \B_{in}$, $\Psi\in \B_{out}$. It follows from Lemma \ref{lem:orth-eq-est} that
\begin{align}
    \|\vec{T}^3
    [\Psi, \vec{\mu}_{1}]
    \|_{\sigma}\leq C^{\mathcal{S}}(R^{-2\rho}+O(R^{-2}))\leq R^{-\rho}.
\end{align}
Therefore $\vec{T}:\B\to \B$.

Next we need to show $\vec{T}$ is a compact mapping. Thus for any sequence  $(\vec{\phi}^m,\Psi^m,\vec{\mu}_{1}^{m})\in \B$, where $\vec{\phi}^m = (\phi^m_{1},\dots,\phi^m_{k})$, $\vec{\mu}_{1}^{m}=(\mu^{m}_{1},\dots, \mu^{m}_{k})$, we have to show $\vec{T}[\vec{\phi}^m,\Psi^m,\vec{\mu}_{1}^m]$ has a convergent subsequence. Let us consider first the sequence $\tilde \phi_{j}^{m}=\TT^{in}_{j} [\HH_j[\Psi^m,\vec{\mu}_1^m]
- c_{j}\left[\Psi^m, \vec{\mu}_{1}^m\right] Z_{n+1}
]$. It satisfies
\[\partial_\tau \tilde \phi_{j}^{m}=\Delta_y\tilde \phi_{j}^{m}+h_{j}^{m}(y,\tau), \quad h_{j}^{m}=\HH_j[\Psi^m,\vec{\mu}_1^m]
- c_{j}\left[\Psi^m, \vec{\mu}_{1}^m\right] Z_{n+1}.
\]
Using Lemma \ref{lem:inner-h} and interior estimate of parabolic equations, we get that in any compact set $K\subset B_{8R}\times(-\infty, t_0)$, we have $\tilde\phi_{j}^m\in C^{1+\gamma,\frac{1+\gamma}{2}}$ in $K$ for each fixed $\gamma\in(0,1)$. Thus $\tilde\phi_{j}^m$ and $\nabla_y \tilde\phi_{j}^m$ are equi-continuous in $K$. By Arzel\`a-Ascoli theorem, going to a subsequence if necessary, $\tilde \phi_{j}^m$ will converge uniformly in compact sets of $B_{8R}\times(-\infty,t_0)$. Since $\tilde{\phi}_{j}^m\in \B_{in}$, then the limit will also belong to $\B_{in}$.

Second, consider $\tilde \Psi^m=\mathcal{T}^{out}[\GG[\vec{\phi}^m,\Psi^m,\vec{\mu}_{1}^m]]$. Since $\GG[\vec{\phi}^m,\Psi^m,\vec{\mu}_{1}^m]$ are uniformly bounded, $\tilde{\Psi}^m$ have a uniform $C^{1+\gamma,\frac{1+\gamma}{2}}$ bound in compact sets of $\Rn\times(-\infty,t_0)$.
By Arzel\`a-Ascoli theorem, $\tilde{\Psi}^m$ (up to a subsequence) converges uniformly to a function $\tilde{\Psi}\in \B_{out}$.

Third, consider $\mathcal{S}[\Psi^m,\vec{\mu}_{1}^m]$. Note that \eqref{def:M1} and \eqref{def:Mj} imply $M_1[\Psi^m,\vec{\mu}_{1}^m]$ and $M_j[\Psi^m,\vec{\mu}_1^m]$ are $ C^1(-\infty,t_0)$. Thus $\mathcal{S}[\Psi^m,\vec{\mu}_{1}^m]\in C^2(-\infty,t_0)$. Consequently $\mathcal{S}[\Psi^m,\vec{\mu}_{1}^m]$ has a convergent subsequence in $\B_{mu}$.

By Schauder's fixed point theorem, $\vec{T}:\B\to \B$ has a fixed point $(\vec{\phi},\Psi,\vec{\mu}_1)$. Then \eqref{T3j} implies $c_j[\Psi,\vec{\mu}_1]=0$ and consequently  $(\vec{\phi},\Psi,\vec{\mu}_1)$ makes \eqref{W-3.6} and \eqref{W-3.7} hold.

We have constructed a bubble tower solution for \eqref{intro:eq:main}. Recall that $u=\bar U+\varphi_0+\sum_{j=1}^k\varphi_j \eta_j+\Psi$. One can see \eqref{def:barU}, \eqref{W-2.6} and \eqref{W-3.4} for their respective definitions. We shall prove that $\bar U$ dominates in the sum in the sense of $L^{\infty}$ and $H^{1}$ when $t_0$ is negative enough.

$\bullet$ Convergence in  $L^{\infty}(\Rn)$ and positiveness.
It is easy to see the first approximation and inner solutions are smaller than $\bar U$. Namely, by Lemma \ref{lem:phi0}, we know $|\varphi_0| \lesssim |t|^{-\epsilon} \bar{U}$. For $j=1,\dots, k$, by \eqref{phi-est}  and \eqref{final topology},
\begin{equation*}
\begin{aligned}
\left|
\mu_{j}^{-\frac{n-2}{2}}
\phi_{j}(\frac{x}{\mu_j},t) \eta_{j}
\right| & \lesssim
|t|^{\gamma_{j}} R^{n+1-a}
\langle y_{j} \rangle^{-n-1} \1_{\{|x|\le 4R \mu_{0j} \} }
\\
& \approx
\mu_{j}^{\frac{n-2}{2}} |t|^{\gamma_j}
 R^{n+1-a}
\langle y_{j} \rangle^{- 3} \1_{\{|x|\le 4R \mu_{0j} \} } U_{j}(x,t)
 \lesssim
|t|^{-\sigma}  U_{j}(x,t) .
\end{aligned}
\end{equation*}
The solution $\Psi$ in the outer problem is more involved to estimate. First, by \eqref{final topology}, we have
\begin{equation*}
	|\Psi| \lesssim
	\sum\limits_{j=1}^k w_{1j}^* + \sum\limits_{j=1}^{k-1} w_{2j}^* + w_3^*.
\end{equation*}
We will estimate it in several regions. We will use Lemma \ref{lem:w1j*-cmp} and \ref{lem:w2j*-cmp} repeatedly in the following argument.

In $\{|x|\le \bar{\mu}_{0k} \}$, we have
\begin{equation*}
|\Psi|  \lesssim
w_{1k}^{*} + w_{2,k-1}^*
 \lesssim
 |t|^{\gamma_{k}}
 \approx
 |t|^{\gamma_{k}}
 \mu_{k}^{\frac{n-2}{2}}
 \langle y_k \rangle^{n-2} U_{k}(x,t) \lesssim
 |t|^{-\sigma}
 U_{k}(x,t) .
\end{equation*}
In $\{\bar{\mu}_{0,i+1}\le |x| \le \bar{\mu}_{0i} \}$, $i=2,\dots k-1$, we have
\begin{equation*}
\begin{aligned}
	|\Psi| & \lesssim
w_{1i}^{*} + w_{1,i+1}^{*} + w_{2i}^{*} +w_{2,i-1}^{*}
\\
& \approx
 (w_{1i}^{*} + w_{1,i+1}^{*} + w_{2i}^{*} +w_{2,i-1}^{*})
\mu_{i}^{\frac{n-2}{2}}
\langle y_i \rangle^{n-2} U_{i}(x,t)
\lesssim
|t|^{-\sigma} U_{i}(x,t).
\end{aligned}
\end{equation*}
This is because
\begin{align*}
w_{1i}^{*}
\mu_{i}^{\frac{n-2}{2}}
\langle y_i \rangle^{n-2}
&
\lesssim
|t|^{\gamma_i}
\mu_{i}^{\frac{n-2}{2}}
(\frac{\bar{\mu}_{0i}}{\mu_{i}})^{n-2}
\lesssim |t|^{-\sigma} ,
\\
 w_{1,i+1}^{*}
\mu_{i}^{\frac{n-2}{2}}
\langle y_i \rangle^{n-2}
&
\lesssim
|t|^{\gamma_{i+1}} \bar{\mu}_{i+1}^{n-2} |x|^{2-n}
\mu_{i}^{\frac{n-2}{2}}
\langle y_i \rangle^{n-2}
\lesssim
|t|^{\gamma_{i+1}} \bar{\mu}_{i+1}^{n-2}
\mu_{i}^{\frac{n-2}{2}}
\bar{\mu}_{0,i+1}^{2-n}
\lesssim
|t|^{-\sigma},
\\
w_{2i}^{*}
\mu_{i}^{\frac{n-2}{2}}
\langle y_i \rangle^{n-2}
&
\approx
|t|^{-2\sigma}
\mu_{i+1}^{\frac n2 -2} \mu_{i}^{\frac n2 -2} |x|^{4-n} \langle y_i \rangle^{n-2}
\\
&
\lesssim
|t|^{-2\sigma}
\mu_{i+1}^{\frac n2 -2} \mu_{i}^{\frac n2 -2}
\left(
\bar{\mu}_{0,i+1}^{4-n}
+
\bar{\mu}_{0i}^2 \mu_{i}^{2-n}
\right)
\lesssim
|t|^{-2\sigma},
\\
w_{2,i-1}^{*}
\mu_{i}^{\frac{n-2}{2}}
\langle y_i \rangle^{n-2}
&
\lesssim
|t|^{-2\sigma} \mu_{i-1}^{1-\frac n2}
\mu_{i}^{\frac{n-2}{2}}
(\frac{\bar{\mu}_{0i}}{\mu_{i}})^{n-2}
\lesssim
|t|^{-2\sigma} .
\end{align*}
In $\{ \bar{\mu}_{02} \le |x| \le \bar{\mu}_{01} \}$,
similarly, we have
\begin{equation*}
		|\Psi| \lesssim
	(w_{11}^{*} + w_{12}^{*} + w_{21}^{*} )
		(1+|x|)^{n-2} U_{1}(x,t)
	 \lesssim
		|t|^{-\sigma} U_{1}(x,t),
\end{equation*}
since for $\delta\le (n-2-\alpha)^{-1}$,
\begin{equation*}
\begin{aligned}
w_{11}^{*}
(1+|x|)^{n-2}
& \lesssim
|t|^{\gamma_1}\langle x \rangle^{-\alpha} (1+|x|)^{n-2}\lesssim |t|^{\gamma_1}|t|^{\delta(n-2-\alpha)}=|t|^{-\sigma},
\\
 w_{12}^{*}
(1+|x|)^{n-2}
& \lesssim
|t|^{\gamma_2} \bar{\mu}_{2}^{n-2} |x|^{2-n} (1+|x|)^{n-2}
\lesssim
|t|^{-\sigma},
\\
 w_{21}^{*}
(1+|x|)^{n-2}
&
\lesssim
|t|^{-2\sigma}
\bar{\mu}_{2}^{n-4} |x|^{2-n}
\min\{1,|x|^2\}
(1+|x|)^{n-2}
\lesssim
|t|^{-2\sigma} .
\end{aligned}
\end{equation*}
In $\{   \bar{\mu}_{01} \le |x| \le |t|^{\frac 12}\}$,
\begin{equation*}
|\Psi|\lesssim w_{11}^{*} +w_{3}^{*} \lesssim |t|^{-\frac{\sigma}{2}} U_{1}(x,t).
\end{equation*}
In  $\{   |x| \ge |t|^{\frac 12}\}$,
\begin{equation*}
	|\Psi|\lesssim w_{3}^{*} \lesssim |t|^{-\frac{\sigma}{2}} U_{1}(x,t).
\end{equation*}
 Therefore $|\Psi|\lesssim |t|^{-\frac{\sigma}{2}}\bar U$.

Combining the above analysis, we have $u = \bar{U}(1+O(|t|^{-\epsilon})) >0$. By the parabolic regularity theory, we improve the regularity of $u$ to be smooth.

$\bullet$ Convergence in  $H^{1}(\Rn)$. The solution we construct is $u=\bar U+\varphi_0+\varphi$, where $\varphi_0$ is from \eqref{W-2.6} and $\varphi$ is from  \eqref{W-3.4}.
Set $\bar{\varphi} = \varphi_{0} + \varphi$. We have already proved $|\bar{\varphi}|\lesssim |t|^{-\epsilon} \bar{U}$. Formally, we can expect $|\nabla_x\bar\varphi|\lesssim |t|^{-\epsilon}|\nabla_x \bar U|$.  Note that $\bar{\varphi}$ satisfies
\begin{equation}
\pp_{t} \bar{\varphi} = \Delta_{x} \bar{\varphi} + f(x,t)
\end{equation}
where $f(x,t) = \left( \bar{U} + \bar{\varphi}\right)^p  -\sum_{j=1}^{k} U_j^p -\sum_{j=1}^{k}\partial_{t} U_{j}$. It follows that
\begin{equation*}
\begin{aligned}
|f(x,t) | \lesssim \ &
\begin{cases}
|t|^{-\epsilon}
\mu_{0k}^{-\frac{n+2}{2}}
\langle y_{0k} \rangle^{-n-2}
+
\mu_{0k}^{-\frac n2} |\dot{\mu}_{0k} | \langle y_{0k} \rangle^{2-n}
&
\mbox{ \ \ if \ }
 |x| \le \bar{\mu}_{0k},
	\\
|t|^{-\epsilon}
\mu_{0j}^{-\frac{n+2}{2}}
\langle y_{0j} \rangle^{-n-2}
+
\mu_{0j}^{-\frac n2} |\dot{\mu}_{0j} | \langle y_{0j} \rangle^{2-n}
&
\mbox{ \ \ if \ }
\bar{\mu}_{0,j+1} \le |x| \le \bar{\mu}_{0j},
\quad
j=2,\dots, k-1,
	\\
|t|^{-\epsilon} \langle x\rangle^{-n-2} + |t|^{-1-\sigma}  \langle x\rangle^{2-n}
&
\mbox{ \ \ if \ } |x| \ge \bar{\mu}_{02} ,
\end{cases}
\\
\approx \ &
\begin{cases}
|t|^{-\epsilon}
	\mu_{0k}^{-\frac{n+2}{2}}
	&
	\mbox{ \ \ if \ }
	|x| \le \mu_{0k},
	\\
|t|^{-\epsilon}
	\mu_{0j}^{-\frac{n+2}{2}}
	&
	\mbox{ \ \ if \ }
	\bar{\mu}_{0,j+1} \le |x| \le \mu_{0j},
	\quad
	j=2,\dots, k-1,
	\\
	|t|^{-\epsilon}
	\mu_{0j}^{\frac{n+2}{2}}
	|x|^{-n-2}
	&
	\mbox{ \ \ if \ }
	\mu_{0j} \le |x| \le \bar{\mu}_{0j},
	\quad
	j=2,\dots, k,
	\\
|t|^{-\epsilon}
	&
	\mbox{ \ \ if \ } \bar{\mu}_{02} \le |x| \le 1  ,
\\
|t|^{-\epsilon} |x|^{-n-2} + |t|^{-1-\sigma}  |x|^{2-n}
&
\mbox{ \ \ if \ } 1 \le |x| .
\end{cases}
\end{aligned}
\end{equation*}
since $|t|^{-\epsilon}
\mu_{0j}^{-\frac{n+2}{2}}
\langle y_{0j} \rangle^{-n-2}
\gtrsim
\mu_{0j}^{-\frac n2} |\dot{\mu}_{0j} | \langle y_{0j} \rangle^{2-n}$ in $\{ \bar{\mu}_{0,j+1} \le |x| \le \bar{\mu}_{0j}
\}$, $j=2,\dots, k$ when $\epsilon$ is small.

By the similar argument about uniqueness in Corollary \ref{phi*-corollary-ancient}, we know
\begin{equation*}
\bar{\varphi} = \TT^{out}[f].
\end{equation*}
Then
\begin{align}
	|\nabla_x\bar{\varphi}|
	\lesssim
\TT^{d}[|f|],
\end{align}
where
\begin{equation*}
\TT^d[g] : =
\int_{-\infty}^{t} \frac{d s}{(t-s)^{\frac{n}{2} +1}} \int_{\mathbb{R}^{n}} e^{-\frac{|x-y|^{2}}{4(t-s)}}
|x-y| g(y, s) \dd y \dd s.
\end{equation*}
Claim:
\begin{equation}\label{deri-est}
	|\nabla_x\bar{\varphi}|
\lesssim
|t|^{-\epsilon}
\sum\limits_{j=1}^{n}
\mu_{0j}^{-\frac n2} \langle y_{0j} \rangle^{1-n} .
\end{equation}
Notice $\mu_{0j}^{-\frac n2} \langle y_{0j} \rangle^{1-n} $ is approximate to $|\nabla_{x} U_{j}(x,t)|$.
 Once we complete the proof of \eqref{deri-est}, it is straightforward to have $\| u(\cdot,t) - \bar{U} \|_{H^1(\RR^n)} = O(|t|^{-\epsilon})$.

By Lemma \ref{lem:deri-est}, for $j=2,\dots, k-1$, we get
\begin{equation*}
\begin{aligned}
\TT^{d}[|t|^{-\epsilon}
\mu_{0j}^{-\frac{n+2}{2}}  \1_{\{ \bar{\mu}_{0,j+1} \le |x| \le \mu_{0j} \}}]
\lesssim \ &
\begin{cases}
|t|^{-\epsilon}
	\mu_{0j}^{-\frac{n}{2}}
	&
	\mbox{ \ \ if \ }
	|x|\le \mu_{0j},
	\\
|t|^{-\epsilon}
	\mu_{0j}^{\frac n2 -1} |x|^{1-n}
	&
	\mbox{ \ \ if \ }
	\mu_{0j} \le |x|\le |t|^{\frac 12} ,
	\\
	(|x|^2)^{-\epsilon -\alpha_j(\frac n2-1)}
	|x|^{1-n}
	&
	\mbox{ \ \ if \ }
|t|^{\frac 12} \le  |x| .
\end{cases}
\\
\lesssim \ &
|t|^{-\epsilon}
\mu_{0j}^{-\frac n2} \langle y_{0j} \rangle^{1-n} .
\end{aligned}
\end{equation*}
Similarly,
\begin{equation*}
\begin{aligned}
\TT^{d}[|t|^{-\epsilon}
\mu_{0k}^{-\frac{n+2}{2}} \1_{\{|x|\le \mu_{0k} \} }]
\lesssim \ &
|t|^{-\epsilon} \mu_{0k}^{-\frac n2} \langle y_{0k} \rangle^{1-n},
\\
\TT^{d}[|t|^{-\epsilon}  \1_{\{
	\bar{\mu}_{02} \le |x| \le 1 \}}]
\lesssim \ &
|t|^{-\epsilon}
\mu_{01}^{-\frac n2} \langle y_{01} \rangle^{1-n}.
\end{aligned}
\end{equation*}

For $j=2,\dots, k,$
\begin{equation*}
\begin{aligned}
\TT^{d}[|t|^{-\epsilon}
\mu_{0j}^{\frac{n+2}{2}}
|x|^{-n-2} \1_{\{\mu_{0j} \le |x| \le \bar{\mu}_{0j}\}}]
\lesssim \ &
\begin{cases}
|t|^{-\epsilon} \mu_{0j}^{-\frac n2}
	&
	\mbox{ \ \ if \ }
	|x| \le  \mu_{0j} ,
	\\		
|t|^{-\epsilon}
\mu_{0j}^{\frac n2 -1} |x|^{1-n}
	&
	\mbox{ \ \ if \ }
	\mu_{0j} \le
	|x| \le  |t|^{\frac 12},
	\\
(|x|^2)^{-\epsilon -\alpha_j(\frac n2-1)}
|x|^{1-n}
	&
	\mbox{ \ \ if \ }
	|x| \ge  |t|^{\frac 12},
\end{cases}
\\
\lesssim \ &
|t|^{-\epsilon}
\mu_{0j}^{-\frac n2} \langle y_{0j} \rangle^{1-n} .
\end{aligned}
\end{equation*}
Similarly,
\begin{equation*}
	\TT^{d}[ |t|^{-\epsilon} |x|^{-n-2} \1_{\{1 \le |x| \le |t|^{\frac 12} \}} ]
	\lesssim
	|t|^{-\epsilon}
\mu_{01}^{-\frac n2} \langle y_{01} \rangle^{1-n}.
\end{equation*}

The left part can be transformed into the estimate in Appendix.
\begin{equation*}
\begin{aligned}
	&
\TT^{d}[ \left[|t|^{-\epsilon} |x|^{-n-2} \1_{\{ |t|^{\frac 12} \le |x|\}} + |t|^{-1-\sigma}  |x|^{2-n} \1_{\{1 \le  |x|\}} \right] ]
\\
\lesssim \ &
\int_{-\infty}^{t} \frac{d s}{(t-s)^{\frac{n+1}{2} }} \int_{\mathbb{R}^{n}} e^{-\frac{|x-y|^{2}}{8(t-s)}}
\left[(-s)^{-\epsilon} |y|^{-n-3} \1_{\{ (-s)^{\frac 12} \le |y|\}} + (-s)^{-1-\sigma}  |y|^{1-n} \1_{\{1 \le |y|\}} \right]
|y|
 d y \dd s
 \\
 \lesssim \ &
 \left(
|t|^{-1-\epsilon}
+ |t|^{-\sigma}
\right)
\mu_{01}^{-\frac n2} \langle y_{01} \rangle^{1-n},
\end{aligned}
\end{equation*}
whose estimate process is similar to the convolution of Gaussian kernel in $\RR^{n+1}$.

This completes the proof of \eqref{deri-est}.
\end{proof}

\section*{Acknowledgement} The research of L. Sun and J. Wei is partially supported by NSERC of Canada.


\appendix

\section{Estimates for the data in the outer problem}
We will prove Proposition \ref{prop:G-bd} in this section.
Throughout this section, we assume $\| \Psi\|_{\alpha,\sigma}^{out,*} + \| \vec{\phi}\|_{a,\sigma}^{in,*}< \infty$, $\| \vec{\mu}_1\|_{\sigma}\le 1$.

The parameters are determined in the following order.
First, we choose $R$ as a large fixed positive constant. Second, we choose $\sigma>0$ small. Third, we choose  $\delta>0$ small. Fourth, we choose  $\epsilon>0$ small. Finally, we take $t_{0}$ very negative such that $\mu_{j} \approx \mu_{0j} $, for $j=1, \dots , k$, $\dot{\mu}_{j} \approx \dot{\mu}_{0j} $ for $j=2 , \dots , k$.

We introduce the notation $y_j=x/\mu_j$, $\bar y_{j}=x/\bar{\mu}_j$,
$y_{0j}=x/\mu_{0j}$, $\bar y_{0j}=x/\bar{\mu}_{0j}$ for $j=1 , \dots , k$. One readily sees that $|y_j| \approx |y_{0j}|$,
$|\bar{y}_j| \approx |\bar{y}_{0j}|$ for $j=1 , \dots , k$.
\begin{lemma}\label{lem:Uij}
	Consider the $U_j$ defined in \eqref{def:Uj}. For $j=1,\cdots,k-1$, one has
	\begin{align}\label{A-1}
	    U_j
	    < U_{j+1}\text{ in }\{|x|<\bar{\mu}_{j+1}\}
	    \text{ and }U_j > U_{j+1}\text{ in }\{|x|>\bar{\mu}_{j+1}\}.
	\end{align}
In $\{|x| \le \bar{\mu}_{0k} \}$,
	\begin{equation}
	U_{k} \gtrsim
	U_{k-1} > U_{k-2}
	> \dots
	> U_{1} .
	\end{equation}
In $\{|x|\ge \bar{\mu}_{02} \}$,
	\begin{equation}
	 U_{1} \gtrsim
	 U_{2} > U_3 >\dots
	 >
	 U_{k} .
	\end{equation}
In $\{\bar{\mu}_{0,j+1}\le |x| \le \bar{\mu}_{0j} \}$, $j=2, \dots, k-1$,
\begin{equation}\label{A-4}
U_{j} \gtrsim U_{j+1}
>
U_{j+2}
>
\dots
>
U_{k},
\quad
U_{j} \gtrsim U_{j-1}
>
U_{j-2}
>
\dots
> U_{1} .
\end{equation}
	Moreover
	\begin{align}\label{Uj-ratio-1}
			\frac{U_{j+1}}{U_j}
 & \approx
\lambda_{j+1}^{-\frac{n-2}{2}}
\langle y_{j+1} \rangle^{-(n-2)}
\1_{\{ |x|\le \mu_{0j}\}}
+
\lambda_{j+1}^{\frac{n-2}{2}}
\1_{\{ |x| > \mu_{0j}\}}
\mbox{ \ \ for \ }
j=1, \dots, k-1
,
	\\
\frac{U_{j-1}}{U_j}
& \approx
		\lambda_{j}^{\frac{n-2}{2}}
		\langle y_j \rangle ^{n-2}
		\1_{\{ |x|\le \mu_{0,j-1}\}}
		+
\lambda_{j}^{-\frac{n-2}{2}}
\1_{\{ |x| > \mu_{0,j-1}\}}
\mbox{ \ \ for \ }
j=2, \dots, k
.\label{Uj-ratio-2}
	\end{align}
\end{lemma}
\begin{proof}
\eqref{A-1}-\eqref{A-4} follow from that $\frac{U_{j+1}}{U_j}$ is strictly decreasing about $|x|$ and $\frac{U_{j+1}}{U_j}(\bar{\mu}_{j+1})=1$, (see Figure \ref{fig:3-bubble}). Up to a multiplicity of the constant $\alpha_n$, $U_j=\mu_{j}^{\frac{2-n}{2}}(1+|y_j|^2)^{\frac{2-n}{2}}$ and
	\begin{align}\label{Uj+1}
		U_{j+1}=\frac{\mu_{j+1}^{\frac{n-2}{2}}}{(\mu_{j+1}^2+|x|^2)^{\frac{n-2}{2}}}=\frac{\mu_{j+1}^{\frac{n-2}{2}}\mu_j^{2-n}}{(\lambda_{j+1}^2+|y_j|^2)^{\frac{n-2}{2}}},
	\end{align}
	then
	\begin{equation}
	\frac{U_{j+1}}{U_j}
 =\lambda_{j+1}^{\frac{n-2}{2}}\frac{(1+|y_j|^2)^{\frac{n-2}{2}}}{(\lambda_{j+1}^2+|y_j|^2)^{\frac{n-2}{2}}}
 \approx
\lambda_{j+1}^{-\frac{n-2}{2}}
\langle y_{j+1} \rangle^{-(n-2)}
\1_{\{ |x|\le \mu_{0j}\}}
+
\lambda_{j+1}^{\frac{n-2}{2}}
\1_{\{ |x| > \mu_{0j}\}} ,
	\end{equation}
for $j=1, \dots, k-1$.
	This finishes the proof of \eqref{Uj-ratio-1}.
Similarly,
	\begin{align}
		\frac{U_{j-1}}{U_j}
	=
	\lambda_j^{-\frac{n-2}{2}}
	\frac{(\lambda_j^2 + |y_{j-1}|^2)^{\frac{n-2}{2}}}
	{(1+|y_{j-1}|^2)^{\frac{n-2}{2}}}
		\approx
		\lambda_{j}^{\frac{n-2}{2}}
		\langle y_j \rangle ^{n-2}
		\1_{\{ |x|\le \mu_{0,j-1}\}}
		+
\lambda_{j}^{-\frac{n-2}{2}}
\1_{\{ |x| > \mu_{0,j-1}\}}
	\end{align}
for $j=2, \dots, k$.
Then \eqref{Uj-ratio-2} holds.
\end{proof}
\begin{figure}
	\centering
		\includegraphics[width=0.5\linewidth]{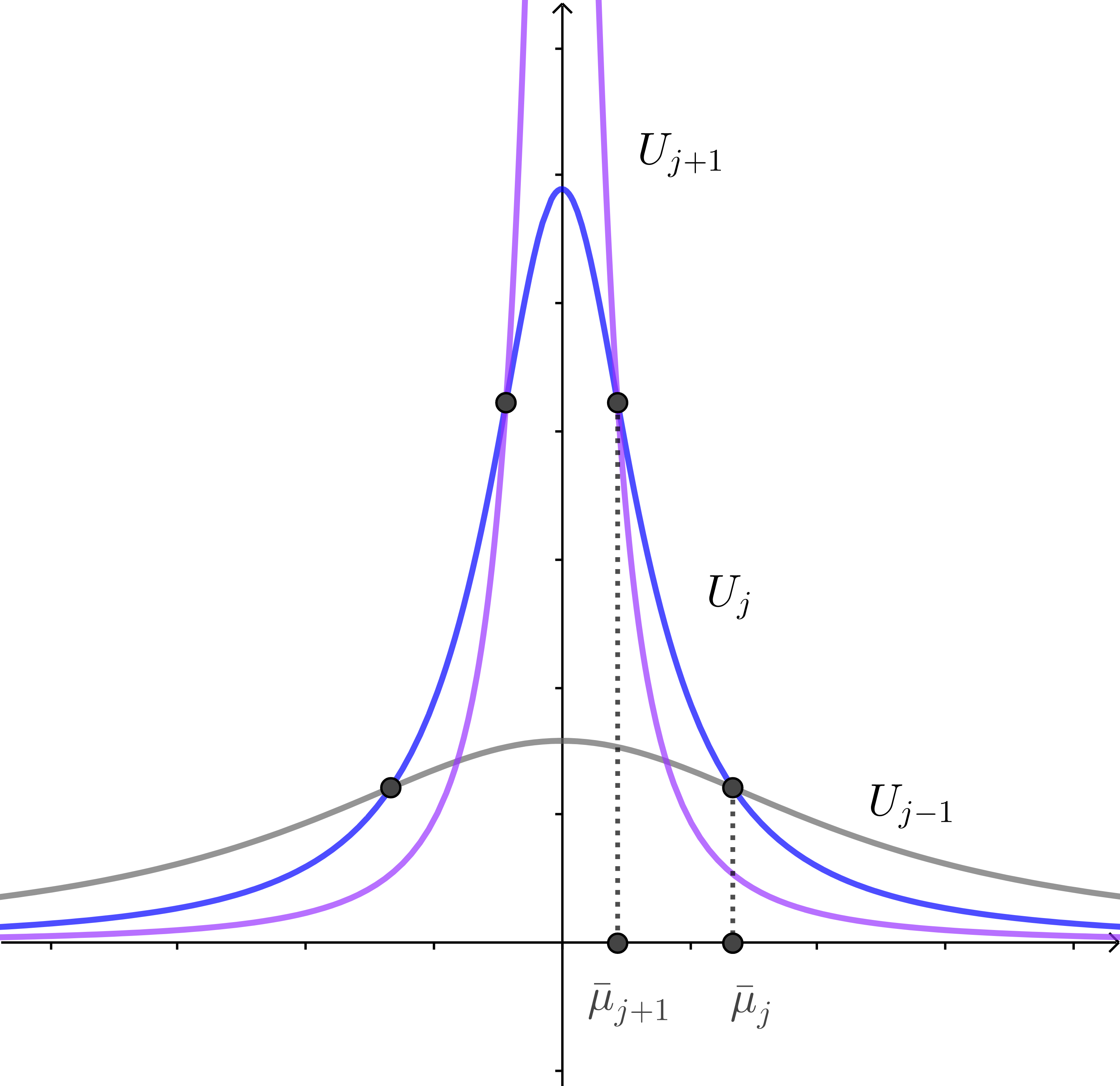}
	\caption{Relation for three bubbles.}
	\label{fig:3-bubble}
\end{figure}
\begin{lemma}\label{lem:phi0}
	Consider $\varphi_0$ defined in \eqref{W-2.6}. One has $|\varphi_0|\lesssim \sum\limits_{i=2}^k\lambda_i U_i\chi_i$.
\end{lemma}
\begin{proof}
By \eqref{W-2.6} and \eqref{W-2.20}, we have
	\begin{align}\label{phi0-f-est}
		|\varphi_0|\lesssim\sum_{i=2}^k\mu_{i-1}^{-\frac{n-2}{2}}\langle y_i\rangle^{-2}\chi_i.
	\end{align}
	It follows from \eqref{spt-chij} that the support of $\chi_i$ are disjoint. More precisely, the support of $\chi_i$ is contained in $ \{\lambda_{i+1}^{\frac12}\leq|y_i|\leq \lambda_i^{-\frac12}\}$. It is easy to verify that $\mu_{i-1}^{-\frac{n-2}{2}}\langle y_i\rangle^{-2}\lesssim \lambda_i U_i$ in this set.
\end{proof}

\begin{lemma}\label{lem:outer-1}For $0<\alpha<a<1$, there exists $R$ large enough and $t_0$ negative enough such that $B[\vec{\phi}]$ defined in \eqref{def:Bphi} satisfies
	\begin{align}
	\begin{split}\label{Bphi-main}
	\|B[\vec{\phi}\,]\|^{out}_{\alpha,\sigma}\lesssim
	R^{\alpha-a}
	\|\vec{\phi}\|^{in,*}_{a,\sigma}
\end{split}
	\end{align}
\end{lemma}
\begin{proof}
By the definition \eqref{inner-symbol} and \eqref{def:phi^*_nu},
	\begin{equation}\label{phi-est}
		\langle y_j\rangle |\nabla_{y_j}\phi_j(y_j,t)|+|\phi_j(y_j,t)|\lesssim |t|^{\gamma_j}\mu_{0j}^{\frac{n-2}{2}}R^{n+1-a}\langle y_j\rangle^{-n-1}\|\phi_j\|^{in,*}_{j,a,\sigma}.
	\end{equation}
	
	$\bullet$ First, using \eqref{W-3.4.5}
	\begin{align}\label{Bphi-1}
		\begin{split}
			|\dot \mu_j\frac{\partial\varphi_j}{\partial\mu_j}\eta_j|= \ &\left| \dot \mu_j\mu_j^{-\frac{n}{2}}\left(\frac{n-2}{2} \phi_j(y_j,t)+y_j\cdot \nabla_{y_j} \phi_j(y_j,t)\right)\eta_{j}\right|\\
			\lesssim\ & | \dot \mu_j| \mu_j^{-1}(-t)^{\gamma_j}R^{n+1-a}\langle y_j\rangle^{-n-1}\eta_j\|\phi_j\|^{in,*}_{j,a,\sigma}
			\\
	\lesssim \ &
	|t_0|^{-\epsilon}
	R^{n+1-a}
	w_{1j}\|\phi_j\|^{in,*}_{j,a,\sigma}.
		\end{split}
	\end{align}
Here we choose $\epsilon$ small such that $|\dot{\mu}_j  \mu_j| \lesssim |t|^{-\epsilon}$ for $j=1, \dots, k$. We have used \eqref{def:w1j} in the last step.
	
	$\bullet$ Second, \eqref{phi-est} implies that
	\begin{align}\label{varphi-est}
	    |\varphi_j(x,t)|\lesssim |t|^{\gamma_j}R^{-a}\|\phi_j\|_{j,a,\sigma}^{in,*}
	    \mbox{ \ \ for \ }
	    2R\le |y_j| \le 4R
	    .
	\end{align}
	Using \eqref{def:eta}, we obtain
	\begin{align}\label{Bphi-2}
	    |\Delta\eta_j \varphi_j|\lesssim  (R\mu_j)^{-2}|t|^{\gamma_j}R^{-a}\|\phi_j\|^{in,*}_{j,a,\sigma}\mathbf{1}_{\{2R\leq |y_j|\leq 4R\}}\lesssim R^{\alpha-a}w_{1j}\|\phi_j\|^{in,*}_{j,a,\sigma}.
	\end{align}
	Similarly, we have	\begin{align}
		\begin{split}\label{Bphi-2-3}
		|\nabla_x\eta_j\cdot\nabla_x\varphi_j|\lesssim&\  R^{-1}\mu_j^{-1}|t|^{\gamma_j}\mu_j^{-1}R^{-1-a}\|\phi_j\|_{j,a,\sigma}^{in,*}\mathbf{1}_{\{2R\mu_j\leq |x|\leq 4R\mu_j\}}
			\lesssim R^{\alpha-a} w_{1j}\|\phi_j\|^{in,*}_{j,a,\sigma}
		\end{split}
	\end{align}
	and
	\begin{align}
	\begin{split}\label{Bphi-2-2}
	    |\partial_t\eta_j\varphi_j|
	    &
	    \lesssim R^{-1}\mu_j^{-2}
	    |\dot \mu_j | |t|^{\gamma_j}R^{-a}\|\phi_j\|^{in,*}_{j,a,\sigma}\mathbf{1}_{\{2R\leq |y_j|\leq 4R\}}
	    \lesssim R^{1+\alpha-a}
	    |\dot \mu_j | w_{1j}\|\phi_j\|^{in,*}_{j,a,\sigma}
	\\
	&
	\lesssim
	R^{1+\alpha-a}
	|t_0|^{-\epsilon}
	w_{1j}\|\phi_j\|^{in,*}_{j,a,\sigma} ,
	\end{split}
	\end{align}
	when $|\dot{\mu}_j|\lesssim |t|^{-\epsilon}$ for $j=1, \dots, k$.

	
	$\bullet$ Third, to estimate $\left|p(u_*^{p-1}-U_j^{p-1})\varphi_j\eta_j\right|$. We only give calculation details for $j=2, \dots , k-1$ since the case $j=1$ and $j=k$ can be dealt with similarly.
	Consider it in $\{\bar\mu_{0,j+1}\leq |x|\leq 4R\mu_{0j}\}$.
	By Lemma \ref{lem:Uij} and Lemma \ref{lem:phi0},
\begin{equation*}
U_{j}^{-1}\left( \sum\limits_{i\ne j} U_{i} + \varphi_{0} \right) \ge 2^{-1}U_{j}^{-1}\left( \sum\limits_{i\ne j} U_{i} \right) -\lambda_{j}^{\frac 12} > -\frac 12,
\end{equation*}
when $t_0$ is very negative. Consequently $u_*>\frac12U_j$. Thus by the mean value theorem and \eqref{Uj-ratio-1} and \eqref{Uj-ratio-2},
\begin{align}
		\begin{split}\label{u*-U}
		&	|u_*^{p-1}-U_j^{p-1}|
			\lesssim U_j^{p-1}\left(\frac{U_{j+1}}{U_j}+\frac{U_{j-1}}{U_j}
			+
	\lambda_j \right)
	\\
	\lesssim \ &
	\mu_j^{-2}
\langle y_j\rangle^{-4}
	\left(
	\lambda_{j+1}^{-\frac{n-2}{2}}
	\langle  y_{j+1} \rangle^{2-n}
	\mathbf{1}_{\{\bar\mu_{0,j+1}\leq|x|\leq \mu_{0j}\}}
	+
	\lambda_{j+1}^{\frac{n-2}{2}}
	+\lambda_j^{\frac{n-2}{2}}\langle y_j\rangle^{n-2}+\lambda_j\right).
		\end{split}
	\end{align}
	Therefore, using $|\varphi_j|\lesssim|t|^{\gamma_j}R^{n+1-a}
	\|\phi_j\|_{j,a,\sigma}^{in,*}$,
\begin{equation}\label{Bphi-3}
		\begin{aligned}
		&
		\left|p(u_*^{p-1}-U_j^{p-1})\varphi_j\eta_j\right|
		\1_{\{\bar\mu_{0,j+1}\leq |x|\leq 4R\mu_j\}}
		\\
		\lesssim & \
		R^{n+1-a} \mu_j^{-2}|t|^{\gamma_j}
		\lambda_{j+1}^{-\frac{n-2}{2}}
		\langle  y_{j+1}\rangle^{2-n}\mathbf{1}_{\{\bar\mu_{0,j+1}\leq|x|\leq \mu_{0j}
		\}}
		\|\phi_j\|^{in,*}_{j,a,\sigma}\\
			&\ +R^{n+1-a}(\lambda_{j+1}^{\frac{n-2}{2}}+\lambda_j^{\frac{n-2}{2}}
			R^{n-2}+\lambda_j)w_{1j}\|\phi_j\|^{in,*}_{j,a,\sigma}\\
			 \lesssim &\
			 R^{n+1-a}|t_0|^{-\epsilon}(w_{2j}+w_{1j})\|\phi_j\|^{in,*}_{j,a,\sigma}.
		\end{aligned}
\end{equation}
		Here we have used the following fact.		\begin{equation*}
\begin{aligned}
&
R^{n+1-a} \mu_j^{-2}|t|^{\gamma_j}
		\lambda_{j+1}^{-\frac{n-2}{2}}
		\langle  y_{j+1}\rangle^{2-n}
		\mathbf{1}_{\{\bar\mu_{0,j+1}\leq|x|\leq \mu_{0j}
		\}}
\\
\approx \ &
R^{n+1-a} \lambda_{j+1} \lambda_{j}^{\frac n2 -1} |t|^{\sigma}
|t|^{-2\sigma} \mu_{j+1}^{\frac n2 -2} \mu_{j}^{-1} |x|^{2-n}
\mathbf{1}_{\{\bar\mu_{0,j+1}\leq|x|\leq \mu_{0j}
		\}}
\\
\lesssim \ &
R^{n+1-a}
|t_0|^{-\epsilon } w_{2j} .
\end{aligned}
	\end{equation*}

In  $\{\bar\mu_{0,m+1}\leq|x|\leq \bar\mu_{0m}\}$, $m=j+1,\cdots,k$, one has $|u_*^{p-1}-U_j^{p-1}|\lesssim U_m^{p-1}
	\approx \mu_m^{-2}
	\langle y_m \rangle^{-4}
	$
	and $|\varphi_j|\lesssim(-t)^{\gamma_j}R^{n+1-a}\|\phi_j\|_{j,a,\sigma}^{in,*}$.
	Then it is easy to see
	\begin{align*}
		\left|p(u_*^{p-1}-U_j^{p-1})\varphi_j\eta_j\right|\lesssim
		R^{n+1-a}
		|t|^{\gamma_j-\gamma_m} w_{1m}\|\phi_j\|^{in,*}_{j,a,\sigma}\lesssim
		R^{n+1-a}|t_0|^{-\epsilon}w_{1m}\|\phi_j\|^{in,*}_{j,a,\sigma}.
	\end{align*}
Taking $t_0$ very negative such that $|t_0|^{-\epsilon}<R^{\alpha-n-1}$, we obtain \eqref{Bphi-main}.
\end{proof}

Recall $E^{out}$ defined in \eqref{def:Eout}. We reorganize the terms as the following.
\begin{align}
	E^{out}=\bar E_{11}+\bar E_2+\bar E_3+\bar E_4+\bar E_5
\end{align}
where $\bar E_{11}$ is defined in \eqref{W-2.10}, and
\begin{align}
	\bar E_2=&\mu_1^{-\frac{n+2}{2}}D_1[\vec{\mu}_1](1-\eta_1)+\sum_{j=2}^k\mu_j^{-\frac{n+2}{2}}D_j[\vec{\mu}_1]
	(\chi_j-\eta_j )+
	\sum_{j=2}^k
	\mu_j^{-\frac{n+2}{2}}\Theta_j[\vec{\mu}_1]\chi_j
	\\
	\bar E_3=&\sum_{j=2}^kp(\bar U^{p-1}-U_j^{p-1})\varphi_{0j}\chi_j\\
	\bar E_4=
	&\sum_{j=2}^{k}\left(2 \nabla_{x} \varphi_{0 j}
	\cdot \nabla_{x}\chi_{j}+
	\Delta_{x}\chi_{j} \varphi_{0 j}\right)-\sum_{j=2}^{k} \partial_{t}\left(\varphi_{0 j} \chi_{j} \right)\\
	\bar E_5=&N_{\bar U}[\varphi_0].
\end{align}

\begin{lemma}\label{lem:outer-2}
	 There exist $\sigma,\epsilon>0$ small and $t_0$ negative enough, such that
	\[E^{out}\lesssim
	R^{-1}
	\left(\sum_{j=1}^kw_{1j}+\sum_{j=1}^{k-1} w_{2j}+w_3\right)
	.
	\]
\end{lemma}
\begin{proof}
	$\bullet$ \emph{Estimate of $\bar E_2$.} Consider the first term in $\bar E_2$. The support of $1-\eta_1$ is $\{|y_{01}|\geq 2R\}$.
	Since we assume $\|\vec{\mu}\|_\sigma<1$, one has $|\dot\mu_{11}|\leq \|\vec{\mu}_1\|_{\sigma} |t|^{-1-\sigma} \le |t|^{-1-\sigma}$. Then using \eqref{W-2.24}
	\begin{equation*}
|\mu_{1}^{-\frac{n-2}{2}} D_{1}[\vec{\mu}_1] (1-\eta_1)| \lesssim
|t|^{-1-\sigma} |x|^{2-n} \1_{\{ |x|\ge 2R \}}
\lesssim R^{4+\alpha-n}w_{11} + R^{-1} w_{3}
\lesssim  R^{-1} \left(w_{11} +w_{3} \right)
.
\end{equation*}

For $j\geq 2$, the support of $\chi_j-\eta_j$ is contained in $\{ |x|\le \bar{\mu}_{0,j+1}\}\cup \{
2R \mu_{0j} \le |x|\le \bar{\mu}_{0j}
\}$. In the first set (it is vacuum if $j=k$),  one has $|\chi_j-\eta_j|\leq \chi(2 \bar y_{0,j+1})$  and $\langle y_{0j}\rangle\approx 1$. It follows from \eqref{Djmu-1} that
	\begin{equation*}
		\left|\mu_j^{-\frac{n+2}{2}}D_j[\vec{\mu}_1](\chi_j-\eta_j)\right|\lesssim \mu_j^{-2}
		|t|^{\gamma_{j}}
		\chi(2 \bar y_{0,j+1})
		\lesssim
|t|^{-\epsilon}	w_{1,j+1}.
	\end{equation*}
In the second set, since $|y_{0j}|\ge 2R$, then
	\begin{equation*}
		\left|\mu_j^{-\frac{n+2}{2}}D_j[\vec{\mu}_1](\chi_j-\eta_j)\right|
		\lesssim
		\mu_{j}^{-2} |t|^{\gamma_{j}}
	|y_j|^{-4}	
	\lesssim R^{-1 } w_{1j}.
	\end{equation*}
It is straightforward to have
	\begin{equation*}
|\mu_j^{-\frac{n+2}{2}}\Theta_j[\vec{\mu}_1]\chi_j|
\lesssim
|t|^{-\sigma} \mu_{j}^{-2} |t|^{\gamma_j} \langle y_j \rangle^{-4} \chi_j
\lesssim
|t|^{-\sigma} w_{1j} .
	\end{equation*}
	
	$\bullet$ \emph{Estimate of $\bar E_3$.}
	In the support of $\chi_j$, by Lemma \eqref{lem:Uij}, we have
	\[|\bar U^{p-1}-U_j^{p-1}|\lesssim U_j^{p-1}(\frac{U_{j+1}}{U_j}+\frac{U_{j-1}}{U_j}).
	\]
	It follows from \eqref{W-2.7} and \eqref{W-2.20} that $\varphi_{0j}\leq |t|^{\gamma_j+\sigma}\langle y_i\rangle^{-2}$.
		Using \eqref{u*-U}, similar to \eqref{Bphi-3}, we get
		\[|\bar E_3|\lesssim \sum\limits_{j=2}^{k}
	\lambda_j^{\frac{n-2}{2}}|t|^{\sigma}w_{1j}
	+
	\sum\limits_{j=1}^{k-1}
	\lambda_{j+1}
	\lambda_j^{\frac{n-2}{2}}|t|^{2\sigma}w_{2j}\lesssim  |t_0|^{-\epsilon}(\sum\limits_{j=2}^{k} w_{1j}+\sum\limits_{j=1}^{k-1}w_{2j}).
		\]
	
	$\bullet$ \emph{Estimate of $\bar E_4$.} Notice
	\begin{align*}
		\left|\nabla_{x} \chi_{j}\right| \lesssim  \bar{\mu}_{0j}^{-1} \mathbf{1}_{\left\{\frac12\bar{\mu}_{0 j}<\left|x\right|< \bar{\mu}_{0 j}\right\}}+ \bar{\mu}_{0,j+1}^{-1} \mathbf{1}_{\left\{\frac 12 \bar{\mu}_{0,j+1}<\left|x\right|< \bar{\mu}_{0,j+1}\right\}}.
	\end{align*}
	\begin{align*}
		\left|\Delta_{x} \chi_{j}\right| \lesssim  \bar{\mu}_{0j}^{-2} \mathbf{1}_{\left\{\frac12\bar{\mu}_{0 j}<\left|x\right|< \bar{\mu}_{0 j}\right\}}+ \bar{\mu}_{0,j+1}^{-2} \mathbf{1}_{\left\{\frac 12 \bar{\mu}_{0,j+1}<\left|x\right|< \bar{\mu}_{0,j+1}\right\}}.
	\end{align*}
	By \eqref{W-2.20}, one has $ |\varphi_{0j}|\lesssim
	|t|^{\gamma_j+\sigma}\langle y_j\rangle^{-2}$, $ |\nabla_x\varphi_{0j}|\lesssim \lambda_{0j}^{\frac{n-2}{2}}\mu_j^{-\frac{n}{2}}\langle y_j\rangle^{-3}
	\approx
	\mu_j^{-1}|t|^{\gamma_j+\sigma}\langle y_j\rangle^{-3}$,
	then
	\begin{align*}
	    |\nabla_x\varphi_{0j}\cdot \nabla_x \chi_j|\lesssim |t|^\sigma \frac{\mu_j}{\bar \mu_j}w_{1j}
	    +|t|^{\gamma_j-\gamma_{j+1}+\sigma}
	   \left(
	   \frac{\mu_{j+1}}{\mu_{j}}
	  \right)^{\frac{1-\alpha}{2}} w_{1,j+1}\lesssim |t_0|^{-\epsilon}(w_{1j}+w_{1,j+1}).
	\end{align*}
	\begin{align*}
	    |\varphi_{0j}  \Delta_x \chi_j|\lesssim |t|^\sigma \frac{\mu_j}{\bar \mu_j}w_{1j}
	    +
	    |t|^{
	  \frac{n-2}{2}(\alpha_{j-1} -\alpha_{j}) +
	  \frac{\alpha}{2}(\alpha_{j+1} - \alpha_{j}) + \sigma
	    }
	  w_{1,j+1}\lesssim |t_0|^{-\epsilon}(w_{1j}+w_{1,j+1}).
	\end{align*}
where we have used that $
	  \frac{n-2}{2}(\alpha_{j-1} -\alpha_{j}) +
	  \frac{\alpha}{2}(\alpha_{j+1} - \alpha_{j}) = (\frac{\alpha}{n-6}-1) \left( \frac{n-2}{n-6}\right)^{j-1} \le \frac{\alpha}{n-6}-1$.

 For $ \partial_t(\varphi_{0j}\chi_j)$, we have
\begin{equation*}
		\begin{aligned}
		&	|\partial_t(\varphi_{0j})\chi_j|+|\varphi_{0j}\partial_t
		\chi_j
		|
			\lesssim  |t|^{\gamma_j+\sigma-1}\langle y_j\rangle^{-2}\left(
		\chi_j +
		| \nabla_{x} \chi_j |
		\right)
	\lesssim  |t|^{\sigma-1}
	\mu_{j}^{2-\frac{\alpha}{2}} \mu_{j-1}^{\frac{\alpha}{2}}
	w_{1j}
\lesssim
|t_0|^{-\epsilon} 	w_{1j} .
		\end{aligned}
\end{equation*}

	$\bullet$ \emph{Estimate of $\bar E_5$.}
	It follows from \eqref{N-phi0}, \eqref{phi0-f-est} and $p\in(1,2)$ that
	\[
	\begin{aligned}
	|N_{\bar U}[\varphi_0]|
 \lesssim |\varphi_0|^p
	\lesssim \sum_{j=2}^k
	|t|^{(\gamma_j+\sigma)p}\langle y_j\rangle^{-2p}
	\chi_j^p
\approx
\sum_{j=2}^k
|t|^{2(\alpha_{j-1} - \alpha_j ) +\sigma}
\langle y_j \rangle^{2+\alpha -2p}
\mu_{j}^{-2}|t|^{\gamma_j} \langle y_j \rangle^{-2-\alpha}
\chi_j^p
.
	\end{aligned}
	\]
	
If $2p \ge 2+\alpha$, it is easy to see $|N_{\bar U}[\varphi_0]|
	\lesssim |t_0|^{-\epsilon} \sum_{j=2}^kw_{1j}.$

If $2p < 2+\alpha$,
	\[
	|N_{\bar U}[\varphi_0]|
	 \lesssim
\sum_{j=2}^k
|t|^{\frac{2-\alpha+2p}{2}(\alpha_{j-1} - \alpha_j ) +\sigma}
\mu_{j}^{-2} |t|^{\gamma_j} \langle y_j \rangle^{-2-\alpha}
\chi_j^p
	\lesssim |t_0|^{-\epsilon} \sum_{j=2}^kw_{1j}
.
	\]
	$\bullet$ \emph{Estimate of $\bar E_{11}$.}
Regrouping the terms in \eqref{W-2.10}, one obtain
\be
\begin{split}
	\bar{E}_{11} =&\sum_{j=2}^{k}p U_j^{p-1}\left[\sum_{l\neq j}U_l -U_{j-1}(0)\right] \chi_{j}+\sum_{j=2}^{k}\left[{\bar{U}}^p-\sum_{i=1}^k U_i^p-p U_j^{p-1}\sum_{l\neq j}U_l\right] \chi_{j}\\
	&+
	\left(-\sum_{j=2}^{k}(1-\chi_{j}) \partial_{t} U_{j}
	\right)
	+\left[{\bar{U}}^p-\sum_{j=1}^{k} U_j^p\right]\left(1-\sum_{i=2}^{k} \chi_{i}\right)\\
	:=&J_1+J_2+J_3+J_4.
\end{split}
\ee
Claim:
\[\bar E_{11}(x,t)\lesssim |t_0|^{-\epsilon}\left(\sum_{j=1}^k w_{1j}+\sum_{j=1}^{k-1}w_{2j}+w_3\right).
\]

\begin{enumerate}
\item
 \emph{Estimate of $J_1$.}
	\[J_1=\sum_{j=2}^{k}pU_j^{p-1}\left(\sum_{l\neq j,j-1}U_l\right) \chi_{j}+\sum_{j=2}^kpU_j^{p-1}(U_{j-1}-U_{j-1}(0))\chi_j.\]
	We will bound each term in the above equation. Fix $j\geq 2$. If $i\leq j-2$,
	\[|pU_j^{p-1}U_i\chi_j|\lesssim \mu_j^{-2}\langle y_j \rangle^{-4}\mu_{j-2}^{-\frac{n-2}{2}}\chi_j\lesssim
	|t|^{-\epsilon}w_{1j}.\]
	If $i\ge j+1$, by Lemma \ref{lem:Uij}
	\begin{align}
		\begin{split}
			&|p U_j^{p-1}U_i\chi_j|\lesssim U_{j}^p\frac{U_{j+1}}{U_j}\chi_j\\
			\lesssim&\     \mu_{j}^{-\frac{n+2}{2}}\langle y_j\rangle^{-n-2}
			\left(
			\lambda_{j+1}^{-\frac{n-2}{2}}
			\langle y_{j+1}\rangle^{2-n}\mathbf{1}_{\{\bar\mu_{0,j+1}\leq |x|\leq \mu_{0j} \}}+\lambda_{j+1}^{\frac{n-2}{2}} \right)\chi_j\\
			\lesssim&\  \mu_{j+1}^{\frac{n-2}{2}}\mu_{j}^{-2}
	|x|^{2-n}\mathbf{1}_{\{\bar\mu_{0,j+1}\leq|x|\leq \mu_{0j}
			\}}+(\frac{\lambda_{j+1}}{\lambda_j})^{\frac{n-2}{2}} |t|^{\sigma}w_{1j}\\
			\lesssim&\  |t|^{-\epsilon}(w_{2j}+w_{1j})
		\end{split}
	\end{align}
when we choose $\sigma$ small first and then chose $\epsilon$ small enough.
	Using $|U_{j-1}-U_{j-1}(0)|\chi_j\lesssim \mu_{j-1}^{-\frac{n-2}{2}}\lambda_j$, we have
	\begin{align}
		|p U_j^{p-1}(U_{j-1}-U_{j-1}(0))\chi_j|\lesssim \mu_{j}^{-2}\langle y_j\rangle^{-4}\mu_{j-1}^{-\frac{n-2}{2}}\lambda_j
		\chi_j
		\lesssim
		|t|^{-\epsilon}w_{1j}.
	\end{align}
	
\item
\emph{Estimate of $J_2$.}
By Lemma \eqref{lem:Uij}, we have
	\[
	\left|{\bar{U}}^p-\sum_{i=1}^k U_{i}^p-p U_j^{p-1}\sum_{l\neq j}U_l\right|
	\chi_j
	\lesssim
	\left(U_{j-1}^p+U_{j+1}^p\right)\chi_j.\]
Therein,
	\[
	\begin{aligned}
	U_{j-1}^p\chi_j
	& \approx
|t|^{\frac{n+2}{2}\alpha_{j-1}} \chi_j
 \lesssim
|t|^{\frac{n+2}{2}\alpha_{j-1}}
\mu_{j}^2 |t|^{-\gamma_j}
(\frac{\bar{\mu}_j}{\mu_j})^{2+\alpha}
w_{1j}
\chi_j
\\
&	\approx (\frac{\mu_j}{\mu_{j-1}})^{\frac{2-\alpha}{2}} |t|^{\sigma}
w_{1j}
\chi_j
\lesssim |t|^{-\epsilon}w_{1j},
\end{aligned}
	\]
and
\[U_{j+1}^{p}\chi_j
\approx
\mu_{j+1}^{\frac{n+2}{2}}|x|^{-2-n}\chi_j\lesssim |t|^{2\sigma}\mu_{j+1}^3\mu_j|x|^{-4}w_{2j} \chi_{j}\lesssim |t|^{-\epsilon}w_{2j},
\]
when we take $\sigma$ small first and then take $\epsilon$ small enough.
	
\item
 \emph{Estimate of $J_3$.}
For $j = 2, \dots,  k$, notice that
	\begin{align}\label{ptU}
		|\partial_t U_j|=|\dot\mu_j\mu_j^{-\frac{n}{2}}Z_{n+1}(y_j)|
	\lesssim \mu_j^{-2}
	\mu_{j-1}^{-\frac{n-2}{2}}
	\langle y_j\rangle^{2-n}.
	\end{align}
	The support of $1-\chi_j$ is contained in $\{|x|\leq \bar\mu_{0,j+1}\}\cup
	\{\frac{1}{2}\bar\mu_{0j}\leq|x| < \bar\mu_{0j}\}
	\cup\{\bar\mu_{0j}\leq |x|\}$. In the first set, it is easy to see $1-\chi_j= \chi(2|x|/\bar\mu_{0,j+1})$, then
	\begin{align*}
		| (1-\chi_j)\partial_t U_j|\lesssim
		(\frac{\mu_{j+1}}{\mu_j})^{\frac{2-\alpha}{2}}
		(\frac{\mu_j}{\mu_{j-1}})^{\frac{n-2}{2}}
		|t|^{\sigma}
		w_{1,j+1}
		\chi(2|x|/\bar{\mu}_{0,j+1})
		\lesssim
		|t|^{-\epsilon}w_{1,j+1} .
	\end{align*}
In the second set,
	\begin{equation*}
| \partial_t U_j|\1_{\{\frac{1}{2}\bar\mu_{0j}\leq|x| < \bar\mu_{0j}\}}
\lesssim
\left(
\frac{\mu_{j}}{\mu_{j-1}}
\right)^{\frac{n-4-
\alpha}{2}}
|t|^{\sigma}
w_{1j}
\1_{\{\frac{1}{2}\bar\mu_{0j}\leq|x| < \bar\mu_{0j}\}}
\lesssim |t_0|^{-\epsilon}w_{1j} .
	\end{equation*}
	In the third set, we  split it further to be $\{\bar\mu_{0 j}\leq |x|\}=\cup_{m=2}^j\{\bar\mu_{0m}\leq |x|\leq \bar \mu_{0,m-1}\}\cup\{\bar\mu_{01}\leq |x|\}$.

	Since $|y_j|$ is very large in the third set, \eqref{ptU} implies $|\partial_tU_j|\lesssim$ $\mu_j^{n-4}\mu_{j-1}^{-\frac{n-2}{2}}|x|^{2-n}$. Note that $\mu_j^{n-4}\mu_{j-1}^{-\frac{n-2}{2}}$ decreases about $j$ up to some constant multiplicity. Then in $\{\bar\mu_{0 m}\leq|x|\leq \bar\mu_{0,m-1}\}$, $ m = 2\dots, j$,
	$$|\partial_t U_j|\lesssim\mu_m^{n-4}\mu_{m-1}^{-\frac{n-2}{2}} |x|^{2-n}
	\lesssim |t|^{-\epsilon}w_{2,m-1}.$$
	In $\{\bar\mu_{01}\leq |x|\}$, we have
	\[|\partial_t U_j|\lesssim \mu_2^{n-4} |x|^{2-n}
	\lesssim |t|^{-\epsilon}w_{3}
	.\]
	
\item
 \emph{Estimate of $J_4$.} Recall the definition of $\chi_i$ in \eqref{spt-chij}, we have the support of $J_4$ is contained in the set $\cup_{m=3}^{k}\{\frac12\bar\mu_{0m}\leq |x|\leq \bar\mu_{0m}\} \cup \{\frac 12 \bar{\mu}_{02} \le |x| \}$.
	
In  $\{\frac12\bar\mu_{0m}\leq |x|\leq \bar\mu_{0m}\}$, for $m=3, \dots, k$, by Lemma \ref{lem:Uij}, one has $U_m\approx U_{m-1}\approx \mu_{m-1}^{-\frac{n-2}{2}} \gg U_i$ for $i\neq m, m-1$.  Therefore
	\[|J_4|
	\1_{\{\frac12\bar\mu_{0m}\leq |x|\leq \bar\mu_{0m}\}}
	\lesssim \mu_{m-1}^{-\frac{n+2}{2}}
	\1_{\{\frac12\bar\mu_{0m}\leq |x|\leq \bar\mu_{0m}\}}
	\lesssim
	(\frac{\mu_{m-1}}{\mu_m})^{\frac{\alpha-2}{2}}
	|t|^{\sigma}
	w_{1m}
	\lesssim |t_0|^{-\epsilon}w_{1m}.\]
In  $\{ \frac 12 \bar\mu_{02}\leq|x|\}$, by Lemma \ref{lem:Uij}, when $\sigma <(n-6)^{-1}$,
	\begin{equation*}
\begin{aligned}
&  |J_4|  \lesssim  U_1^{p-1}U_2
\\
 \approx \ & \mu_2^\frac{n-2}{2}
\left(
|x|^{2-n}
\1_{\{ \frac 12 \bar\mu_{02}\leq |x| \le  \bar\mu_{02} \}}
+
|x|^{2-n}
\1_{\{ \bar\mu_{02}\leq |x| \le 1 \}}
+
|x|^{-2-n} \1_{\{1 \le |x| \le \bar{\mu}_{01}\} } + |x|^{-2-n} \1_{ \{ \bar{\mu}_{01} \le |x|  \}}
\right)
\\
\lesssim \ &
|t_{0}|^{-\epsilon}
\left(w_{12} +
w_{21} + w_{11} +w_{3}
\right) .
\end{aligned}
	\end{equation*}

	\end{enumerate}
\end{proof}

\begin{lemma}\label{lem:w1j*-cmp}
There exist $\sigma>0$ small enough and $t_0$ negative enough, such that for $t<t_0$,
\begin{enumerate}
    \item In $ \{
|x|\le \bar{\mu}_{0i} \}
$, $ i = 1,\dots, k$,
we have $w_{1j}^* \lesssim w_{1i}^*$ for $j = 1, \dots ,  i-1$ (it is vacuum if $i=1$)
\item In $\{ \bar{\mu}_{0i} \le |x| \}$, $ i = 1,\dots, k$, we have $w_{1j}^* \lesssim w_{1i}^*$ for $j =  i+1, \dots , k$ (it is vacuum if $i=k$).
\item
In $\{|x|\le \bar{\mu}_{0j} \}$, $w_3^* \lesssim w_{1j}^*$ for $j = 1, \dots,  k$. In $\{ |x|\ge |t|^{\frac 12} \}$, $w_3^* \gtrsim w_{1j}^*$ for $j = 1, \dots,  k$ when $\delta \le (n-2-\alpha)^{-1}$.
\end{enumerate}
Consequently,
\begin{align}\label{eq:w1j*-cmp}
\sum_{j=1}^kw_{1j}^*+w_{3}^*\lesssim
\begin{cases}
w_{1k}^{*}
&
\text{if }
|x| \le \bar{\mu}_{0k},
\\
 w_{1i}^*+w_{1,i+1}^*&
 \text{if }\bar\mu_{0,i+1}\leq|x|\leq \bar\mu_{0i}, i=1,\cdots,k-1,\\
w_{11}^{*} + w_{3}^*&\text{if }\bar \mu_{01}\leq |x| \le |t|^{\frac 12} ,
\\
w_{3}^*&\text{if } |x| \ge |t|^{\frac 12} .
\end{cases}
\end{align}
\end{lemma}	
\begin{proof} (1) For $j = 1, \dots, i-1$, in $\{|x|\leq \bar\mu_{0i} \}$, we have $w^*_{1j}(x,t)= |t|^{\gamma_j}\leq |t|^{\gamma_{i-1}}$ and  $w^*_{1i}\geq |t|^{\gamma_i}\mu_i^\alpha\bar\mu_i^{-\alpha}$. It is easy to verify $|t|^{\gamma_{i-1}}\lesssim |t|^{\gamma_i}\mu_i^\alpha\bar\mu_i^{-\alpha}$ if $\alpha<n-6$.

(2) For $j =  i+1, \dots, k$, in $\{ \bar\mu_{0i}\leq |x|\leq |t|^{\frac12} \}$, $w_{1j}^*(x,t)= |t|^{\gamma_j}\mu_j^\alpha\bar\mu_j^{n-2-\alpha}|x|^{2-n}\approx
|t|^{\gamma_j^*}|x|^{2-n}\leq |t|^{\gamma_i^*}|x|^{2-n}\approx w_{1i}^*(x,t)$, because $\gamma_j^*$ is strictly decreasing on $j$, i.e.
\[\gamma_{1}^*
>
\gamma_{2}^*
> \cdots
> \gamma_k^*.\]
In $\{|x|\ge |t|^{\frac12} \}$, we have $w_{1j}^* \lesssim w_{1i}^*$ by the same reason.

(3) Due to $(1)$, we only need to check $w_3^* \lesssim w_{11}^*$ in $\{|x|\le \bar{\mu}_{01} \}$. It is straightforward to have $R|t|^{-1-\sigma +\delta(4-n)} \lesssim |t|^{-1-\sigma}\langle x \rangle^{-\alpha}$ in $\{|x|\le \bar{\mu}_{01} \}$. Due to $(2)$, in $\{|x|\ge |t|^{\frac 12} \}$, we only need to check $w_3^* \gtrsim w_{11}^*$, which is easy to get when $\delta \le (n-2-\alpha)^{-1}$.
\end{proof}

\begin{lemma}\label{lem:w2j*-cmp}
 There exists $t_0$ negative enough such that
\begin{enumerate}
    \item In $\{\bar\mu_{0,i+1}\leq |x| \}$, we have $w^*_{2j}\lesssim w^*_{2i}$ for $j=i+1, \dots , k-1$ (it is vacuum if $i=k, k-1$). In $\{|x|\leq \bar\mu_{0i} \}$, we have $w_{2,i-1}^* \gtrsim w_{2j}^*$ for $j= 1, \dots, i-2$ (it is vacuum if $i =1,2 $).
    \item
    In $\{|x|\geq \bar\mu_{01} \}$, $w_{2j}^*\lesssim w_3^*$ for $ j = 1, \dots,  k-1$.
\end{enumerate}
Consequently
\begin{align}\label{eq:w2j*-cmp}
    \sum_{j=1}^{k-1}w_{2j}^*\lesssim \begin{cases}
    w_{2,k-1}^{*}
    &\text{if }
    |x|\le \bar{\mu}_{0k},
    \\
    w_{2i}^*+w_{2,i-1}^*
    &\text{if }\bar\mu_{0,i+1} \leq |x|\leq \bar\mu_{0i},
    \mbox{ \ for \ }
    i=2, \dots, k-1,
    \\
    w_{21}^{*}
    &\text{if }\bar\mu_{02}\leq |x|\leq \bar\mu_{01},
    \\
    w_3^*&\text{if }|x|\geq \bar\mu_{01}.
\end{cases}
\end{align}
\end{lemma}
\begin{proof}
(1)
In $\{\bar\mu_{0,i+1} \leq |x|\leq \bar\mu_{0i} \}$, it follows from \eqref{def:w2j*} that $w_{2j}^*= |t|^{-2\sigma}\mu_{0,j+1}^{\frac{n}{2}-2}\mu_{0,j-1}|x|^{2-n}$ for $j=i+1, \dots, k-1$ and $w_{2i}^*=|t|^{-2\sigma}\mu_{0,i+1}^{\frac{n}{2}-2}\mu_{0i}^{-1}|x|^{4-n}$.
\begin{equation*}
\frac{w_{2i}^{*}}{w_{2j}^{*}} \gtrsim
(\frac{\mu_{0,i+1}}{\mu_{0,i+2}})^{\frac n2 -2} \frac{\mu_{0,i+1}}{\mu_{0i}} \approx
|t|^{(\frac n2 -2)(\alpha_{i+2}- \alpha_{i+1}) - (\alpha_{i+1} -\alpha_{i})} \gtrsim 1
\end{equation*}
where we have used that  $(\frac n2 -2)(\alpha_{i+2}- \alpha_{i+1}) - (\alpha_{i+1} -\alpha_{i})= \frac{(n-4)^2 +4}{(n-6)^2} (\frac{n-2}{n-6})^{i-1}$.
Thus $w_{2j}^*\lesssim w_{2i}^*$.  It is easy to see that $w_{2j}^*\lesssim w_{2i}^*$ also holds in $\{|x|>\bar\mu_{0i} \}$. This deduces the first part. The second parts holds obviously by \eqref{def:w2j*}.

    (2) In $\{\bar\mu_{01}\leq |x| \le |t|^{\frac 12}\}$, by the definition \eqref{def:w2j*} and \eqref{eq:w3j*}, it is easy to see that $w_{2j}^{*}\lesssim w_{3}^{*}$ for $j=1,\dots, k-1$ since  $w_{2j^*}$ have more time decay than $w_3^*$. In $\{|x| \ge |t|^{\frac 12} \}$, we have $w_{2j}^{*}\lesssim w_{3}^{*}$ by the similar reason.
\end{proof}
\begin{remark}
Lemma \ref{lem:w1j*-cmp} and \ref{lem:w2j*-cmp} help us consider much less terms in the topology of the outer problem in some special domains.
\end{remark}

\begin{lemma}\label{lem:outer-4} There exist $\sigma, \epsilon>0$ small and $t_0$ negative enough such that
\begin{equation}
\|\TT^{out}[V\Psi]\|_{\alpha,\sigma}^{out,*}
\lesssim
R^{-1}
\|\Psi\|_{\alpha,\sigma}^{out,*}.
\end{equation}
\end{lemma}
\begin{proof}
	Without loss of generality, we assume $\|\Psi\|_{\alpha,\sigma}^{out,*}=1$.
By \eqref{Def:NV}, we rewrite $V$ as
	\begin{align}\label{V-decomp}
		V=pu_*^{p-1}(1-\sum_{j=1}^k\zeta_j)+\sum_{j=1}^k\zeta_jp(u_*^{p-1}-U_j^{p-1}) .
	\end{align}
We shall handle terms respectively.
	
Consider the first term in \eqref{V-decomp}. Using \eqref{def:zeta}, the support of $1-\sum_{j=1}^k\zeta_j$ is  $\cup_{i=2}^{k}\{R\mu_{0i}\leq |x|\leq 2R^{-1}\mu_{0,i-1}\} \cup \{R\mu_{01}\leq |x|\}$.

$\bullet$ In $\{R\mu_{01}\leq |x|\}$, we have $p u_*^{p-1}\lesssim \mu_1^2|x|^{-4}\leq R^{-1}|x|^{-3}$ by Lemma \ref{lem:Uij} and \ref{lem:phi0}. Split the region into $\{R\mu_{01}\leq |x|\leq \bar\mu_{01}\}\cup \{\bar\mu_{01}\leq |x| \le |t|^{\frac 12}\} \cup \{|x|\ge |t|^{\frac 12}\}$.
In the first set, one has $|\Psi|\lesssim w_{11}^*+w_{12}^*+w_{21}^*$ by \eqref{eq:w1j*-cmp} and  \eqref{eq:w2j*-cmp}. Notice $w_{12}^* + w_{21}^* \lesssim w_{11}^*$ in $\{R\mu_{01}\le |x| \le \bar\mu_{01}\}$. Therefore
\begin{equation}\label{V-b1}
		\big|
		p u_{*}^{p-1} (1-\sum_{j=1}^{k} \zeta_{j})
	\Psi \big|
	\lesssim
R^{-1}  |x|^{-3}
w_{11}^*
\lesssim
R^{-1}
w_{11}
.
\end{equation}
In the second set, by \eqref{eq:w1j*-cmp} and \eqref{eq:w2j*-cmp}, one has $|\Psi|\lesssim w_{11}^{*} + w_3^* $. Then
\begin{equation}\label{V-b2}
		\big|
		pu_{*}^{p-1} (1-\sum_{j=1}^{k} \zeta_{j})
		\Psi
		\big|
		\lesssim \
		R^{-1}  |x|^{-3}
	(w_{11}^{*} + w_3^* )
		\lesssim \
	R^{-1}
	(w_{11} + w_3 )
	.
\end{equation}
In the third set,
similarly we have
\begin{equation}
		\big|
		pu_{*}^{p-1} (1-\sum_{j=1}^{k} \zeta_{j})
		\Psi
		\big|
		\lesssim
		R^{-1}  |x|^{-3}
	 w_3^*
		\lesssim \
		R^{-1}  w_3
		.
\end{equation}

$\bullet $ Consider the region $\{R\mu_{0i}\le |x| \le 2R^{-1}\mu_{0,i-1}\}$, $i = 2, \dots,  k$. We divide it further into two parts, $\{R\mu_{0i}\leq |x|\leq \bar\mu_{0i}\}\cup \{\bar\mu_{0i}\leq |x|\leq 2R^{-1}\mu_{0,i-1}\}$.
In $\{R\mu_{0i}\le |x| \le \bar{\mu}_{0i}\}$, we have $p u_*^{p-1}\lesssim \mu_i^{2}|x|^{-4}$ by Lemma \ref{lem:Uij} and \ref{lem:phi0}. Moreover, one has $|\Psi|\leq w_{1i}^*+w_{1,i+1}^*+w_{2,i-1}^*+w_{2i}^*$ by Lemma \ref{lem:w1j*-cmp} and \ref{lem:w2j*-cmp}. One readily has $w_{1,i+1}^* \lesssim w_{1i}^*  $ in $\{ R\mu_{0i}\le |x| \le \bar{\mu}_{0i} \}$.
Thus
\begin{equation}
	\begin{aligned}
		\big| p u_{*}^{p-1} (1-\sum_{j=1}^{k} \zeta_{j})
		\Psi
		\big|
\lesssim \ &
\mu_{i}^{2}
|x|^{-4 }
\left(
 w_{1i}^* + w_{1,i+1}^*
+
 w_{2,i-1}^*
 +
 w_{2i}^*
\right)
\\
\lesssim \ &
\mu_{i}^{2}
|x|^{-4 }
\left(
w_{1i}^*
+
w_{2,i-1}^*
+
w_{2i}^*
\right)
\\
\lesssim \ &
R^{-1}
\left(
w_{1i}
+
w_{2i}
\right)
,
	\end{aligned}
\end{equation}
where we have used the fact that in $\{R\mu_{0i}\leq |x|\leq \bar\mu_{0i} \}$,
\begin{align*}
&\mu_i^{2}|x|^{-4}w_{1i}^*\lesssim |t|^{\gamma_i}\mu_i^{2+\alpha}|x|^{-4-\alpha}
\lesssim
R^{-2}w_{1i},
\\
&\mu_i^{2}|x|^{-4}w_{2,i-1}^*
\lesssim
\mu_i^{2}|x|^{-4}|t|^{-2\sigma}\mu_{i-1}^{1-\frac{n}{2}}
\lesssim
R^{\alpha -2} |t|^{-\sigma} w_{1i}, \\
&\mu_i^{2}|x|^{-4}w_{2,i}^*
\lesssim
\mu_i^{2}|x|^{-4}|t|^{-2\sigma}\mu_{i+1}^{\frac{n}{2}-2}\mu_{i}^{-1}
|x|^{4-n}
\lesssim
R^{-2}w_{2i}.
\end{align*}

In the other part $\{\bar{\mu}_{0i} \le |x| \le 2R^{-1}\mu_{0,i-1}\}$, we have $p u_*^{p-1}\lesssim U_{i-1}^{p-1}\lesssim \mu_{i-1}^{-2} $ and $w_{2,i-2}^*\lesssim w_{1,i-1}^*$ (which is vacuum if $i=2$). Then
\begin{equation}\label{Vpsi:1-zeta}
\begin{aligned}
		\big| p u_*^{p-1} (1-\sum_{j=1}^{k} \zeta_{j})
		\Psi \big|
		\lesssim \ &
\mu_{i-1}^{-2}
\left(
 w_{1,i-1}^*
 +
 w_{1i}^*
+
 w_{2,i-2}^*
+
 w_{2,i-1}^*
\right)
\\
\lesssim \ &
\mu_{i-1}^{-2}
\left(
w_{1,i-1}^*
+
w_{1i}^*
+
w_{2,i-1}^*
\right)
\\
\lesssim \ &
\left(
\mu_{i-1}^{-2} |t|^{\gamma_{i-1}}
\1_{\{  |x| \le 2R^{-1}\mu_{i-1} \}}
+ R^{-2}w_{2,i-1}
\right)
,
\end{aligned}
\end{equation}
where we have used the fact that in $\{\bar\mu_{0i}\leq |x|\leq 2R^{-1}\mu_{0,i-1}\}$,
\begin{align*}
    &\mu_{i-1}^{-2}w_{2,i-1}^*\lesssim \mu_{i-1}^{-2}|t|^{-2\sigma}\mu_i^{\frac{n}{2}-2}\mu_{i-1}^{-1}|x|^{4-n}\lesssim R^{-2} |t|^{-2\sigma}\mu_i^{\frac{n}{2}-2}\mu_{i-1}^{-1}|x|^{2-n}\lesssim R^{-2}w_{2,i-1}
    ,
    \\   &\mu_{i-1}^{-2}w_{1i}^*\leq \mu_{i-1}^{-2}|t|^{\gamma_i}\mu_i^{\alpha}\bar\mu_i^{n-2-\alpha}|x|^{2-n}
    \lesssim |t|^{\gamma_i+2\sigma}\mu_i^{\frac{\alpha}{2}+1}\mu_{i-1}^{\frac{n-4-\alpha}{2}}w_{2,i-1}\lesssim |t|^{-\epsilon}w_{2,i-1} .
\end{align*}
By Lemma \ref{lem:annulus},
\begin{equation}\label{no-space-decay}
	\begin{aligned}
		&\TT^{out}
		[
		\mu_{i-1}^{-2} |t|^{\gamma_{i-1}
		}
		\1_{\{|x|\le 2R^{-1}\mu_{0,i-1}\}}
		]\\
		\lesssim \ &
\begin{cases}
 |t|^{\gamma_{i-1}} R^{-2}
&
\mbox{ \ \ if \ } |x|\le R^{-1} \mu_{0,i-1},
\\
\mu_{i-1}^{-2} |t|^{\gamma_{i-1}} (R^{-1} \mu_{i-1})^n
|x|^{2-n}
&
\mbox{ \ \ if \ }  R^{-1} \mu_{0,i-1} \le |x| \le |t|^{\frac 12},
\\
R^{-n}
(|x|^2)^{(2-n)\alpha_{i-1} +\gamma_{i-1}} |x|^{2-n}
&
\mbox{ \ \ if \ }   |x| \ge |t|^{\frac 12} .
\end{cases}\\
\lesssim \ &
R^{-1} w_{1,i-1}^*.
	\end{aligned}
\end{equation}

Next we consider the second term in \eqref{V-decomp}. Recall the support of $\zeta_j$ \eqref{spt:zeta} is  contained in $\{ R^{-1} \mu_{0j} \le |x| \le 2R \mu_{0j}\}$, which are mutually disjoint.

$\bullet$ For $j=1$, we have $|\zeta_1(u_*^{p-1}-U_1^{p-1})|\lesssim U_2^{p-1}\zeta_{1}
\lesssim \mu_2^{2}|x|^{-4}\1_{\{ R^{-1}\mu_{01} \le |x| \le 2 R\mu_{01}\}}$ since $\varphi_{0} =0$ in the support of $\zeta_{1}$.
Then
\begin{equation}\label{j=1case}
	\left|	 \zeta_{1}
	(
	u_*^{p-1}
	-
	U_j^{p-1}
	) \Psi  \right|
\lesssim
\mu_2^2 |x|^{-4}
\1_{\{ R^{-1}\mu_{01} \le |x| \le 2 R\mu_{01}
\}}
\left(
 w_{11}^* + w_{12}^*
+
 w_{21}^*
\right)
\lesssim
|t|^{-\epsilon}
w_{11} ,
\end{equation}
where we have used the fact that in $\{R^{-1}\mu_{01}\leq |x|\leq 2R\mu_{01}\}$,
\begin{align*}
    &\mu_2^2|x|^{-4}w_{11}^*\lesssim \mu_2^2|x|^{-4}|t|^{\gamma_1}(1+|x|)^{-\alpha}\leq R^{4}\mu_2^2 w_{11}\leq R^4|t|^{-2\epsilon} w_{11},\\
    &\mu_2^2|x|^{-4}w_{21}^*\lesssim \mu_2^2|x|^{-4} |t|^{-2\sigma}\mu_2^{\frac{n}{2}-2}|x|^{2-n}\lesssim R^{n+2}|t|^{-2\epsilon}w_{11},\\
    &\mu_2^2|x|^{-4}w_{12}^*\lesssim \mu_2^2|x|^{-4}|t|^{\gamma_2}\mu_2^\alpha\bar\mu_2^{n-2-\alpha}|x|^{2-n}\lesssim R^{n+2}|t|^{-2\epsilon}w_{11}.
\end{align*}

$\bullet$ For $ j =2,\dots, k$, we have  $|\zeta_j(u_*^{p-1}-U_j^{p-1})|\lesssim (U_{j-1}^{p-1}+U_{j+1}^{p-1} + \varphi_{0j}^{p-1})|\zeta_j|\lesssim \mu_{j-1}^{-2} |\zeta_j|$ where we have used the fact that $\varphi_{0}= \varphi_{0j} \chi_{j} $ in $\{ R^{-1} \mu_{0j} \le |x| \le 2R \mu_{0j}\}$ and $U_{j+1}$ is vacuum if $j=k$. Then
\begin{equation}
		\left|	\zeta_{j}
		(
		u_{*}^{p-1}
		-
		U_j^{p-1}
		) \Psi \right|
\lesssim
\mu_{j-1}^{-2}
\left(
 w_{1j}^*
 +
 w_{1,j+1}^*
+
 w_{2,j-1}^*
+
w_{2j}^*
\right)
|\zeta_j|
\lesssim
|t|^{-\epsilon}
\left( w_{1j}
+
 w_{2j}
\right)
,
\end{equation}
where we have used the fact that in $\{R^{-1}\mu_{0j}\leq |x|\leq 2R\mu_{0j}\}$, $w_{2,j-1}^* \lesssim w_{1j}^*$ and
\begin{align*}
    &\mu_{j-1}^{-2}w_{1j}^*\lesssim \mu_{j-1}^{-2}|t|^{\gamma_j}\lesssim \lambda_j^{2}\langle y_j\rangle^{2+\alpha}w_{1j}\lesssim  R^{2+\alpha}|t|^{-2 \epsilon}w_{1j},\\
    &\mu_{j-1}^{-2}w_{1,j+1}^*\lesssim \mu_{j-1}^{-2}|t|^{\gamma_{j+1}}\mu_{j+1}^\alpha\bar\mu_{j+1}^{n-2-\alpha}|x|^{2-n}\lesssim \mu_{j-1}^{-2}\mu_j^{1-\frac{\alpha}{2}}\mu_{j+1}^{1+\frac{\alpha}{2}}|t|^{\sigma}w_{2j}\lesssim |t|^{-\epsilon}w_{2j},\\
    &\mu_{j-1}^{-2}w_{2j}^*\lesssim \mu_{j-1}^{-2}|t|^{-2\sigma}\mu_{j+1}^{\frac{n}{2}-2}\mu_j^{-1}|x|^{4-n}\lesssim \mu_{j-1}^{-2}(R\mu_i)^2w_{2j}\lesssim R^2|t|^{-2\epsilon}w_{2j}.
\end{align*}
Combining the above calculations of the two terms in \eqref{V-decomp}, we get the conclusion.
\end{proof}

\begin{lemma}\label{lem:outer-5} There exist $\sigma,\epsilon>0$ small and $t_0$ negative enough such that
	\[\|\mathcal{N}[\vec{\phi}, \Psi , \vec{\mu}_{1}]\|^{out}_{\alpha,\sigma}\lesssim |t_0|^{-\epsilon}\left(\|\vec{\phi}\|^{in,*}_{a,\sigma}+\|\Psi\|^{out,*}_{\alpha,\sigma}\right)^p.\]
\end{lemma}
\begin{proof}
	By \eqref{Def:NV} and some elementary inequality
	\begin{align}\label{N-decomp}
		|\mathcal{N}[\vec{\phi},\Psi,\vec{\mu}_{1}]|\lesssim \sum_{j=1}^k \mu_j^{-\frac{n+2}{2}}|\phi_j|^p\eta_j+|\Psi|^p.
	\end{align}
For the first part on the RHS, recalling \eqref{phi-est}, we obtain
\begin{equation}\label{N1}
\begin{aligned}
\mu_j^{-\frac{n+2}{2}}|\phi_j|^p\eta_j
\lesssim \ &
|t|^{\gamma_jp}
\frac{R^{(n+1-a)p}}{\langle y_j\rangle ^{(n+1)p}}
\1_{\{|x|\le 4R\mu_{0j} \}}\left(
\|\phi_j\|^{in,*}_{j,a,\sigma}
\right)^p
\lesssim
R^{(n+1-a)p}|t|^{-2\epsilon}
w_{1j}
\left(
\|\phi_j\|^{in,*}_{j,a,\sigma}
\right)^p
.
\end{aligned}
\end{equation}
For the second part on the RHS of \eqref{N-decomp},

$\bullet$ In $\{|x|\geq \bar{\mu}_{01}  \}$, by \eqref{eq:w1j*-cmp}, we have $|\Psi|\lesssim (w_{11}^{*} + w_3^{*}) \|\Psi\|^{out,*}_{\alpha,\sigma}$ in $ \{ \bar{\mu}_{01} \le |x| \le |t|^{\frac 12}  \}$ and $|\Psi|\lesssim  w_3^{*} \|\Psi\|^{out,*}_{\alpha,\sigma}$ in $ \{ |x| > |t|^{\frac 12}\}$.
Notice
\begin{equation*}
\begin{aligned}
(w_3^*)^p = \ & R^{p-1}|t|^{1+(1-p)\sigma} |x|^{-4} w_3\lesssim |t|^{-1}w_3
\mbox{ \ \ if \ }
|x| > |t|^{\frac12} ,
\\
(w_3^*)^p
= \ &  R^{p-1} |t|^{-(1+\sigma)(p-1)}
|x|^{-\frac{2n-12}{n-2}} w_{3}
\lesssim |t|^{-\epsilon} w_{3}
\mbox{ \ \ if \ }\bar{\mu}_{01} \le |x| \le |t|^{\frac 12} ,
\\
	(w_{11}^*)^p
= \ & R^{-1} |t|^{-\frac{4}{n-2}(1+\sigma) +\delta(n-2-p\alpha)}  w_{3}
\lesssim |t|^{-\epsilon} w_{3}
\mbox{ \ \ if \ }
\bar{\mu}_{01} \le |x| \le |t|^{\frac 12},
\end{aligned}
\end{equation*}
where we take $\delta \le 2(n-2)^{-2}$ in the last formula.
Thus
\begin{align}\label{Psip-b2}
	|\Psi|^p
	\lesssim \
 |t|^{-\epsilon}
w_3
\left( \| \Psi\|^{out,*}_{\alpha,\sigma}
\right)^p .
\end{align}

$\bullet$ In $\{1\leq |x|\leq \bar\mu_{01}\}$, we have $|\Psi|\lesssim (w_{11}^*+w_{12}^*+w_{21}^*)\|\Psi\|^{out,*}_{\alpha,\sigma}\lesssim w_{11}^*\|\Psi\|^{out,*}_{\alpha,\sigma}$ since  $
\left(
w_{12}^*
+
w_{21}^*
\right)
\1_{\{1\le |x|\le \bar{\mu}_{01}\}} \lesssim w_{11}^*$. Therefore
\begin{equation}\label{Psip-b3}
		|\Psi|^p
\lesssim \
\left(
w_{11}^*\right) ^p
\left( \| \Psi\|^{out,*}_{\alpha,\sigma}
\right)^p
\lesssim
|t|^{-(1+\sigma)\frac{4}{n-2} +2\delta}
w_{11}
\left( \| \Psi\|^{out,*}_{\alpha,\sigma}
\right)^p
\lesssim
|t|^{-\epsilon}
w_{11}
\left( \| \Psi\|^{out,*}_{\alpha,\sigma}
\right)^p
 ,
\end{equation}
for $\delta\le (n-2)^{-1}$.

$\bullet $ In $\{\bar \mu_{02}\leq |x|\leq 1\}$, we have
\begin{equation*}
\begin{aligned}
    &(w_{11}^*)^p\lesssim |t|^{-(1+\sigma)p}\lesssim |t|^{-\epsilon}w_{11},
    \\
    &(w_{12}^*)^p\lesssim |t|^{\gamma_2p}\bar\mu_2^{n+2}|x|^{-n-2}\lesssim |t|^{\gamma_2p}\bar\mu_2^{n-2}|x|^{2-n}\lesssim |t|^{-\epsilon}w_{21},
    \\
    &(w_{21}^*)^p\lesssim |t|^{-2\sigma p}\mu_2^{(\frac{n}{2}-2)p}|x|^{(4-n)p}\lesssim |t|^{-2\sigma p}\mu_2^{\frac{n}{2}-1}|x|^{2-n}\lesssim |t|^{-\epsilon} w_{21},
\end{aligned}
\end{equation*}
where we have used $n-2\le (n-4)p$ when $n\ge 6$ in the last inequality. Then
\begin{align}\label{Psip-b4}
    |\Psi|^p\lesssim [(w_{11}^*)^p + (w_{12}^*)^p+(w_{21}^*)^p](\|\Psi\|^{out,*}_{\alpha,\sigma})^p\lesssim |t|^{-\epsilon}(w_{11}+w_{21})(\|\Psi\|^{out,*}_{\alpha,\sigma})^p.
\end{align}

$\bullet$ In $\{\bar\mu_{0,i+1}\leq |x|\leq \bar\mu_{0i}\}$, $i=2,\dots, k$, we have
\begin{align*}
	(w_{1i}^*)^p\lesssim \ & |t|^{\gamma_ip}\langle y_i\rangle^{-\alpha p}
	\lesssim
	\mu_{i}^{2}
	|t|^{\gamma_{i}(p-1)} \langle y_i \rangle^{2+\alpha -\alpha p} w_{1i}
\lesssim
|t|^{-\sigma(p-1)} \mu_{i}\mu_{i-1}^{-1} w_{1i}
	\lesssim |t|^{-\epsilon}w_{1i}
\end{align*}
provided $\epsilon<\frac{2}{n-6}$.
\begin{align*}
    (w_{1,i+1}^*)^p\lesssim &\  |t|^{p \gamma_{i+1}}\bar\mu_{i+1}^{n+2}|x|^{-n-2}\lesssim \mu_{i}^{-\frac{n+2}{2}}|t|^{-\sigma p}\bar \mu_{i+1}^{n-2}|x|^{2-n}\\
    \approx&\  \lambda_{i+1}\mu_{i+1}^{\frac{n}{2}-2}\mu^{-1}_i|t|^{-\sigma p}|x|^{2-n}\lesssim |t|^{-\epsilon}w_{2i}
\end{align*}
provided $\epsilon<-2\sigma+\sigma p+\frac{2}{n-6}$.
\begin{align}
    (w_{2,i-1}^*)^p\lesssim&\  |t|^{-2\sigma p}\mu_{i-1}^{-\frac{n+2}{2}}\lesssim |t|^{-2\sigma p}\mu_{i-1}^{-\frac{n+2}{2}}(\frac{\bar\mu_i}{\mu_i})^{2+\alpha}\langle y_i\rangle^{-2-\alpha}\\
    \lesssim&\ |t|^{-2\sigma p}(\frac{\mu_i}{\mu_{i-1}})^{1-\frac{\alpha}{2}}\mu_{i}^{-2}\mu_{i-1}^{1-\frac{n}{2}}\langle y_i\rangle^{-2-\alpha}\lesssim|t|^{-\epsilon}w_{1i}
\end{align}
provided $\epsilon<\frac{1}{n-6}$.
\begin{align}
    (w_{2i}^*)^p\lesssim &\  |t|^{-2\sigma p}\mu_{i+1}^{(\frac{n}{2}-2)p}\mu_i^{-p}|x|^{(4-n)p}\lesssim |t|^{-2\sigma p}\mu_{i+1}^{(\frac{n}{2}-2)p}\mu_i^{-p}|x|^{2-n}\bar\mu_{i+1}^{(4-n)p-(2-n)}\\
    \approx & \
    \lambda_{j+1} |t|^{-2\sigma p}\mu_{i+1}^{\frac{n}{2}-2}\mu_i^{-1}|x|^{2-n}\lesssim |t|^{-\epsilon}w_{2i}
\end{align}
provide $\epsilon<\frac{2}{n-6}$. Here we have used $(4-n)p-(2-n)\le 0$ when $n\ge 6$.

Therefore, for $ i =2,\dots, k$,
	\begin{align*}
		|\Psi|^p
		\lesssim &
		\left(
		 w_{1i}^*
		 +
		  w_{1,i+1}^*
		+
		 w_{2,i-1}^*
		 +
		 w_{2i}^*
		\right)^p
		\left( \| \Psi\|^{out,*}_{\alpha,\sigma}
		\right)^p
\lesssim
|t|^{-\epsilon} \left(w_{1i}
+ w_{2i}
\right)
\left( \| \Psi\|^{out,*}_{\alpha,\sigma}
\right)^p,
	\end{align*}
where $w_{1,i+1}^*,
w_{2i}^*
$ are vacuum if $i=k$ and $
w_{1,i+1}^*$ are vacuum if $i=k,k-1$.
\end{proof}

\begin{proof}[Proof of Proposition \ref{prop:G-bd}]
This is a combination of results in Lemma \ref{lem:outer-1}, \ref{lem:outer-2}, \ref{lem:outer-4}, \ref{lem:outer-5} and Lemma \ref{lem:w1j-w1j*}, \ref{lem:w2-w2*} \ref{lem:w3-w3*}.
\end{proof}


\section{Some estimates for the outer problem}
\subsection{Basic estimates}
Let $G(x,t)$ denote the standard heat kernel on $\Rn$, that is
\begin{align}\label{def:G}
	G(x,t)=\frac{1}{(4\pi t)^{n/2}}e^{-\frac{|x|^2}{4t}} .
\end{align}
Recall the $\TT^{out}[g](x,t)$ defined by \eqref{duhamel}.
\begin{lemma}\label{lem:int-G-est}
	Suppose $n >2$, $ a\ge 0$, $d_1\leq d_2\leq\frac12$ and $b$ satisfies
\begin{equation}\label{b-require}
\begin{cases}
\frac{n}{2}-b+d_2(a-n)>1
&
\mbox{ \ \ if  \ } a<n,
\\
\frac{n}{2}-b+d_1(a-n)>1
&
\mbox{ \ \ if  \ } a\ge n,
\end{cases}
\end{equation}	
$ 0\leq c_1,c_2\leq c_{**}$. Then there exists $C$ depending on $n,a,b,d_1,d_2,c_{**}$ such that for $t<-1$
	\[\TT^{out}\left[
\frac{	|t|^{b} }{|x|^{a} }
	\mathbf{1}_{\{c_1|t|^{d_1}\leq |x|\leq c_2|t|^{d_2}\}}\right](0,t)\leq C\begin{cases}|t|^{b}(c_1|t|^{d_1})^{2-a}&\text{if  }a\in(2,\infty),
	\\|t|^{b}\ln (c_2|t|^{d_2}/(c_1|t|^{d_1})) &\text{if }a=2,
	\\|t|^{b}(c_2|t|^{d_2})^{2-a}&\text{if }a\in[0,2).\end{cases}\]

\end{lemma}
\begin{proof}
	Using \eqref{def:G}, we obtain
	\begin{align}
		&\int_{-\infty}^t\int_{\Rn}\frac{1}{(t-s)^{n/2}}e^{-\frac{|y|^2}{4(t-s)}}\frac{|s|^{b}}{|y|^a}\mathbf{1}_{\{c_1|s|^{d_1}\leq|y|\leq c_2|s|^{d_2}\}}dyds\notag\\
\approx &\int_{-\infty}^t\int_{c_1|s|^{d_1}}^{c_2|s|^{d_2}}\frac{|s|^{b}}{(t-s)^{n/2}}e^{-\frac{r^2}{4(t-s)}}r^{n-1-a}drds
		\approx \int_{-\infty}^t\int_{\frac{c_1^2|s|^{2d_1}}{4(t-s)}}^{\frac{c_2^2|s|^{2d_2}}{4(t-s)}}e^{-z}z^{\frac{n-a-2}{2}}\frac{|s|^{b}}{(t-s)^{a/2}}dzds\notag\\
		=&\int_{-\infty}^t\frac{|s|^{b}}{(t-s)^{a/2}}F\left(\frac{c_1^2|s|^{2d_1}}{4(t-s)},\frac{c_2^2|s|^{2d_2}}{4(t-s)}\right)ds\label{intsF},
	\end{align}
	where
	\[F(A,B):=\int_A^Be^{-z}z^{\frac{n-a-2}{2}}dz.
	\]
	We shall split \eqref{intsF} into four integrals $J_1,J_2,J_3,J_4$ according to the regions of $s$. First, in the region $s\in [t-c_1^2|t|^{2d_1},t]$, one has $|s-t|\leq c_1^2 |t|$ when $t<-1$. Therefore
	\begin{equation*}
\begin{aligned}
		J_1=&\int_{t-c_1^2|t|^{2d_1}}^t\frac{|s|^{b}}{(t-s)^{a/2}}F\left(\frac{c_1^2|s|^{2d_1}}{4(t-s)},\frac{c_2^2|s|^{2d_2}}{4(t-s)}\right)ds
		\\
		\lesssim &\ |t|^{b}\int_{t-c_1^2|t|^{2d_1}}^t\frac{1}{(t-s)^{a/2}}
		e^{-\frac{c_1^2|t|^{2d_1}}{4(t-s)} (1+c_{**}^2)^{2(d_1)^{-}} }ds
		\\
		 \approx&\  c_1^{2-a}|t|^{b+d_1(2-a)}\int_{\frac{(1+c_{**}^2)^{2(d_1)^{-}}}{4}}^{\infty}e^{-\tilde s}\tilde s^{\frac a2 -2}d\tilde s\lesssim c_1^{2-a}|t|^{b+d_1(2-a)},
\end{aligned}
	\end{equation*}
where $(d_1)^{-}=\min\{0,d_1\}$.

	Second, in the region $s\in [t-c_2^2|t|^{2d_2},t-c_1^2|t|^{2d_1}]$,
\begin{equation}\label{J2a}
	\begin{aligned}
		J_2=&\int_{t-c_2^2|t|^{2d_2}}^{t-c_1^2|t|^{2d_1}}\frac{|s|^{b}}{(t-s)^{a/2}}F\left(\frac{c_1^2|s|^{2d_1}}{4(t-s)},\frac{c_2^2|s|^{2d_2}}{4(t-s)}\right)ds
		\\
\lesssim
&
\int_{t-c_2^2|t|^{2d_2}}^{t-c_1^2|t|^{2d_1}}\frac{|s|^{b}}{(t-s)^{a/2}}
F\left(\frac{c_1^2|t|^{2d_1}(1+c_{**}^2)^{2(d_1)^{-}}}{4(t-s)},\frac{c_2^2|t|^{2d_2}(1+c_{**}^2)^{2(d_{2})^{+}}}{4(t-s)}\right)ds
\\
		\lesssim&\int_{t-c_2^2|t|^{2d_2}}^{t-c_1^2|t|^{2d_1}}\frac{|t|^{b}}{(t-s)^{a/2}}
\begin{cases}
1
&
\mbox{ \ \ if \ } a<n,
\\
\left|
\ln \left(
\frac{c_1^2|t|^{2d_1}(1+c_{**}^2)^{2(d_1)^{-}}}{4(t-s)}
\right)\right|
&
\mbox{ \ \ if \ } a=n,
\\
\left(
\frac{c_1^2 |t|^{2d_1}}{t-s}
\right)^{\frac{n-a}{2}}
&
\mbox{ \ \ if \ } a>n,
\end{cases}
		ds\\
		\lesssim & \begin{cases}c_1^{2-a}|t|^{b+d_1(2-a)}&\text{if  }2<a,\\
			|t|^{b}\log (c_2/c_1|t|^{2(d_2-d_1)}) &\text{if }a=2,\\c_2^{2-a}|t|^{b+d_2(2-a)}&\text{if }0\leq a<2,
		\end{cases}
	\end{aligned}
\end{equation}
where $(d_2)^{+} = \max\{0,d_2\}$.

	Third, when $s\in[2t-c_2^2|t|^{2d_2},t-c_2^2|t|^{2d_2}]$,
	\begin{equation}
	\begin{aligned}
			J_3=&\int_{2t-c_2^2|t|^{2d_2}}^{t-c_2^2|t|^{2d_2}}\frac{|s|^{b}}{(t-s)^{a/2}}F\left(\frac{c_1^2|s|^{2d_1}}{4(t-s)},\frac{c_2^2|s|^{2d_2}}{4(t-s)}\right)ds
			\\
\lesssim
&\int_{2t-c_2^2|t|^{2d_2}}^{t-c_2^2|t|^{2d_2}}\frac{|t|^{b}}{(t-s)^{a/2}}F\left(\frac{c_1^2|t|^{2d_1}(2+c_{**}^2)^{2(d_2)^{-}}}{4(t-s)},\frac{c_2^2|t|^{2d_2} (2+c_{**}^2)^{2(d_2)^{+}}}{4(t-s)}\right)ds
\\
			\lesssim& \int_{2t-c_2^2|t|^{2d_2}}^{t-c_2^2|t|^{2d_2}}\frac{|t|^{b}}{(t-s)^{a/2}}
\begin{cases}
			\left(\frac{c_2^2|t|^{2d_2}}{t-s}\right)^{\frac{n-a}{2}}
&
\mbox{ \ \ if \ } a<n,
\\
\ln \left(\frac{c_2^2|t|^{2d_2} (2+c_{**}^2)^{2(d_2)^{+}}}{c_1^2|t|^{2d_1}(2+c_{**}^2)^{2(d_2)^{-}}}\right)
&
\mbox{ \ \ if \ } a=n,
\\
\left(\frac{c_1^2|t|^{2d_1}}{t-s}\right)^{\frac{n-a}{2}}
&
\mbox{ \ \ if \ } a>n,
		\end{cases}
			ds
			\\
\lesssim &
\begin{cases} c_2^{2-a}|t|^{b+d_2(2-a)}
&
\mbox{ \ \ if \ } a<n,
\\
c_{2}^{2-n} |t|^{b+d_2(2-n) }
\ln \left(\frac{c_2^2|t|^{2d_2} (2+c_{**}^2)^{2(d_2)^{+}}}{c_1^2|t|^{2d_1}(2+c_{**}^2)^{2(d_2)^{-}}}\right)
&
\mbox{ \ \ if \ } a=n,
\\
|t|^b
c_{1}^{n-a} |t|^{d_{1}(n-a)}
c_{2}^{2-n} |t|^{d_{2}(2-n)}
&
\mbox{ \ \ if \ } a>n.
\end{cases}
	\end{aligned}
	\end{equation}
	Fourth, when $s\in(-\infty,2t-c_2^2|t|^{2d_2}]$, we have  $\frac{-s}{2} \le t-s \le -s$.
For $a<n$,
\begin{equation*}
	\begin{aligned}
		J_4
		=&\int_{-\infty}^{2t-c_2^2|t|^{2d_2}}\frac{|s|^{b}}{(t-s)^{a/2}}F\left(\frac{c_1^2|s|^{2d_1}}{4(t-s)},\frac{c_2^2|s|^{2d_2}}{4(t-s)}\right)ds\\
\lesssim
&\int_{-\infty}^{2t-c_2^2|t|^{2d_2}}
|s|^{b-\frac a2}
F\left(\frac{c_1^2|s|^{2d_1}}{4|s|},\frac{c_2^2|s|^{2d_2}}{2|s|}\right)ds
\\
		\lesssim& \int_{-\infty}^{2t-c_2^2|t|^{2d_2}}|s|^{b-\frac{a}{2}}
		c_2^{n-a}
		|s|^{(2d_2-1)\frac{n-a}{2}}
		ds
\\
		\lesssim &
		 c_2^{n-a}|t|^{b+d_2(n-a)-\frac{n}{2}+1}
		 \lesssim c_2^{2-a}|t|^{b+d_2(2-a)},
	\end{aligned}
\end{equation*}
	where \eqref{b-require} is needed to guarantee the integrability and the last step is using $d_2\leq 1/2$, $c_2\leq c_{**}$ and $n>2$. For $a\ge n$, similarly we have
\begin{equation*}
\begin{aligned}
J_4 \lesssim &
\begin{cases}
c_{1}^{n-a} |t|^{b+d_{1}(n-a) +1-\frac n2}
&
\mbox{ \ \ if \ } a>n,
\\
|t|^{b+1-\frac a2}\ln \left(
\frac{2c_2|t|^{d_2}}{c_1|t|^{d_1}}
\right)
&
\mbox{ \ \ if \ } a=n,
\end{cases}
\\
\lesssim &
c_1^{2-a}|t|^{b+d_1(2-a)} .
\end{aligned}
\end{equation*}
	
	 Collecting the estimates of $J_1$ to $J_4$, we get the conclusion.
\end{proof}

\begin{remark}\label{rmk:int-G-est}
	After close examination of the proof, only \eqref{J2a} needs the comparison between  $a$ and 2.
	In fact, if $a<2$, one can let $c_1=0$ to get
	\[\TT^{out}\left[\frac{|t|^{b}}{|x|^a}\mathbf{1}_{\{ |x|\leq c_2|t|^{d_2}\}}\right](0,t)\leq C c_2^{2-a}|t|^{b+d_2(2-a)}.\]

\end{remark}

\begin{lemma}\label{lem:annulus}
	Suppose $n >2$, $a\ge 0$, $d_1\leq d_2\leq\frac12$ and $b$ satisfies \eqref{b-require}, $0\leq c_1,c_2\leq c_{**}$.
	Denote
	\[
	u(x,t)=\mathcal{T}^{out}\left[\frac{|t|^{b}}{|x|^a}\mathbf{1}_{\{ c_1|t|^{d_1}\leq |x|\le c_2|t|^{d_2}\}}\right] (x,t).
	\]
	 Then there exists $C$ depending on $n,a,b,d_1,d_2,c_{**}$ such that for $t<-1$
	\begin{align}\label{u_infty_bd}
	u(x,t)\leq C \begin{cases}c_1^{2-a}|t|^{b+d_1(2-a)}&\text{if  }a\in(2,\infty),\\c_2^{2-a}|t|^{b+d_2(2-a)}&\text{if }a\in[0,2).\end{cases}
	\end{align}
	Moreover, when $|x|>2c_2|2t|^{d_2}$, $a<n$,
	\begin{align}\label{u-d2}
		u(x,t) \le C c_2^{n-a}\begin{cases}	|t|^{b+d_2(n-a)}|x|^{2-n}&\text{if }2c_2|2t|^{d_2}\leq |x|\le|t|^{\frac{1}{2}},\\|x|^{2b+2d_2(n-a)+2-n}&\text{if }|x|\geq |t|^{\frac12}.
		\end{cases}
	\end{align}
When $|x|\ge 2c_1 |2t|^{d_1}$, $a>n$,
\begin{equation}\label{u-d3}
	u(x,t)
	\le C
	c_1^{n-a}\begin{cases}
	|t|^{b+d_{1}(n-a)}
		|x|^{2-n}
		&
		\mbox{ \ \ if \ }
		2c_1 |2t|^{d_1}
		\le |x|\le |t|^{\frac 12},
		\\
		|x|^{2b+2d_1(n-a)} |x|^{2-n}
		&
		\mbox{ \ \ if \ } |x|\ge |t|^{\frac 12} .
	\end{cases}
\end{equation}
\end{lemma}
\begin{proof}
	$\bullet$
	Since
	\[\frac{|t|^{b}}{|x|^a} \mathbf{1}_{\{ c_1|t|^{d_1}\leq |x|\le c_2|t|^{d_2}\}}\leq |t|^b\min\left\{\frac{1}{c_1^a|t|^{d_1a}},\frac{1}{|x|^a}\right\}\mathbf{1}_{\{|x|\leq c_2|t|^{d_2}\}}=f(x,t) ,
	\]
	then
	\begin{align*}
		u(x,t)\lesssim \int_{-\infty}^t\int_{\Rn}G(x-y,t-s)f(y,s)dyds .
	\end{align*}
	Since $G$ and $f$ are both decreasing functions for each time slice, using Hardy-Littlewood rearrangement inequality, then
	\[u(x,t)\leq u(0,t)=J_1+J_2,\]
	where
	\begin{align}
		J_1=&\int_{-\infty}^t\int_{\Rn}G(y,t-s)|s|^b\min\left\{\frac{1}{c_1^a|s|^{d_1a}},\frac{1}{|y|^a}\right\}\mathbf{1}_{\{|y|\leq c_1|s|^{d_1}\}}dyds,\\
		J_2=&\int_{-\infty}^t\int_{\Rn}G(y,t-s)|s|^b\min\left\{\frac{1}{c_1^a|s|^{d_1a}},\frac{1}{|y|^a}\right\}\mathbf{1}_{\{c_1|s|^{d_1}\leq|y|\leq c_2|s|^{d_2}\}}dyds.
	\end{align}
	Applying Lemma \ref{lem:int-G-est} and Remark \ref{rmk:int-G-est}, we obtain
	\begin{align*}
		J_1\lesssim c_1^{2-a}|t|^{b+d_1(2-a)},\quad
		J_2\lesssim \begin{cases}c_1^{2-a}|t|^{b+d_1(2-a)}&\text{if  }a\in(2,\infty),\\
c_2^{2-a}|t|^{b+d_2(2-a)}&\text{if }a\in(0,2).\end{cases}
	\end{align*}
	Therefore \eqref{u_infty_bd} is established.

	$\bullet $
	Next we will establish \eqref{u-d2} when $d_2\leq 0$. In this case, for $x>2c_2|t|^{d_2}$, one has
	\begin{align}\label{xy-d2<0}
	    \frac12|x|\leq |x-y|\leq 2|x| \text{ for } y \text{ with } |y|\leq |s|^{d_2} \text{ and }s\leq t.
	\end{align}
	Then
	\begin{align*}
		u(x,t)\lesssim& \int_{-\infty}^t\int_{\Rn}\frac{1}{(t-s)^{n/2}}e^{-\frac{|x|^2}{16(t-s)}}\frac{|s|^{b}}{|y|^a}\mathbf{1}_{\{c_1|s|^{d_1}\leq |y|\leq c_2|s|^{d_2}\}}dyds\\
		\approx&\int^t_{-\infty}\int_{c_1|s|^{d_1}}^{c_2|s|^{d_2}}\frac{1}{(t-s)^{n/2}}e^{-\frac{|x|^2}{16(t-s)}}|s|^{b}r^{n-1-a}drds\\
		\lesssim& c_2^{n-a}\int^t_{-\infty}\frac{1}{(t-s)^{n/2}}e^{-\frac{|x|^2}{16(t-s)}}|s|^{b+d_2(n-a)}ds\leq c_2^{n-a}\left(\max\{|t|,|x|^2\}\right)^{b+d_2(n-a)}|x|^{2-n}.
	\end{align*}
	The last step follows from the following two facts
	\begin{align*}
		&\int_{2t}^t\frac{1}{(t-s)^{n/2}}e^{-\frac{|x|^2}{16(t-s)}}|s|^{b+d_2(n-a)}ds
		\approx |t|^{b+d_2(n-a)}\int_{2t}^t\frac{1}{(t-s)^{n/2}}e^{-\frac{|x|^2}{16(t-s)}}ds\\
		\approx& |t|^{b+d_2(n-a)}|x|^{2-n}\int_{\frac{|x|^2}{16|t|}}^\infty e^{-z}z^{\frac{n}{2}-2}dz
		\lesssim\begin{cases}
			|t|^{b+d_2(n-a)}|x|^{2-n}&\text{if }|x|<|t|^{\frac12},\\
			|t|^{b+d_2(n-a)}\int_{\frac{|x|^2}{16|t|}}^\infty e^{-z}z^{\frac{n}{2}-2}dz&\text{if }|x|\geq |t|^{\frac12}.
		\end{cases}
	\end{align*}
	and
	\begin{align*}
		&\int_{-\infty}^{2t}\frac{1}{(t-s)^{n/2}}e^{-\frac{|x|^2}{16(t-s)}}|s|^{b+d_2(n-a)}ds\approx \int_{-\infty}^{2t}e^{-\frac{|x|^2}{16|s|}}|s|^{b+d_2(n-a)-\frac{n}{2}}ds\\
		\approx&|x|^{2b+2d_2(n-a)-n+2}\int_{0}^{\frac{|x|^2}{32|t|}}e^{-z}z^{-b-d_2(n-a)+\frac{n}{2}-2}dz
		\approx\begin{cases}
			|t|^{b+d_2(n-a)-\frac{n}{2}+1}&\text{if }|x|<|t|^{\frac{1}{2}},\\|x|^{2b+2d_2(n-a)-n+2}&\text{if }|x|\geq |t|^{\frac12},
		\end{cases}
	\end{align*}
	where $-b-d_2(n-a)+\frac{n}{2}-2>-1$ is needed to guarantee the integrability. Thus \eqref{u-d2} is established.
	
	$\bullet $ Next we will establish \eqref{u-d2} when $d_2>0$. We do not have \eqref{xy-d2<0} anymore. In this case, $|x|\geq 2c_2|2t|^{d_2}$ is equivalent to $-(\frac{|x|}{2c_2})^{1/d_2}\leq 2t$. We write
	\[(-\infty,t)=(-(\frac{|x|}{2c_2})^{1/d_2},t)\cup (-(\frac{2|x|}{c_2})^{1/d_2},-(\frac{|x|}{2c_2})^{1/d_2})\cup (-\infty,-(\frac{2|x|}{c_2})^{1/d_2})\]
	and thus
	\begin{align}
	    u\lesssim \int_{-\infty}^t\int_{\Rn}\frac{1}{(t-s)^{n/2}}e^{-\frac{|x-y|^2}{4(t-s)}}\frac{|s|^{b}}{|y|^a}\mathbf{1}_{\{|y|\leq c_2|s|^{d_2}\}}dyds=u_1+u_2+u_3
	\end{align}
	where $u_1,u_2,u_3$ are the integrations according to the three intervals respectively.
	We shall verify that $u_1,u_2,u_3$ all satisfy \eqref{u-d2}. For $u_1$, one has $c_2|s|^{d_2}\leq |x|/2$ for such $s$. Then
	\begin{align*}
	    u_1=&\int^t_{-(|x|/2c_2)^{1/d_2}}\int_{\Rn}\frac{1}{(t-s)^{n/2}}e^{-\frac{|x-y|^2}{4(t-s)}}\frac{|s|^{b}}{|y|^a}\mathbf{1}_{\{|y|\leq c_2|s|^{d_2}\}}dyds\\
	    \lesssim &\int^t_{-(|x|/2c_2)^{1/d_2}}\int_{\Rn}\frac{1}{(t-s)^{n/2}}e^{-\frac{|x|^2}{16(t-s)}}\frac{|s|^{b}}{|y|^a}\mathbf{1}_{\{|y|\leq c_2|s|^{d_2}\}}dyds\\
	    \lesssim &\left(\int_{2t}^t+\int^{2t}_{-(|x|/2c_2)^{1/d_2}}\right)\int_{\Rn}\frac{1}{(t-s)^{n/2}}e^{-\frac{|x|^2}{16(t-s)}}\frac{|s|^{b}}{|y|^a}\mathbf{1}_{\{|y|\leq c_2|s|^{d_2}\}}dyds\\
	    \lesssim&c_2^{n-a} \left(\max\{|t|,|x|^2\}\right)^{b+d_2(n-a)}|x|^{2-n}.
	\end{align*}
	The last step follows from some easy integration which has been done many times in this section. For $u_2$,
	\begin{align*}
	    u_2\lesssim& \int_{-(2|x|/c_2)^{1/d_2}}^{-(|x|/2c_2)^{1/d_2}}\int_{\Rn} \frac{1}{(t-s)^{n/2}}e^{-\frac{|x-y|^2}{4(t-s)}}\frac{|s|^{b}}{|y|^a}\mathbf{1}_{\{|y|\leq c_2|s|^{d_2}\}}dyds\\
	    \lesssim& (|x|/c_2)^{b/d_2}\int_{-(2|x|/c_2)^{1/d_2}}^{-(|x|/2c_2)^{1/d_2}}\int_{\Rn} \frac{1}{(t-s)^{n/2}}e^{-\frac{|x-y|^2}{4(t-s)}}\frac{1}{|y|^a}\mathbf{1}_{\{|y|\leq 2|x|\}}dyds\\
	    \lesssim& c_2^{-b/d_2}|x|^{b/d_2+n-a}\int_{-(2|x|/c_2)^{1/d_2}}^{-(|x|/2c_2)^{1/d_2}} \frac{1}{(t-s)^{n/2}}ds\lesssim c_2^{(-b+\frac{n}{2}-1)/d_2}(|x|^{1/d_2})^{b+d_2(n-a)+1-\frac{n}{2}}\\
	   \lesssim& c_2^{n-a}\begin{cases}	|t|^{b+d_2(n-a)}|x|^{2-n}&\text{if }2c_2|2t|^{d_2}\leq |x|<|t|^{\frac{1}{2}},\\|x|^{2b+2d_2(n-a)+2-n}&\text{if }|x|\geq |t|^{\frac12}.
		\end{cases}
	\end{align*}
	The last step follows from $0<d_2\leq \frac12$ and $b+d_2(n-a)+1-\frac{n}{2}<0$. Similarly, for $u_3$,
	\begin{align*}
	    u_3\lesssim& \int^{-(2|x|/c_2)^{1/d_2}}_{-\infty}\int_{\Rn} \frac{1}{(t-s)^{n/2}}e^{-\frac{|x-y|^2}{4(t-s)}}\frac{|s|^{b}}{|y|^a}\mathbf{1}_{\{|y|\leq c_2|s|^{d_2}\}}dyds\\
	    \lesssim&(|x|/c_2)^{(b+1-\frac{n}{2})/d_2}|x|^{n-a}+\int^{-(2|x|)^{1/d_2}}_{-\infty}\frac{1}{|s|^{\frac{n}{2}}}e^{-\frac{|y|^2}{16(t-s)}}\frac{|s|^a}{|y|^b}\mathbf{1}_{\{2|x|\leq |y|\leq c_2|s|^{d_2}\}}dyds\\
	    \lesssim& c_2^{(-b+\frac{n}{2}-1)/d_2}(|x|^{1/d_2})^{b+d_2(n-a)+1-\frac{n}{2}}\\
	    \lesssim& c_2^{n-a}\begin{cases}	|t|^{b+d_2(n-a)}|x|^{2-n}&\text{if }2c_2|2t|^{d_2}\leq |x|<|t|^{\frac{1}{2}},\\|x|^{2b+2d_2(n-a)+2-n}&\text{if }|x|\geq |t|^{\frac12}.
	    \end{cases}
	\end{align*}
	Collecting the results of $u_1,u_2,u_3$, we can get \eqref{u-d2}.
	
$\bullet$ By the similar calculation like \eqref{u-d2}, we will get \eqref{u-d3}.
\end{proof}

\begin{lemma}\label{lem:sing} Suppose $2<a<n$, $0\leq d_2\leq \frac12$, $\frac{n}{2} -b>1$ and $0<c_2\leq c_{**}$.  Then there exists $C$ depending on $n,a,b,d_2,c_{**}$ such that for $t<-1$,
	\begin{align}\label{u-d2=0}
		\mathcal{T}^{out}\left[\frac{|t|^{b}}{|x|^a}\mathbf{1}_{\{|x|\leq c_2|t|^{d_2}\}}\right]\le C |t|^{b}|x|^{2-a} \quad \text{for }|x|<c_2|t|^{d_2}.
	\end{align}
\end{lemma}
\begin{proof}
 We divide $u$ into three parts
	\begin{align*}
		u(x,t)= \int_{-\infty}^t\int_{\Rn}G(x-y,t-s)\frac{|s|^{b}}{|y|^a}\mathbf{1}_{\{|y|\leq c_2|s|^{d_2}\}}dyds=u_1+u_2+u_3,
	\end{align*}
	where $u_1$ is the term with $\mathbf{1}_{\{|y|\leq \frac{1}{2}|x|\}}$ inside the integrand, $u_2$ is the one with $\mathbf{1}_{\{\frac{1}{2}|x|\leq |y|\leq 2|x|\}}$ and $u_3$ is the one with $\{ 2|x|\leq |y|\leq c_2|s|^{d_2}\}$. Since most of the calculation are similar to the proof of the previous lemma. We omit some details here. For $u_1$, we proceed as
	\begin{align*}
		u_1\lesssim& \int_{-\infty}^t\int_{\Rn}\frac{1}{(t-s)^{n/2}}e^{-\frac{|x|^2}{16(t-s)}}\frac{|s|^{b}}{|y|^a}\mathbf{1}_{\{|y|\leq \frac12|x|\}}dyds\\
		\approx\, &|x|^{n-a}\int_{-\infty}^t\frac{1}{(t-s)^{n/2}}e^{-\frac{|x|^2}{16(t-s)}}|s|^{b}ds\\
		\lesssim&\, |t|^{b}|x|^{2-a}\int_{\frac{|x|^2}{16(-t)}}^\infty e^{-z}z^{\frac{n}{2}-2}dz+|x|^{2b+2-a}\int_0^{\frac{|x|^2}{32(-t)}}e^{-z}z^{-b+\frac{n}{2}-2}dz\\
		\lesssim& |t|^{b}|x|^{2-a}+|x|^{n-a}|t|^{1-\frac{n}{2}+b},
	\end{align*}
	where we have used the fact that $\frac{n}{2} -b>1$. For $u_2$, we have
	\begin{align*}
		u_2\lesssim&\int_{-\infty}^t\int_{\Rn} \frac{1}{(t-s)^{n/2}}e^{-\frac{|x-y|^2}{4(t-s)}}\frac{|s|^{b}}{|y|^a}\mathbf{1}_{\{\frac{1}{2}|x|\leq |y|\leq 2|x|\}}dyds\\
		\lesssim& |x|^{-a}\int_{-\infty}^t\int_{\Rn}\frac{1}{(t-s)^{n/2}}e^{-\frac{|x-y|^2}{4(t-s)}}|s|^{b}\mathbf{1}_{\{|x-y|\leq 3|x|\}}dydx\\
		\approx& |x|^{-a}\int_{-\infty}^t\int_0^{3|x|}\frac{1}{(t-s)^{n/2}}e^{-\frac{r^2}{4(t-s)}}|s|^{b}r^{n-1}dyds
		\approx |x|^{-a}\int_{-\infty}^t\int_0^{\frac{9|x|^2}{4(t-s)}}|s|^{b}e^{-z}z^{\frac{n}{2}-1} d z  d s\\
		\approx& |x|^{-a}\int_{-\infty}^t|s|^{b}\min\left\{1,\left(\frac{|x|^2}{t-s}\right)^{\frac{n}{2}}\right\}
		d s
		\approx |t|^{b}|x|^{2-a}+|t|^{1+b-\frac{n}{2}}|x|^{n-a}.
	\end{align*}
	For $u_3$, we have
	\begin{align*}
		u_3\lesssim&\int_{-\infty}^t\int_{\Rn} \frac{1}{(t-s)^{n/2}}e^{-\frac{|x-y|^2}{4(t-s)}}\frac{|s|^{b}}{|y|^a}\mathbf{1}_{\{2|x|\leq |y|\leq c_2|s|^{d_2}\}}dyds\\
		\lesssim& \int_{-\infty}^t\int_{\Rn} \frac{1}{(t-s)^{n/2}}e^{-\frac{|y|^2}{16(t-s)}}\frac{|s|^{b}}{|y|^a}\mathbf{1}_{\{2|x|\leq |y|\leq c_2 |s|^{d_2}\}}dyds\\
		\approx& \int_{-\infty}^t\int_{2|x|}^{c_2|s|^{d_2}}\frac{1}{(t-s)^{n/2}}e^{-\frac{r^2}{16(t-s)}}|s|^{b}r^{n-1-a}drds\\
		\approx&\int_{-\infty}^t\int^{\frac{c_2|s|^{d_2}}{16(t-s)}}_{\frac{|x|^2}{4(t-s)}}(t-s)^{-\frac{a}{2}}|s|^{b}e^{-z}z^{\frac{n-a}{2}-1}dzds\lesssim
		|t|^{b} |x|^{2-a} .
	\end{align*}
	Combining the estimate of $u_1,u_2,u_3$ and using the fact that $|x|^{n-a}|t|^{1-\frac{n}{2}+b}\leq |t|^{b}|x|^{2-a}$ because $|x|\leq c_2|t|^{d_2}$, we get \eqref{u-d2=0}.
\end{proof}

\begin{corollary}\label{Coro-B4}
Suppose that $n>2$, $d_1\leq d_2\leq \frac12$, $b$ satisfies \eqref{b-require}, $0\leq c_1,c_2\leq c_{**}$. Denote
\begin{equation*}
u(x,t) = \mathcal{T}^{out}\left[
|t|^{b} |x|^{-a}\mathbf{1}_{\{ c_1|t|^{d_1}\leq |x|\le c_2|t|^{d_2}\}}\right] .
\end{equation*}
Then there exists $C$ depending on $n,a,b,d_1,d_2,c_{**}$ such that for $t<-1$,\\
if $0\le a <2$,
\begin{equation}
\label{Tout-a<2}
u(x,t)
\le C
\begin{cases}
	c_2^{2-a}|t|^{b+d_2(2-a)}
	&\text{if }|x|\leq c_2|t|^{d_2},
	\\
	c_2^{n-a}|t|^{b+d_2(n-a)}|x|^{2-n}&\text{if }c_2|t|^{d_2}\leq |x|\leq |t|^{\frac12},
	\\
	c_2^{n-a} |x|^{2b+2d_2(n-a)+2-n}&\text{if }|x|\geq |t|^{\frac12}.
\end{cases}
\end{equation}
If $2<a<n$,
\begin{equation}\label{Tout_a>2}
    u(x,t)
    \le C
    \begin{cases}
    c_1^{2-a}|t|^{b+d_1(2-a)}&\text{if }|x|\leq c_1|t|^{d_1},
    \\
    |t|^{b}|x|^{2-a}&\text{if }c_1|t|^{d_1}\leq |x|\leq c_2|t|^{d_2},
    \\
    c_2^{n-a}|t|^{b+d_2(n-a)}|x|^{2-n}&\text{if }c_2|t|^{d_2}\leq |x|\leq |t|^{\frac12},
    \\
   c_2^{n-a} |x|^{2b+2d_2(n-a)+2-n}&\text{if }|x|\geq |t|^{\frac12}.
    \end{cases}
\end{equation}
If $a>n$,
\begin{equation}\label{Tout-a>n}
	u(x,t)
	\le C
	\begin{cases}
		c_1^{2-a}|t|^{b+d_1(2-a)}&\text{if }|x|\leq c_1|t|^{d_1},
		\\
		c_1^{n-a}|t|^{b+d_1(n-a)}|x|^{2-n}
		&\text{if }c_1|t|^{d_1}\leq |x|\leq |t|^{\frac12},
		\\
		c_1^{n-a} |x|^{2b+2d_1(n-a)+2-n}&\text{if }|x|\geq |t|^{\frac12}.
	\end{cases}
\end{equation}
\end{corollary}
\begin{proof}
\eqref{Tout-a<2} and \eqref{Tout-a>n} follow Lemma \ref{lem:annulus} directly.
\eqref{Tout_a>2} follows from \eqref{u-d2} and \eqref{u-d2=0}.
\end{proof}

\begin{lemma}\label{lem:outer-du}
	Suppose that $\frac a2 -b >1$. Then
\begin{equation}
	\begin{aligned}
	&
	\TT^{out}
	\left[
	|t|^{b}
	|x|^{-a} \1_{\{|x|\ge  |t|^{\frac 12}\}}
	\right]
\\
 \lesssim \ &
|t|^{1+b-\frac a2} \1_{\{|x|\le |t|^{\frac 12}\} }
	+
	\1_{\{|x|\ge |t|^{\frac 12}\}}
		|x|^{-a}
		\begin{cases}
			|t|^{1+b},
			&
			\mbox{ \ \ if \ } b<-1,
			\\
			1+
			\ln\left(
			\frac{|x|^2}{|t|}
			\right)
			&
			\mbox{ \ \ if \ } b= - 1,
			\\
			|x|^{2+2b}
			&
			\mbox{ \ \ if \ } b> - 1 .
		\end{cases}
	\end{aligned}
\end{equation}
\end{lemma}

\begin{proof}
Denote $u(x,t) = \TT^{out}
\left[
|t|^{b}
|x|^{-a} \1_{[|x|\ge  |t|^{\frac 12}]}
\right] $.

$\bullet$	For $|x|\le \frac 12 |t|^{\frac 12}$, we have $|x-y|\geq \frac12|y|$ for $|y|\geq |s|^{\frac 12}\geq |t|^{\frac 12}$. Then
	\begin{equation}
		\begin{aligned}
			u(x,t)
			\lesssim \ &
			\int_{-\infty}^t \int_{\RR^n}
			\frac{1}{ (t-s)^{n/2}}
			e^{-\frac{|y|^2}{16(t-s)}}
			|s|^{b}
			|y|^{-a} \1_{\{|y|\ge  |s|^{\frac 12}\}} \dd \xi \dd s
			\\
			\lesssim \ &
			\left(
			\int_{-\infty}^{2t}
			+
			\int_{2t}^t
			\right)
			|s|^{b}
			(t-s)^{-a/2} e^{-\frac{(-s)}{32(t-s)}} \dd s
\lesssim \
|t|^{1+b-\frac a2} ,
		\end{aligned}
	\end{equation}
where  $\frac a2 -b >1$ is used to guarantee the integrability in the last step.

$\bullet$	Consider $|x|\ge 4|t|^{\frac 12}$. We make the following decomposition.
	\begin{align*}
			u(x,t)
			=  \ &	\left(\int_{-\frac14|x|^2}^t+\int_{-4|x|^2}^{-\frac14|x|^2}+\int^{-4|x|^2}_{-\infty}\right)\int_{\Rn}\frac{1}{(t-s)^{n/2}}e^{-\frac{|x-y|^2}{4(t-s)}}|s|^{b}|y|^{-a}\mathbf{1}_{\{|y|\geq |s|^{\frac12}\}}\dd y\dd s
			\\
			:= \ &
			P_1 + P_2 +P_3 .
	\end{align*}
For $P_1$, we divide it further to be
\begin{align*}
    P_1=\int_{-\frac14|x|^2}^t\int_{\Rn}\frac{1}{(t-s)^{n/2}}e^{-\frac{|x-y|^2}{4(t-s)}}|s|^{b}|y|^{-a}\mathbf{1}_{\{|y|\geq |s|^{\frac12}\}}dyds=P_{11}+P_{12}+P_{13}
\end{align*}
where $P_{11}$ is the term with $\mathbf{1}_{\{|s|^{\frac{1}{2}}\leq |y|\leq \frac{1}{2}|x|\}}$ in the integrand, $P_{12}$ is the one with $\mathbf{1}_{\{\frac12|x|\leq |y|\leq 2|x|\}}$ and $P_{13}$ is the one with $\mathbf{1}_{\{2|x|\leq |y|\}}$. For $P_{11}$, when $a<n$,
	\begin{equation}
		\begin{aligned}
			P_{11}
			\lesssim \ &
			\int_{
				-\frac14|x|^2
			}^t
			\int_{\RR^n}
			\frac{1}{(t-s)^{n/2}}
			e^{-\frac{|x|^2}{16(t-s)}}
			|s|^{b}
			|y|^{-a}
			\1_{\{ |s|^{\frac 12} \le |y|\le
				\frac{|x|}{2}\}}
			\dd y \dd s
\\
			\lesssim \ &
			\int_{
				-\frac14|x|^2
			}^t
			\frac{1}{(t-s)^{n/2}}
			e^{-\frac{|x|^2}{16(t-s)}}
			|s|^{b}
			|x|^{n-a} \dd s
\\
= \ &
|x|^{n-a}
\int_{
	-\frac14|x|^2
}^{t}
\frac{1}{(t-s)^{n/2}}
e^{-\frac{|x|^2}{16(t-s)}}
|s|^{b}
 \dd s
\lesssim \
 |x|^{2-a+2b}
 .
		\end{aligned}
	\end{equation}
When $a\ge n$, by similar calculation,
$
P_{11} \lesssim
|x|^{2-a+2b}
$ still holds.

	For $P_{12}$,
	\begin{equation}
		\begin{aligned}
			P_{12}
			\lesssim
			\ &
			|x|^{-a}
			\int_{
				-\frac14|x|^2
			}^t
			\int_{\RR^n}
			\frac{1}{(t-s)^{n/2}}
			e^{-\frac{|x-y|^2}{4(t-s)}}
			|s|^{b}
			\1_{\{
				\frac{|x|}{2} \le |y| \le 2|x|\}}
			\dd y \dd s
			\\
			\lesssim
\ &
|x|^{-a}
\left(
\int_{2t
}^t
+
\int_{
	-\frac14|x|^2
}^{2t}
\right)
\int_{0}^{3|x|}
\frac{1}{(t-s)^{n/2}}
e^{-\frac{r^2}{4(t-s)}}
|s|^{b}
r^{n-1}
\dd r \dd s
\\
\lesssim \ &
|x|^{-a}
|t|^{1+b}
+ |x|^{-a}
\begin{cases}
|t|^{1+b}
	&
	\mbox{ \ \ if \ } b< - 1,
	\\
	1+
	\ln\left(
	\frac{|x|^2}{-t}
	\right)
	&
	\mbox{ \ \ if \ } b=-1,
	\\
	(|x|^2)^{1+b}
	&
	\mbox{ \ \ if \ } b>- 1 .
\end{cases}
		\end{aligned}
	\end{equation}

	For $P_{13}$,
	\begin{equation}
		\begin{aligned}
			P_{13}
			\lesssim \ &
			\int_{
				-\frac14|x|^2
			}^t
			\int_{\RR^n}
			\frac{1}{(t-s)^{n/2}}
			e^{-\frac{|y|^2}{ 16(t-s)}}
			|s|^{b}
			|y|^{-a}
			\1_{\{
				2|x| \le  |y|\}}
			\dd y \dd s
			\\
			\approx \ &
			\int_{
				-\frac14|x|^2
			}^t
			\int_{2|x|}^\infty
			\frac{(-s)^{b}}{(t-s)^{n/2}}
			e^{-\frac{ r^2}{ 16(t-s)}}
			r^{n-1-a}
			\dd r \dd s
			\\
=
 \ &
 \left(
\int_{2t
}^{t}
+
\int_{
	-\frac14|x|^2
}^{2t}
\right)
\int_{2|x|}^\infty
\frac{(-s)^{b}}{(t-s)^{n/2}}
e^{-\frac{ r^2}{ 16(t-s)}}
r^{n-1-a}
\dd r \dd s
\\
\lesssim \ &
(-t)^{b}
|x|^{2-a} e^{-\frac{|x|^2}{16(-t)}}
+
|x|^{2+2b- a} .
		\end{aligned}
	\end{equation}

	\medskip
	
	For $P_2$, since $-\frac14|x|^2\leq 4t$ in this case,
	\begin{equation}
		\begin{aligned}
			P_2
			\lesssim  \ &
			\int_{-4|x|^2}^{
				-\frac14|x|^2
			}
			\int_{\RR^n}
			\frac{1}{ |s|^{n/2}}
			e^{-\frac{|x-y|^2}{4(-s)}}
			|s|^{b}
			|y|^{-a} \1_{\{|y|\ge \frac{ |x|}{2} \}} \dd y \dd s
			\\
			\approx
			\ &
			|x|^{-n+2b}
			\int_{-4|x|^2}^{
				-\frac14|x|^2
			}
			\int_{\RR^n}
			e^{-\frac{|x-y|^2}{4(-s)}}
			|y|^{-a}
			\left(
			\1_{\{\frac{ |x|}{2}  \le |y|\le 2|x|\}}
			+
			\1_{\{2|x|\le|y| \}}
			\right)
			\dd y \dd s
			\\
			\lesssim \ &
			|x|^{2+2b-a} .
		\end{aligned}
	\end{equation}
	
	\medskip
	
	For $P_3$, in this case, $|x|\le \frac 12 |s|^{\frac 12}$,
	\begin{equation}
		\begin{aligned}
			P_3
			\lesssim
			\ &
			\int_{-\infty}^{-4|x|^2}
			\int_{\RR^n}
			\frac{1}{(t-s )^{n/2}}
			e^{-\frac{|y|^2}{ 16(t-s)}}
			|s|^{b}
			|y|^{-a} \1_{\{|y|\ge  |s|^{\frac 12}\}} \dd y \dd s
			\\
	\lesssim
\ &
\int_{-\infty}^{-4|x|^2}
\int_{\RR^n}
(-s)^{b-\frac n2}
e^{-\frac{|y|^2}{ 16(-s)}}
|y|^{-a} \1_{\{|y|\ge  |s|^{\frac 12}\}} \dd y \dd s
\lesssim \
|x|^{2+2b-a} ,
		\end{aligned}
	\end{equation}
where $\frac a2 -b>1$ is required to guarantee the integrability. Combining the above estimates of $P_1$, $P_2$ and $P_3$, we get, when $|x|\geq 4|t|^{\frac 12}$
	\begin{equation}
		u(x,t) \lesssim
			|x|^{-a}
			\begin{cases}
				|t|^{1+b}
				&
				\mbox{ \ \ if \ } b<-1,
				\\
				1+
				\ln\left(
				\frac{|x|^2}{-t}
				\right)
				&
				\mbox{ \ \ if \ } b=-1,
				\\
				|x|^{2+2b}
				&
				\mbox{ \ \ if \ } b>-1 .
			\end{cases}
	\end{equation}
	
$\bullet$	Consider the case $ \frac 12 |t|^{\frac 12} \le |x|\le 4|t|^{\frac 12}$,
	\begin{equation}
		\begin{aligned}
			u(x,t) = \ &
			\int_{-\infty}^t \int_{\RR^n}
			\frac{1}{(4\pi(t-s))^{n/2}}
			e^{-\frac{|x-y|^2}{4(t-s)}}
			|s|^{b}
			|y|^{-a} \1_{[|y|\ge  |s|^{\frac 12}]} \dd y \dd s
			\\
			\lesssim \ &
			\left(
			\int_{64t}^t
			+
			\int_{-\infty}^{64t}
			\right)
			\int_{\RR^n}
			\frac{1}{(t-s)^{n/2}}
			e^{-\frac{|x-y|^2}{4(t-s)}}
			|s|^{b}
			|y|^{-a} \1_{\{|y|\ge  |s|^{\frac 12}\}} \dd y \dd s
	\\
		\lesssim \ &
	|t|^{1+b -\frac a2} .
		\end{aligned}
	\end{equation}
\end{proof}

In order to get the gradient estimate of $\bar{\varphi}$, we need the following lemma.
\begin{lemma}\label{lem:deri-est}
	For $d_{1}\le d_{2} \le \frac 12$,
	$\frac n2 -b-d_2 n>0$, $c_1, c_2 \approx 1$, we have
	\begin{equation}
\TT^{d}[|t|^{b} \1_{[c_1 |t|^{d_1} \le |x| \le c_2 |t|^{d_2}]}]
		\lesssim
		\begin{cases}
		|t|^{b+d_2}
			&
			\mbox{ \ \ if \ } |x| \le |t|^{d_2},
			\\
			|t|^{b+d_2 n} |x|^{1-n}
			&
			\mbox{ \ \ if \ }
			|t|^{d_2} \le
			|x| \le  |t|^{\frac 12},
			\\
			(|x|^2)^{b+d_2n} |x|^{1-n}
			&
			\mbox{ \ \ if \ }
			|x| \ge |t|^{\frac 12}.
		\end{cases}
	\end{equation}	
	
	For $d_{1}\le d_{2} \le \frac 12$, $a>n$, $\frac n2 -b -d_1(n-a)>0$, $c_1, c_2 \approx 1$, we have
	\begin{equation}
	\TT^{d}[\frac{|t|^{b}}{|x|^a} \1_{[c_1 |t|^{d_1} \le |x| \le c_2 |t|^{d_2}]}] \lesssim
		\begin{cases}
			|t|^{b+d_1(1-a)}
			&
			\mbox{ \ \ if \ }
			|x| \le  |t|^{d_1} ,
			\\		
			|t|^{b+d_1(n-a)} |x|^{1-n}
			&
			\mbox{ \ \ if \ }
			|t|^{d_1} \le
			|x| \le  |t|^{\frac 12},
			\\
			(|x|^2)^{b+d_1(n-a)} |x|^{1-n}
			&
			\mbox{ \ \ if \ }
		  |x| \ge |t|^{\frac 12} .
		\end{cases}
	\end{equation}
\end{lemma}
We omit the proof since it relies splitting integral domain like above.




\subsection{Proofs of three lemmas in the outer problem}\label{out-linear-pf}
\begin{proof}[Proof of Lemma \ref{lem:w1j-w1j*}]
For $j=2,\dots,k$, by Corollary \ref{Coro-B4},
	\begin{equation*}
		\TT^{out}[\mu_{0j}^{-2}(t)|t|^{\gamma_j}
		\1_{\{ |x|\le \mu_{0j} \}}]
		\lesssim
		\begin{cases}
			|t|^{\gamma_{j}}
			&
			\mbox{ \ \ if \ } |x|\le \mu_{0j},
			\\
			|t|^{\gamma_{j}}
			\mu_{0j}^{n-2} |x|^{2-n}
			&
			\mbox{ \ \ if \ } \mu_{0j} \le |x|\le |t|^{\frac 12},
			\\
			|x|^{2\gamma_{j} +(4-2n)\alpha_j +2-n}
			&
			\mbox{ \ \ if \ }  |x|\ge |t|^{\frac 12},
		\end{cases}
	\end{equation*}
	\begin{equation*}
		\TT^{out} [\mu_{0j}^{\alpha}(t)|t|^{\gamma_j}
		|x|^{-2-\alpha}
		\1_{ \{\mu_{0j} \le |x|\le \bar{\mu}_{0j} \} }]
		\lesssim
		\begin{cases}
			|t|^{\gamma_{j}}
			&
			\mbox{ \ \ if \ } |x|\le \mu_{0j},
			\\
			|t|^{\gamma_{j}} \mu_{0j}^{\alpha} |x|^{-\alpha}
			&
			\mbox{ \ \ if \ } \mu_{0j} \le |x|\le\bar{\mu}_{0j} ,
			\\
			|t|^{\gamma_{j}}
			\mu_{0j}^{\alpha}
			\bar{\mu}_{0j}^{n-2-\alpha} |x|^{2-n}
			&
			\mbox{ \ \ if \ }  \bar{\mu}_{0j} \le |x| \le |t|^{\frac 12} ,
			\\
			|x|^{2\gamma_j^*+2-n}
			&
			\mbox{ \ \ if \ }   |x| \ge |t|^{\frac 12} .
		\end{cases}
	\end{equation*}
	Thus $\TT^{out}[w_{1j}] \lesssim w_{1j}^{*}$.
	
For $j=1$, the first part of $w_{11}$ can be dealt with by the same method above. For the rest part,
by Corollary \ref{Coro-B4},
\begin{equation*}
	\begin{aligned}
		\TT^{out}[|t|^{\gamma_{1}}
		\bar{\mu}_{01}^{n-2-\alpha} |x|^{-1-n} \1_{\{ \bar{\mu}_{01} \le |x| \le |t|^{\frac 12} \}}]
		\lesssim \ &
		\begin{cases}
			|t|^{\gamma_{1} } \bar{\mu}_{01}^{-1-\alpha}
			&
			\mbox{ \ \ if \ } |x|\le \bar{\mu}_{01}
			,
			\\
			|t|^{\gamma_1} \bar{\mu}_{01}^{n-3-\alpha} |x|^{2-n}
			&
			\mbox{ \ \ if \ }  \bar{\mu}_{01} \le |x|\le |t|^{\frac 12}
			,
			\\
			|x|^{2(\gamma_1 +\delta (n-3-\alpha)) +2-n}
			&
			\mbox{ \ \ if \ } |x| \ge  |t|^{\frac 12} .
		\end{cases}
		\\
		\lesssim \ & w_{11}^{*} .
	\end{aligned}
\end{equation*}
\end{proof}
\begin{proof}[Proof of Lemma \ref{lem:w2-w2*}]
	This just follows from \eqref{Tout_a>2}. $b=-2\sigma-(\frac{n}{2}-2)\alpha_{j+1}+\alpha_j$, $d_2=-\frac{1}{2}(\alpha_j+\alpha_{j-1})$, $a=n-2$.
\end{proof}
\begin{proof}[Proof of Lemma \ref{lem:w3-w3*}]
By Corollary \ref{Coro-B4}, we have
\begin{equation*}
	\TT^{out}[|t|^{-1-\sigma} |x|^{2-n} \1_{\{ \bar{\mu}_{01} \le |x| \le |t|^{\frac 12}\}}]
	\lesssim
	\begin{cases}
		|t|^{-1-\sigma}
		\bar{\mu}_{01}^{4-n}
		&
		\mbox{ \ \ if \ } |x|\le \bar{\mu}_{01},
		\\
		|t|^{-1-\sigma} |x|^{4-n}
		&
		\mbox{ \ \ if \ } \bar{\mu}_{01} \le  |x|\le |t|^{\frac 12},
		\\
		|x|^{2-2\sigma-n}
		&
		\mbox{ \ \ if \ }
		|x| \ge |t|^{\frac 12 }.
	\end{cases}
\end{equation*}
By Lemma \ref{lem:outer-du}, we have
\begin{equation*}
	\TT^{out}[|t|^{-1-\sigma} |x|^{2-n} \1_{\{  |x| \ge |t|^{\frac 12}\}}]
	\lesssim
	\begin{cases}
		|t|^{1-\sigma-\frac n2}
		&
		\mbox{ \ \ if \ }  |x|\le |t|^{\frac 12},
		\\
		|t|^{-\sigma}|x|^{2-n}
		&
		\mbox{ \ \ if \ }
		|x| \ge |t|^{\frac 12 }.
	\end{cases}
\end{equation*}
Then \eqref{eq:w3j*} follows when $\delta\le \frac 12$.
\end{proof}




 \bibliographystyle{plainnat}
 \bibliography{Ancient-ref}

\end{document}